\documentclass{article}
\usepackage{pgfplots}
\pgfplotsset{compat=1.17}
\usepackage{multicol}
\usepackage{array}
\usepackage{booktabs}

\usepackage[margin=1in]{geometry}
\usepackage[parfill]{parskip}
\usepackage[utf8]{inputenc}
\usepackage{amsmath,amssymb,amsfonts,amsthm,mathrsfs,color}
\usepackage{graphicx}
\usepackage{subcaption}
\usepackage[skins]{tcolorbox}
\usepackage{bm}
\usepackage{units}
\usepackage{multirow}
\usepackage{comment}
\usepackage[colorlinks=true,linkcolor=blue,citecolor=blue]{hyperref}

\usepackage{xr-hyper} 
\usepackage{hyperref} 
\hypersetup{
    colorlinks=true, 
    linkcolor=blue, 
    citecolor=blue, 
    filecolor=blue, 
    urlcolor=blue
}

\newtheorem{Theorem}{Theorem}[section]
 
\newtheorem{Lemma}[Theorem]{Lemma}

\newtheorem{Remark}[Theorem]{Remark}  
\newtheorem{Assumption}[Theorem]{Assumption} 
\newtheorem{Example}[Theorem]{Example}

\DeclareMathOperator*{\argmax}{arg\,max}
\DeclareMathOperator*{\argmin}{arg\,min}

\newcommand\hl{\color{orange}\bf}

\let\OLDthebibliography\thebibliography
\renewcommand\thebibliography[1]{
  \OLDthebibliography{#1}
  \setlength{\parskip}{1pt}
  \setlength{\itemsep}{0pt plus 0.0ex}
}

\title{Stein's Method of Moments}
\author{Bruno Ebner\footnote{Bruno Ebner, Karlsruher Institut f\"ur Technologie, Germany. E-mail: bruno.ebner@kit.edu}, Adrian Fischer\footnote{Adrian Fischer, Université libre de Bruxelles, Belgium, University of Oxford, UK. E-mail: adrian.fischer@stats.ox.ac.uk}, Robert E. Gaunt\footnote{Robert E. Gaunt, The University of Manchester, UK. E-mail: robert.gaunt@manchester.ac.uk}, Babette Picker\footnote{Babette Picker, Karlsruher Institut f\"ur Technologie, Germany. E-mail: babette.picker@kit.edu}, and Yvik Swan\footnote{
Yvik Swan, Université libre de Bruxelles, Belgium. E-mail: yvik.swan@ulb.be}}
\begin{document}

\date{}
\maketitle

\vspace{-5mm}

\begin{abstract}
Stein operators allow to characterise probability distributions via differential operators. Based on these characterisations, we develop a new method of point estimation for marginal parameters of strictly stationary and ergodic processes, which we call \emph{Stein's Method of Moments} (SMOM). These SMOM estimators satisfy the desirable classical properties such as consistency and asymptotic normality. As a consequence of the usually simple form of the operator, we obtain explicit estimators in cases where standard methods such as (pseudo-) maximum likelihood estimation require a numerical procedure to calculate the estimate. In addition, with our approach, one can choose from a large class of test functions which allows to improve significantly on the moment estimator. Moreover, for i.i.d.\ observations, we retrieve data-dependent functions that result in asymptotically efficient estimators and give a sequence of explicit SMOM estimators that converge to the maximum likelihood estimator. Our simulation study demonstrates that for a number of important univariate continuous probability distributions our SMOM estimators possess excellent small sample behaviour, often outperforming the maximum likelihood estimator and other widely-used methods in terms of lower bias and mean squared error. We also illustrate the pertinence of our approach on a real data set related to rainfall modelisation. 
\end{abstract}

\section{Introduction}

Point estimation in a parametric model is one of the most classical problems in statistics.
In the case of independent and identically distributed (i.i.d.) data, maximum likelihood estimation (MLE) can count itself among the most sought-after, which is mostly due to its simple idea and asymptotic efficiency for regular target distributions. On the other hand, several difficulties can occur including highly complex probability density functions (PDFs), failure of numerical procedures due to local extrema of the likelihood function or to censoring and the complexity of extending the method to the non-i.i.d.\ case. Additionally, the efficient asymptotic behaviour of the MLE does not necessarily guarantee a high performance for smaller sample sizes. 

The method of moments provides a simple alternative to the MLE but requires that the moments of the target distribution can be calculated analytically. This is often the case for basic univariate probability distributions resulting in an explicit estimator which can serve as an initial guess for the numerical procedure in order to calculate the MLE. However, if the moments are of a complicated form, the moment estimator itself can only be computed through a numerical algorithm and loses its simplicity. Moreover, it is well-known that moment estimation is in general outplayed by the MLE regarding the asymptotic behaviour in the i.i.d.\ case. The generalised method of moments was introduced in \cite{hansen1982large} and is applicable
for stationary and ergodic time series and does not require an i.i.d.\ setting. The generalised method of moments incorporates a wide class of estimation techniques such as MLE and the classical method of moments. A difficulty that comes along with the method is the problem of finding a suitable target function. Moreover, estimation can get numerically tedious if the target function is complicated, and necessitates a first-step estimator if one wishes to minimise the asymptotic variance. \par

A vast number of alternative estimation techniques have been developed over the years. Amongst others, different kinds of minimum-distance approaches have been considered that compare characterising functions of the target distributions, such as the Fourier or Laplace transform, to empirical approximations. We refer to \cite[Chapter 3]{adler1998practical} ($\alpha$-stable distributions), \cite{koutrouvelis1982estimation} (Cauchy distribution), \cite{meintanis2016review} (mixtures of normal distributions), \cite{weber2006minimum} (Gompertz and Power exponential distribution among others), to name just a few references. \par

However, the methods mentioned above can run into numerical hardships as soon as the characterising object used for estimation becomes complicated. In this context, several approaches have been developed based on Stein characterisations of probability distributions, which lie at the heart of the powerful probabilistic technique Stein's method (\cite{stein1972bound}). Through Stein characterisations it is possible to eliminate the normalising constant; for example, Stein characterisations based on the \textit{density approach} to Stein's method (\cite{ley2013stein, ley2017distances}) involve the ratio $p'/p$, where $p$ is the density of the target distribution. 

\cite{betsch2021minimum} developed a new class of minimum-distance-type estimators based on Stein characterisations, in which new representations of the cumulative distribution function (CDF), which do not involve the normalising constant, are obtained in terms of an expectation and compare the respective sample mean to the empirical CDF (see also \cite{betsch2018characterizations}).
Recently, \cite{barp2019minimum} (see \cite{oates2022minimum} for a more recent reference) introduced a new class of estimators obtained through minimising a Stein discrepancy, whereupon their method incorporates the score matching approach, a further technique to estimate the parameters of non-normalised model based on the score function (see \cite{hyvarinen2005estimation}). However, through these approaches explicit estimators are only obtained in simple models and estimation becomes computationally challenging as soon as a numerical procedure is required.


This is where we want to tie in. In this paper, we study a new class of point estimators, which we refer to as \emph{Stein's Method of Moments (SMOM)} estimators, that are obtained through a Stein characterisation based on the density approach by applying the corresponding Stein operator to selected test functions and solving the resulting empirical version of the Stein identity for the unknown parameter. This combines the benefits of independence from a possibly complicated normalising constant and the simplicity of the estimator. A similar idea was already proposed in \cite{arnold2001multivariate}, in which the authors considered a generalised version of Hudson's identity in order to develop parameter estimators for exponential families. 
However, our work can be seen as a significant extension in which we consider a larger class of probability distributions and Stein operators. We also develop an asymptotic theory for our Stein estimators, addressing measurability, existence, (strong) consistency and asymptotic normality for marginal parameters of strictly stationary and ergodic processes, without even the need for an i.i.d.\ assumption. We make a further major contribution by addressing the problem of how to choose `optimal' test functions that result in asymptotically efficient estimators, and we are able to obtain sequences of explicit Stein estimators that converge to the MLE. 

Stein's method of moments is a highly universal approach to parameter estimation. Stein's density approach yields tractable Stein characterisations for many of the most important univariate distributions, and, with a suitable Stein characterisation at hand, one can readily deduce estimators with the following desirable features: (i) simple, explicit moment estimators, which through suitable choices of test functions, typically offer improvements on the usual moment estimators in terms of efficiency or mean squared error (MSE); or (ii) asymptotically efficient estimators that remain fully explicit. Moreover, as illustrated in extensive simulations presented in Section \ref{section_applications} and the Supplementary Information, we observe that SMOM estimators often possess excellent small sample behaviour; for a number of important univariate distributions our estimator improve on the MLE and specialised state-of-the-art methods in terms of bias and MSE across a wide range of parameter constellations. For these reasons, we believe that Stein's method of moments should become part of the statistician's toolkit as one of the default parameter estimation methods, at least in the context of univariate continuous probability distributions which is studied in detail in this paper. Moreover, as discussed in Section \ref{section_discussion}, SMOM has recently been extended to multivariate continuous probability distributions (see \cite{fgs23truncated,fgs23}), and the performance of the estimators in this setting remains very competitive, which demonstrates the versatility of SMOM for challenging estimation problems beyond the univariate setting of this paper.
We hope that this paper will inspire further research into this exciting method, and hope to see it further extended to a discrete and multivariate setting beyond the recent work \cite{fgs23truncated,fgs23}.



\vspace{2mm}

\noindent \emph{Notation.} For a real-valued function $(\theta,x) \mapsto g_{\theta}(x)$, where $\theta \in \mathbb{R}^p$ and $x \in \mathbb{R}$, we write $\frac{\partial}{\partial \theta} g_{\theta}(x)$ for its gradient with respect to $\theta=(\theta_1,\ldots,\theta_p)^\top$, which is a column vector of size $p$. If the function $g_{\theta}(x)$ takes values in $\mathbb{R}^q$ with $q \geq 2$, $\frac{\partial}{\partial \theta} g_{\theta}(x)$ is its Jacobian with respect to $\theta$, which is a $(q \times p)$-matrix. By $\frac{\partial}{\partial \theta_i}g_{\theta}(x)$ we mean the partial derivative with respect to $\theta_i$. If we want to address the derivative with respect to the argument in the parentheses (with respect to $x$ in $g_{\theta}(x)$) we simply write $g_{\theta}'(x)$ which remains real-valued or a column vector of size $q$. For a vector $x=(x_1,\ldots,x_n)^\top \in \mathbb{R}^n$, we denote by $\Vert x \Vert=(x_1^2 + \ldots + x_n^2)^{1/2}$ the standard Euclidean norm. Finally, for a (possibly non-square) matrix $X \in \mathbb{R}^{p \times q}$, we let $\Vert X \Vert$ be the spectral norm, which is defined as the square-root of the largest eigenvalue of $X^\top X$.  We also introduce the vectorisation map that stacks the columns of a matrix $M=(m_{ij}$, $1 \leq i \leq p$, $1 \leq j \leq q) \in \mathbb{R}^{p \times q}$, given by
 $  \mathrm{vec}:\mathbb{R}^{p \times q} \rightarrow \mathbb{R}^{pq}: M \mapsto (m_{11},\ldots , m_{p1},  m_{12}, \ldots, m_{p2}, \ldots, \ldots, m_{pq})^{\top}$.

\section{Elements of Stein's Method} \label{section_steins_method}
We begin with a short introduction to the version of Stein's method 
employed in this paper.

Let $\mathbb{P}_{\theta}$ be a probability distribution on $(a,b) \subset \mathbb{R}$ with corresponding differentiable PDF $p_{\theta}(x)$ that depends on a parameter $\theta \in \Theta \subset \mathbb{R}^p$, where we assume that $\mathbb{P}_{\theta_1}=\mathbb{P}_{\theta_2} $ implies $\theta_1=\theta_2 $ for $\theta_1,\theta_2 \in \Theta$. Throughout the paper, we assume  that $\Theta$ is open and convex as well as that $-\infty \leq a < b \leq \infty$ and $p_{\theta}(x)>0$ for all $\theta \in \Theta$ and $x \in (a,b)$. Let $X$ be a real-valued random variable with values in $(a,b)$, $\mathscr{F}_{\theta}$ a class of functions $f:(a,b) \rightarrow \mathbb{R}$, and $\mathcal{A}_{\theta}$ an operator defined on $\mathscr{F}_{\theta}$. We call $(\mathcal{A}_{\theta},\mathscr{F}_{\theta})$ a \emph{Stein pair} for $\mathbb{P}_{\theta}$ if the following is satisfied:
\begin{align}
\mathbb{E}[\mathcal{A}_{\theta}f(X)] = 0  \mbox{ for all } f \in \mathscr{F}_{\theta} \qquad \text{if and only if} \qquad X \sim \mathbb{P}_{\theta};
\label{Stein_equation}
\end{align}
operator $\mathcal{A}_{\theta}$ is called a \emph{Stein operator} for $\mathbb{P}_\theta$, and $\mathscr{F}_\theta$ is the associated \emph{Stein class}. There exist many ways to obtain Stein pairs for any given distribution, see e.g.\ \cite{anastasiou2023stein,ley2013stein}. In this paper, we consider those obtained via the density approach as developed in \cite{ley2017distances,ley2013stein}. First-order density approach Stein operators are of the form
\begin{align} \label{definition_Stein_kernel_operator}
\mathcal{A}_{\theta}f(x)= \frac{(\tau_{\theta}(x)p_{\theta}(x)f(x))'}{p_{\theta}(x)},
\end{align}
where $\tau_{\theta}$ is some differentiable function $\tau_{\theta}:(a,b) \rightarrow \mathbb{R}$; they act on the function class
\begin{align} \label{definition_function_class}
\mathscr{F}_{\theta}:= \bigg\{f:(a,b) \rightarrow \mathbb{R} \ \vert \ f \text{ is differentiable and} \int_a^b \big( f(x) \tau_{\theta}(x) p_{\theta}(x)\big)'\,dx=0  \bigg\}.
\end{align}
The next theorem states that  $(\mathcal{A}_{\theta},\mathscr{F}_{\theta})$ is a Stein pair for $\mathbb{P}_{\theta}$. 
In the Supplementary Information, we give a proof that goes along the lines of \cite[Theorem 2.2]{ley2013stein}. 
\begin{Theorem} \label{theorem_stein_operator}
Let $\mathcal{A}_{\theta}$ be the Stein operator defined in \eqref{definition_Stein_kernel_operator} and $\mathscr{F}_{\theta}$ the corresponding class of functions introduced in \eqref{definition_function_class}. Moreover, assume that $\tau_{\theta}(x) \neq 0$ almost everywhere on $(a,b)$ and let $X$ be a random variable with values in $(a,b)$. Then the Stein characterisation \eqref{Stein_equation} holds. 
\end{Theorem}

It is often convenient to use in \eqref{definition_Stein_kernel_operator} the so-called {\it Stein kernel} $\tau_{\theta}(x)=(1/p_{\theta}(x)) \int_x^{b} (\mathbb{E}[X]-y)p_{\theta}(y)\,dy$, $x \in (a,b)$, whose corresponding density approach  Stein operator is
\begin{align}\label{sop}
\mathcal{A}_{\theta}f(x)= \tau_{\theta}(x)f'(x) + (\mathbb{E}[X]-x)f(x).
\end{align} 
This last operator takes a simple form in many cases. For instance  $\tau_{\theta} $ is polynomial for  members of the Pearson family (see \cite[Theorem 1, p.\ 65]{stein86} and \cite[Lemma 2.9]{gms19}). 
We refer to \cite{ernst2020first,saumard2019weighted} for an overview of Stein kernels and their properties.
Other choices of functions $\tau_\theta$ in \eqref{definition_Stein_kernel_operator} are also sometimes better suited. 

\begin{Example}[Gaussian distribution] \label{example_gaussian}
Consider the Gaussian distribution $N(\mu,\sigma^2)$ with parameter $\theta=(\mu,\sigma^2)$, $\mu \in \mathbb{R}$, $\sigma^2 >0$.
A simple calculation gives that the Stein kernel is $\tau_{\theta}(x)=\sigma^2$. We have $\mathbb{E}[X]=\mu$ and retrieve from (\ref{sop}) the well-known Stein operator of \cite{stein1972bound},
\begin{align}\label{steinop1}
\mathcal{A}_{\theta}f(x)=\sigma^2f'(x)+(\mu- x)f(x).
\end{align}
\end{Example}

\begin{Example}[Gamma distribution]
\label{example_gamma}
Consider the gamma distribution $\Gamma(\alpha,\beta)$ with parameter $\theta=(\alpha,\beta)$, $\alpha, \beta >0$, and density $p_{\theta}(x)=\beta^{\alpha}x^{\alpha-1}e^{-\beta x}/\Gamma(\alpha)$, $x >0$.
The Stein kernel is $\tau_{\theta}(x)=x$. Since $\mathbb{E}[X]=\alpha/\beta$, we recover the gamma Stein operator of \cite{diaconis1991closed},
\begin{align}\label{gammasteineqn}
\mathcal{A}_{\theta}f(x)=xf'(x)+(\alpha-\beta x )f(x).
\end{align}
\end{Example}

\section{Stein's Method of Moments} 

\subsection{Definition and properties} \label{estimators_def_prop}
Let $\{X_n, n \in \mathbb{Z}\}$ be a real-valued strictly stationary and ergodic discrete process defined on a common probability space $(\Omega,\mathcal{F},\mathbb{P})$. In order to clarify the terminology we elaborate on what we mean by strict stationarity and ergodicity. We say that $\{X_n,n\in \mathbb{Z}\}$ is strictly stationary if
$
    \{X_n, n \in \mathbb{Z} \} =_D \{X_{n+k}, n \in \mathbb{Z} \}
$
for each $k \in \mathbb{Z}$. Moreover, let $\zeta:\Omega \rightarrow \Omega$ be measurable such that $X_{n+1}(\omega)=X_n(\zeta(\omega))$ for each $\omega \in \Omega$ and $n \in \mathbb{N}$. Then we say that $\{X_n, n\in \mathbb{Z}\}$ is ergodic if $\zeta$ is measure-preserving ($\mathbb{P}(\zeta^{-1}(A))=\mathbb{P}(A)$ for all $A \in \mathcal{F}$) and the $\sigma$-algebra of invariant events $\mathcal{I}=\{A \in \mathcal{F} \, \vert \, \zeta^{-1}(A)=\zeta(A) \}$ is $\mathbb{P}$-trivial, i.e.\ $\mathbb{P}(A) \in \{0,1\}$ for all $A \in \mathcal{I}$. We assume that the marginal distribution of each $X_n$, $n \in \mathbb{Z}$, is $\mathbb{P}_{\theta_0}$ for some $\theta_0 \in  \Theta$. 
Now suppose that the measures $\mathbb P_{\theta}$ are characterised  through  Stein pairs $(\mathcal{A}_{\theta}, \mathscr{F}_{\theta})$ and let $\mathscr{F} = \cap_{\theta \in \Theta} \mathscr{F}_\theta$. 

For the purpose of estimating the unknown parameter $\theta_0$ from a sample $X_1,\ldots,X_n$ drawn from the stochastic process $\{X_n,n\in \mathbb{Z}\}$, we choose $p$ measurable test functions $f_1,\ldots,f_p$ (belonging to $\mathscr{F}$)  and, in light of \eqref{Stein_equation}, replace the expectations with their empirical counterparts. Therefore, we get the following system of equations:
\begin{equation}\label{stein_equation_system}
    \frac{1}{n}\sum_{i=1}^n \mathcal{A}_{\theta}f(X_i) = 0,
\end{equation}
where we write $\mathcal{A}_{\theta}f:\Theta \times (a,b) \rightarrow \mathbb{R}^p$ for the function defined by
$
(\theta,x) \mapsto \mathcal{A}_{\theta}f(x):=(\mathcal{A}_{\theta}f_1(x),\ldots,\mathcal{A}_{\theta}f_p(x))^\top.
$
In the following, we will refer to \eqref{stein_equation_system} as the \textit{empirical} Stein identity. Moreover, we will call any solution to this system of equations with respect to $\theta$ a \textit{Stein estimator}, which we denote by $\hat{\theta}_n$.

With this definition at hand, one observes that Stein estimators can be seen as moment estimators (resp.\ generalised moment estimators as proposed in \cite{hansen1982large}), whereupon we suggest suitable target functions through Stein's method. The necessary conditions on the test functions $f_1,\ldots,f_p$, the Stein operator $\mathcal{A}_{\theta}$ and the target distribution $\mathbb{P}_{\theta}$ in order to achieve existence, measurability and asymptotic normality of the Stein estimator $\hat{\theta}_n$ will be introduced below.  
Subsequently, we will impose the following assumptions. 
\begin{Assumption} \label{general_assumptions} \leavevmode
\begin{description} 
\item[(a)] Let $X \sim \mathbb{P}_{\theta_0}$ and $\theta \in \Theta$. Then $f = (f_1, \ldots, f_p) \in \mathscr{F}$ is such that $\mathbb{E}[\mathcal{A}_{\theta}f(X)]=0$ if and only if $\theta=\theta_0$.
\item[(b)] Let $q \geq p$ and $X \sim \mathbb{P}_{\theta_0}$. We can write $\mathcal{A}_{\theta}f(x) =M(x)g(\theta)$ for some measurable $p \times q$ matrix $M$ with $\mathbb{E}[\Vert M(X) \Vert] < \infty $ and a continuously differentiable function $g=(g_1,\ldots, g_q)^\top:\Theta \rightarrow \mathbb{R}^q $ for all $\theta \in \Theta$, $x \in (a,b)$. We also assume that $\mathbb{E}[M(X)]\frac{\partial}{\partial \theta} g(\theta) \vert_{\theta=\theta_0}$ is invertible.
\end{description}
\end{Assumption}
Assumption \ref{general_assumptions}(a) assures that the true parameter $\theta_0$ can be well identified by means of the Stein operator  $\mathcal A_\theta$; this assumption can be easily verified (for a proper choice of test functions) for operators of the form \eqref{definition_Stein_kernel_operator} with Theorem \ref{theorem_stein_operator}. Assumption \ref{general_assumptions}(b) requires that the parameters can be well separated from the sample. Moreover, if the function $g$ is fairly simple, we are likely to obtain explicit estimators; this turns out to be the case for all examples considered in this paper. 
\begin{Theorem} \label{theorem_consistency}
Suppose Assumption \ref{general_assumptions}(a)--(b) is satisfied. The probability that a solution to \eqref{stein_equation_system}, $\hat{\theta}_n$, exists and is measurable converges to $1$ as $n \rightarrow \infty$. Furthermore, $\hat{\theta}_n$ is strongly consistent in the following sense: There is a set $A \subset \Omega$ with $\mathbb{P}(A)=1$ such that for all $\omega \in A$  there exists $N=N(\omega) \in \mathbb{N}$ such that $\hat{\theta}_n$ exists for all $n \geq N$ and 
$\hat{\theta}_n(\omega) \rightarrow \theta_0.$
\end{Theorem}

As we will see in Section \ref{section_applications} and the further examples in the Supplementary Information, the new estimators will mostly be solutions to systems of linear equations which exist and are measurable with probability $1$ for any sample size if $\Theta=\mathbb{R}^p$. Nonetheless, it can happen that an estimator returns a value which lies outside of the truncation domain if the parameter space is a strict subset of $\mathbb{R}^p$. These issues will be addressed separately for each example in Section \ref{section_applications} and the further examples in the Supplementary Information.

Asymptotic normality can be obtained similarly as for the classical moment estimators. We state the result in the next theorem. We slightly change the meaning of $\hat{\theta}_n$ as we need our estimator to be a random variable in order to establish weak convergence. To this end, let $X \sim \mathbb{P}_{\theta_0}$. Define the function $F(M,\theta)=Mg(\theta)$, where $M \in \mathbb{R}^{p \times q}$ and $g$ is as in Assumption \ref{general_assumptions}(b). In the proof of Theorem \ref{theorem_consistency}, we see that there are neighbourhoods $U \subset \mathbb{R}^{p \times q}$, $V \subset \mathbb{R}^p$ of $\mathbb{E}[M(X)]$ and $\theta_0$ such that there exists a continuously differentiable function $h:U \rightarrow V$ with $F(M,h(M))=0$ for all $M \in U$. 
We now define $A_n=\big\{n^{-1} \sum_{i=1}^n M(X_i) \in U \big\}$, and note that it is shown in the proof of Theorem \ref{theorem_consistency} that $\mathbb{P}(A_n) \rightarrow 1$ as $n \rightarrow \infty$.
We now denote by $\hat{\theta}_n$ any measurable map $\hat{\theta}_n : \Omega \rightarrow \Theta$ that solves \eqref{stein_equation_system} for any $\omega  \in A_n$ and is equal to any other measurable function outside of $A_n$. 
We will also need some further assumptions, which can be efficiently stated by recalling a version of a central limit theorem for strictly stationary and ergodic time series stated in \cite{hannan1973central} and originally proved in \cite{gordin1969central} (see also \cite{hansen1982large}).
\begin{Theorem} \label{ergodic_central_limit_theorem}
    Let $\{Y_n,n \in \mathbb{Z}\}$ be strictly stationary and ergodic with values in $\mathbb{R}^p$. Moreover, suppose that $\mathbb{E}[Y_1]=0$ and $\mathbb{E}[\Vert Y_1 Y_1^{\top} \Vert]<\infty$ as well as
    $
        \mathbb{E}[\mathbb{E}[Y_0 \, \vert \, Y_{-j},\dots] \mathbb{E}[Y_0 \, \vert \, Y_{-j},\dots]^{\top} ] \rightarrow 0$, $j \rightarrow \infty$.
    Furthermore, for
   $
        Y_j'= \mathbb{E}[Y_0 \, \vert \, Y_{-j},\dots] - \mathbb{E}[Y_0 \, \vert \, Y_{-j-1},\dots]$, $ j\geq 0$,   
    we suppose that $\sum_{j=0}^{\infty}\mathbb{E}[(Y_j')^{\top} Y_j']^{1/2} < \infty$. Then
    $
        n^{-1/2}\sum_{i=1}^n Y_i\rightarrow_D N(0,\Xi),
    $
    where $\Xi = \sum_{i \in \mathbb{Z}} \mathbb{E}[Y_0Y_i^{\top}]$.
\end{Theorem}

We can now state our asymptotic normality result.

\begin{Theorem} \label{theorem_asymptotic_normality}
Let $X \sim \mathbb{P}_{\theta_0}$. Suppose Assumption \ref{general_assumptions}(a)--(b) is fulfilled. Moreover, assume that the matrix $\mathbb{E}[\mathrm{vec}(M(X))\mathrm{vec}(M(X))^{\top}]$ exists and that the time series $\{Y_n,n \in \mathbb{Z}\}$, where $Y_n=\mathrm{vec}(M(X_n))$, $n \in \mathbb{Z}$, satisfies the assumptions of Theorem \ref{ergodic_central_limit_theorem}. Now let
\begin{align*}
\Psi=\sum_{j \in \mathbb{Z}}\mathbb{E}\big[\mathcal{A}_{\theta_0}f(X_0) \mathcal{A}_{\theta_0}f(X_{j})^\top\big] \quad \text{and} \quad G=\mathbb{E} \bigg[\frac{\partial}{\partial \theta}\mathcal{A}_{\theta}f(X) \Big\vert_{\theta=\theta_0} \bigg],
\end{align*}
and let $\hat{\theta}_n$ be defined as in the preceding paragraph. Then the sequence $\sqrt{n}(\hat{\theta}_n-\theta_0)$ is asymptotically normal with mean zero and covariance matrix $G^{-1}\Psi G^{-\top}$.
\end{Theorem}

\begin{Remark}
    Note that in the case where $\{X_n, n \in \mathbb{Z}\}$ is i.i.d., the assumptions of Theorem \ref{ergodic_central_limit_theorem} are easily verified and the matrix $\Psi$ appearing in the asymptotic covariance simplifies to
    \begin{align*}
        \Psi=\mathbb{E}\big[\mathcal{A}_{\theta_0}f(X) \mathcal{A}_{\theta_0}f(X)^\top\big], \quad X \sim \mathbb{P}_{\theta_0}.
    \end{align*}
\end{Remark}

Let us consider two simple examples to demonstrate our estimation method and its flexibility. For that purpose, we write $\overline{f(X)}=n^{-1} \sum_{i=1}^nf(X_i)$ for a measurable function $f:(a,b) \rightarrow \mathbb{R}$.
\begin{Example}[Gaussian distribution, continuation of Example \ref{example_gaussian}]
Since we have two unknown parameters, we choose two test functions $f_1,f_2$ and therefore from (\ref{steinop1}) get
\begin{align*}
\begin{cases}
\overline{f_1(X)}\mu+\overline{f_1'(X)} \sigma^2 = \overline{Xf_1(X)} \\
\overline{f_2(X)}\mu+\overline{f_2'(X)} \sigma^2  = \overline{Xf_2(X)}.
\end{cases}
\end{align*}
By solving this system of linear equations for $\mu$ and $\sigma^2$ we obtain the Stein estimators
\begin{align}\label{norest}
\hat{\mu}_n=\frac{\overline{f_2'(X)} \ \overline{Xf_1(X)}-\overline{f_1'(X)} \ \overline{Xf_2(X)}}{\overline{f_1(X)} \ \overline{f_2'(X)}- \overline{f_1'(X)} \ \overline{f_2(X)}}, \quad
\hat{\sigma}_n^2=\frac{\overline{f_1(X)} \ \overline{Xf_2(X)} -\overline{f_2(X)} \ \overline{Xf_1(X)}}{\overline{f_1(X)} \ \overline{f_2'(X)}- \overline{f_1'(X)} \ \overline{f_2(X)}}.
\end{align}
Taking $f_1(x)=1$, $f_2(x)=x$ yields the MLE or moment estimators $\hat{\mu}_n=\overline{X}$ and $
\hat{\sigma}_n^2=\overline{X^2}-\overline{X}^2$. 
\end{Example}

\begin{Example}[Gamma distribution, continuation of Example \ref{example_gamma}]\label{example_gamma2}
We  choose two different test functions $f_1$, $f_2 $, and from (\ref{gammasteineqn}) we readily
obtain the estimators
\begin{align*}
\hat{\alpha}_n=\frac{\overline{Xf_2(X)} \ \overline{Xf_1'(X)}-\overline{Xf_1(X)} \ \overline{Xf_2'(X)}}{\overline{Xf_1(X)} \ \overline{f_2(X)}-\overline{f_1(X)} \ \overline{Xf_2(X)}}, \quad
\hat{\beta}_n=\frac{\overline{f_2(X)} \ \overline{Xf_1'(X)}-\overline{f_1(X)} \ \overline{Xf_2'(X)}}{\overline{Xf_1(X)} \ \overline{f_2(X)}-\overline{f_1(X)} \ \overline{Xf_2(X)}}.
\end{align*}
By choosing $f_1(x)=1$ and $f_2(x)=x$ we retrieve the moment estimators 
\begin{align*}
\hat{\alpha}_n^{\mathrm{MO}}=\frac{\overline{X}^2}{\overline{X^2}-\overline{X}^2} \quad \text{and} \quad 
\hat{\beta}_n^{\mathrm{MO}}=\frac{\overline{X}}{\overline{X^2}-\overline{X}^2}. 
\end{align*}
Moreover, by choosing $f_1(x)=1$ and $f_2(x)=\log x$ we obtain the logarithmic estimators
\begin{align} \label{gamma_log_estimators}
\begin{split}
\hat{\alpha}_n^{\mathrm{LOG}}=\frac{\overline{X}}{\overline{X \log X}-\overline{X} \ \overline{\log X}} \quad \text{and} \quad 
\hat{\beta}_n^{\mathrm{LOG}}=\frac{1}{\overline{X \log X}-\overline{X} \ \overline{\log X}},
\end{split}
\end{align} 
which show a behaviour close to asymptotic efficiency and were obtained through the generalised gamma distribution in \cite{ye2017closed} (see also \cite{wiens2003class} for an earlier reference). 
\end{Example}

\subsection{Optimal functions} \label{subsection_optimal_function}

We show that it is possible to achieve asymptotic efficiency under certain regularity conditions using Stein estimators by using specific parameter-dependent test functions. To this end, we suppose in this section without further notice that the sequence of random variables $\{X_n, n \in \mathbb{Z}\}$ is i.i.d.\ (for possible extensions to non-i.i.d.\ data see Remark \ref{remark_extension_ergodic_optimal functions}). In addition, we assume that the Stein operator $\mathcal{A}_{\theta}$ can be written in the form \eqref{definition_Stein_kernel_operator}. Within this framework we compare our estimators to the MLE, which we will denote by 
$\hat{\theta}_n^{\mathrm{ML}}$, and 
which, under certain regularity conditions on the likelihood function, is defined through the equation
\begin{align} \label{MLE_definition}
\frac{\partial}{\partial \theta} \overline{\log p_{\theta}(X)} \Big\vert_{\theta=\hat{\theta}_n^{\mathrm{ML}}}=0.
\end{align}
It is well-known that for regular probability distributions the expectation of the latter expression is equal to zero.
It is a standard result that, under certain regularity conditions, a suitable standardisation of the MLE $\hat{\theta}_n^{\mathrm{ML}}$ is asymptotically efficient with covariance matrix $I_{\mathrm{ML}}^{-1}(\theta_0)$, the inverse of the Fisher-information matrix $I_{\mathrm{ML}}(\theta)$.

Motivated by the definition of the MLE, we consider the score function as the right-hand side of the Stein identity
\begin{align}
\mathcal{A}_{\theta}f(x)= \frac{\partial}{\partial \theta} \log p_{\theta}(x). 
\label{Stein_operator_equal_score}
\end{align}
This is an ordinary differential equation whose solution $f_{\theta}$ clearly depends on the unknown parameter $\theta$. If the Stein operator is of the form \eqref{definition_Stein_kernel_operator}, then the solution of \eqref{Stein_operator_equal_score} is given by
\begin{align}
f_{\theta}(x)= \big(f_{\theta}^{(1)}(x),\ldots,f_{\theta}^{(p)}(x)\big)^{\top} = \frac{\frac{\partial}{\partial \theta} P_{\theta}(x) +c}{\tau_{\theta}(x)p_{\theta}(x)}, \quad x \in (a,b), \label{eq:opsol}
\end{align}
where $c \in \mathbb{R}$ and $P_{\theta} $ is the CDF corresponding to $p_{\theta}$, with the convention that  $f_{\theta}(x)=0$ at all $x\in (a, b)$ such that  $\tau_{\theta}(x)=0$. We will refer to the functions \eqref{eq:opsol} as the \emph{optimal functions}. Thus, 
$
\overline{\mathcal{A}_{\theta}f_{\theta}(X)}=0
$
is the maximum likelihood equation rewritten in terms of Stein operators. One can now use a consistent first-step estimator $\tilde{\theta}_n$ for the unknown parameter $\theta$ in $f_{\theta}= \big(f_{\theta}^{(1)},\ldots,f_{\theta}^{(p)}\big)^{\top}$ and resolve the system of equations \eqref{stein_equation_system} with respect to these test functions. This holds the advantage that estimators may remain explicit if the Stein operator is simple. Mathematically speaking, given a first-step estimator $\tilde{\theta}_n$, we define $\hat{\theta}_n^{\star}$ through the equation
\begin{align} \label{definition_two_step_esti}
\overline{\mathcal{A}_{\hat{\theta}_n^\star} f_{\tilde{\theta}_n}(X)}=0
\end{align}
if such a solution exists. In this setting, the matrix $M$ from Assumption \ref{general_assumptions}(b) depends on the parameter $\theta$ through data-dependent test functions. Hence, we introduce a new set of assumptions.

\begin{Assumption} \label{assumptions_optimal_func} \leavevmode
\begin{description} 
\item[(a)] $\tilde{\theta}_n$ is a consistent estimator, i.e. $\tilde{\theta}_n \rightarrow_{\mathbb{P}} \theta_0$.
\item[(b)] Let $X \sim \mathbb{P}_{\theta_0}$ and $\theta_1, \theta_2 \in \Theta$.  Then    $f_{\theta_2}\in \mathscr{F}$, and $\mathbb{E}[\mathcal{A}_{\theta_1}f_{\theta_2}(X)]=0$ if and only if $\theta_1=\theta_0$.
\item[(c)] For $q \geq p$, we can write $\mathcal{A}_{\theta_1}f_{\theta_2}(x) =M_{\theta_2}(x)g(\theta_1)$ for some measurable $p \times q$ matrix $M_{\theta_2}$ and a continuously differentiable function $g=(g_1,\ldots, g_q)^\top:\Theta \rightarrow \mathbb{R}^q $ for all $\theta_1,\theta_2 \in \Theta$, $x \in (a,b)$. Moreover, we assume that $\mathbb{E}[M_{\theta_0}(X)]\frac{\partial}{\partial \theta} g(\theta) \vert_{\theta=\theta_0}$, where $X \sim \mathbb{P}_{\theta_0}$, is invertible and the function $\theta \mapsto \mathrm{vec}(M_{\theta}(x))$ is continuously differentiable on $\Theta$ for all $x \in (a,b)$.
\item[(d)] For $X \sim \mathbb{P}_{\theta_0}$ there exist two functions $F_1$, $F_2$ on $(a,b)$ with $\mathbb{E}[F_i(X)]<\infty$, $i=1,2$, and compact neighbourhoods $\Theta', \Theta''$ of $\theta_0$ such that $\Vert M_{\theta}(x) \Vert  \leq F_1(x) $ for all $\theta \in \Theta'$ and $\Vert \frac{\partial}{\partial \theta} \mathrm{vec}(M_{\theta}(x)) \Vert \leq F_2(x)$ for all $\theta \in \Theta''$, $x \in (a,b)$.
\end{description}
\end{Assumption}
Assumptions \ref{assumptions_optimal_func}(b)--(c) are adapted versions of Assumptions \ref{general_assumptions}(a)--(b) with the supplement that the optimal function $f_{\theta}$ needs to be an element of the Stein class $\mathscr{F}$ for each $\theta \in \Theta$. The invertibility of $\mathbb{E}[M_{\theta_0}(X)]\frac{\partial}{\partial \theta} g(\theta) \vert_{\theta=\theta_0}$, $X \sim \mathbb{P}_{\theta_0}$, in (c) is easily verified for Stein operators that are linear in $\theta$. However, we have the additional Assumption \ref{assumptions_optimal_func}(d) which can be tedious to verify if $f_{\theta}$ is complicated. Nevertheless, the latter assumption is satisfied for all our applications in Section \ref{section_applications} and the Supplementary Information. 

\begin{Theorem} \label{theorem_consistency_two_step}
   Suppose  Assumptions \ref{assumptions_optimal_func}(a)--(d)
   are fulfilled. The probability that a solution to \eqref{definition_two_step_esti}, $\hat{\theta}_n^\star$, exists and is measurable converges to $1$ as $n \rightarrow \infty$, and $\hat{\theta}_n^\star$ is (weakly) consistent in the following sense: There is a sequence of sets $C_n \subset \Omega$, $n \in \mathbb{N}$ with $\mathbb{P}(C_n)\rightarrow 1$ such that, for all $\epsilon >0$,
$\mathbb{P}(\vert \hat{\theta}_n^\star - \theta_0\vert > \epsilon \, \vert \, C_n ) \rightarrow 0.$
\end{Theorem}

We have only shown weak consistency in the previous theorem in contrast to strong consistency in Theorem~\ref{theorem_consistency}. However, if we have $\tilde{\theta}_n \overset{\mathrm{a.s.}}{\longrightarrow} \theta_0$, it is an easy task to show that we also have $\hat{\theta}_n^\star \overset{\mathrm{a.s.}}{\longrightarrow} \theta_0$ (in the sense of Theorem \ref{theorem_consistency}). In the following theorem, we show that, under some additional technical assumptions, the two-step Stein estimators are asymptotically normal and reach asymptotic efficiency. Again, in order to maneuver around existence and measurability issues, we define $\hat{\theta}_n^\star$ to be a random variable that is equal to the solution of \eqref{definition_two_step_esti} on the sets $C_n$ as defined in the proof of Theorem \ref{theorem_consistency_two_step} and equal to some other measurable function otherwise.
\begin{Theorem} \label{theorem_efficiency_optimal_functions}
Suppose that Assumptions \ref{assumptions_optimal_func}(a)--(d) are satisfied. Moreover, assume that $p_{\theta}$ is differentiable with respect to $\theta$ and
\begin{itemize}
\item[(i)] the sequence of random vectors $\sqrt{n}(\tilde{\theta}_n-\theta_0)$ is uniformly tight;
\item[(ii)] $\mathcal{A}_{\theta}$ is of the form \eqref{definition_Stein_kernel_operator} with $\tau_{\theta}$ differentiable with respect to $\theta$;
\item[(iii)] we have $\lim_{x \rightarrow a,b} \frac{\partial}{\partial \theta} \big( 
p_{\theta}(x)\tau_{\theta}(x) \big) \big\vert_{\theta=\theta_0}f_{\theta_0}(x) =0$;
\item[(iv)] and $I_{\mathrm{ML}}(\theta_0)$ exists and is finite.
\end{itemize}
Then, for $\hat{\theta}_n^\star$ as defined in the preceding paragraph, 
$\sqrt{n}(\hat{\theta}_n^{\star} - \theta_0) \rightarrow_D N(0,I_{\mathrm{ML}}^{-1}(\theta_0))$, as $n \rightarrow \infty$.
\end{Theorem}

We refer the reader as well to \cite[Section 6]{neweylarge}, in which the asymptotic theory of two-step estimators is studied - although under slightly different assumptions and with the additional restriction that the first-step estimate needs to be obtained through the generalised method of moments.

\begin{Remark} \label{remark_extension_ergodic_optimal functions}
    It is possible to extend the results from Theorems \ref{theorem_consistency_two_step} and \ref{theorem_efficiency_optimal_functions} to strictly stationary and ergodic time series as introduced in Section \ref{estimators_def_prop}. For Theorem \ref{theorem_consistency_two_step}, it suffices to apply an adapted uniform strong law of large numbers as stated in \cite[Theorem 2.1]{hansen2012proofs} (note that with Assumption \ref{assumptions_optimal_func}(d) the random function $\theta \rightarrow M_{\theta}(X), X \sim \mathbb{P}_{\theta_0}$, is automatically first-moment-continuous, compare \cite[p.\ 206]{degroot2005optimal}). With the latter result together with Theorem \ref{ergodic_central_limit_theorem} we can also generalise Theorem \ref{theorem_efficiency_optimal_functions}, although we need the additional assumption that the sequence $\{\mathcal{A}_{\theta_0}f_{\theta_0}(X_n), n\in \mathbb{Z}\}$ satisfies the assumptions of Theorem \ref{ergodic_central_limit_theorem}. We then get
    \begin{align*}
        \sqrt{n}(\hat{\theta}_n^{\star} - \theta_0) \stackrel{D}{\longrightarrow} N\bigg(0,I_{\mathrm{ML}}^{-1}(\theta_0)  \bigg(\sum_{j \in \mathbb{Z}} \mathbb{E}\big[\mathcal{A}_{\theta_0}f_{\theta_0}(X_0)\mathcal{A}_{\theta_0}f_{\theta_0}(X_{j})^{\top} \big] \bigg) I_{\mathrm{ML}}^{-1}(\theta_0)\bigg), \quad n \rightarrow \infty.
    \end{align*}
\end{Remark}

\begin{Remark}
    There is another possibility to achieve asymptotic efficiency of point estimators. \cite{carrasco2000generalization} proposed a generalised method-of-moments-type estimator with a continuum of moment conditions. The idea is based on using an uncountably infinite number of moment conditions, i.e.\ a class of functions $h^t:\Theta \times (a,b) \rightarrow \mathbb{R}, t \in \Pi \subset \mathbb{R}$ such that $\mathbb{E}[h^t(\theta_0,X)]=0$ for all $t \in \Pi$, where $X \sim \mathbb{P}_{\theta_0}$. Under some conditions, the sequence of functions $n^{-1/2} \sum_{i=1}^n h^t(\theta,X_i)$, $t \in \Pi$,
    converges to some zero-mean Gaussian process with covariance operator $\Upsilon$ by the functional central limit theorem as $n \rightarrow \infty$. Let $\Upsilon^{\alpha_n}$ be its Tikhonov regularisation with smoothing term $\alpha_n$. Then under additional assumptions it can be shown that the estimator 
    \begin{align*}
        \hat{\theta}_n= \mathrm{arg\,min}_{\theta \in \Theta} \big\Vert (\Upsilon_n^{\alpha_n})^{-1/2} h_n^t(\theta,X) \big\Vert_{\mathscr{L}^2},
    \end{align*}
    where $(\Upsilon_n^{\alpha_n})^{-1/2}$ is an estimate of $(\Upsilon^{\alpha_n})^{-1/2}$, $h_n^t(\theta,X)=n^{-1} \sum_{i=1}^n h^t(\theta,X_i)$ and $\Vert \cdot \Vert_{\mathscr{L}^2}$ is the standard ${\mathscr{L}^2}$-norm with respect to some positive measure, is asymptotically efficient (see \cite{carrasco2014asymptotic}). Note that this procedure requires an estimation of a covariance operator and is computationally ambitious. For more information and some applications see also \cite{carrasco2007efficient,carrasco2002efficient,carrasco2002simulation}.
\end{Remark}

\begin{Example}[Gamma distribution, continuation of Example \ref{example_gamma2}]\label{gammaexample5}
Let us now introduce a two-step Stein estimator for the gamma distribution. We first recall that the CDF of the gamma distribution is given by $P_{\theta}(x)=\gamma(\alpha,\beta x)/\Gamma(\alpha)$,
where $\gamma(\cdot,\cdot)$ is the lower incomplete gamma function. With this formula at hand, we can calculate the optimal functions, which are given by 
\begin{align*}
& f_{\theta}^{(1)}(x)= e^{\beta  x}\Big(\frac{\gamma (\alpha , \beta x)}{(\beta x)^\alpha}\big( \log (\beta  x)-\psi(\alpha ) \big) -  \frac{1}{\alpha^2} \, _2F_2(\alpha,\alpha;1+\alpha,1+\alpha;-\beta x) \Big), \quad f_{\theta}^{(2)}(x)=\frac{1}{\beta},
\end{align*}
where $\, _2F_2$ denotes the generalised hypergeometric function. Taking $\hat{\theta}_n^{\mathrm{LOG}}$ as a first-step estimate results in a two-step estimator, which we denote by $\hat{\theta}_n^{\mathrm{ST}}$. This estimator takes a rather complicated form, but remains completely explicit. In the Supplementary Information, we show that the assumptions of Theorems \ref{theorem_consistency_two_step} and \ref{theorem_efficiency_optimal_functions} hold, which implies (strong) consistency and asymptotic efficiency of $\hat{\theta}_n^{\mathrm{ST}}$. Simulation results are reported in the Supplementary Information, which show that the Stein estimator $\hat{\theta}_n^{\mathrm{ST}}$ has a marginally improved performance in terms of lower bias and mean square error over the (non-explicit) MLE in small sample sizes across a range of parameter values.

As the CDF of the gamma distribution is expressed in terms of special functions, the optimal functions take a rather complicated form. For distributions with simpler CDFs, simpler optimal functions can be obtained; see, for example, the Cauchy distribution in Section \ref{cauchysec}.
\end{Example}

In the remainder of this section, we study the sequence of Stein estimators which is obtained as follows: Choose some $\theta^0 \in \Theta $ as a value for $\theta$ in $f_{\theta}$ and solve for the two-step Stein estimator $\hat{\theta}_n^{\star}$. Take then the obtained estimate as a new value for $\theta$ in $f_{\theta}$ in order to update the Stein estimator $\hat{\theta}_n^{\star}$. 
Formally speaking, we consider the sequence of Stein estimators $\hat{\theta}_n^{(m)}$ defined by 
\begin{align} \label{definition_iterative_procedure}
0 =   \overline{\mathcal{A}_{\hat{\theta}_n^{(m+1)}}f_{\hat{\theta}_n^{(m)}}(X)} ,
\end{align}
where $\hat{\theta}_n^{(0)} =\theta^0 \in \Theta$ is the starting value of the iterating process. Moreover, let $\Theta_0 \subset \Theta$ be compact and convex with $\theta_0, \theta^0 \in \Theta_0$. We briefly discuss the existence of such a sequence. It is clear from Theorem \ref{theorem_consistency_two_step} that, for fixed $m \in \mathbb{N}$, the probability that $\hat{\theta}_n^{(m)}$ exists converges to $1$. However, this does not guarantee the existence of the sequence. Therefore, when we study the asymptotic behaviour of the sequence $\hat{\theta}_n^{(m)}$, $m \in \mathbb{N}$, we have to assume that such a sequence of solutions of \eqref{definition_iterative_procedure} exist. Before stating the theorem, we introduce a new set of assumptions.
\begin{Assumption} \label{assumptions_iteratice_proc} \leavevmode
\begin{description} 
\item[(a)] The MLE exists and is unique with probability converging to $1$. Moreover, we assume that the MLE is consistent (in the sense of Theorem \ref{theorem_consistency_two_step}) and that if the MLE exists, it is characterised by \eqref{MLE_definition}.
\item[(b)] Let $X \sim \mathbb{P}_{\theta_0}$ and $\theta_1, \theta_2 \in \Theta_0$. Then $f_{\theta_2} \in \mathscr{F}$, and $\mathbb{E}[\mathcal{A}_{\theta_1}f_{\theta_2}(X)]=0$ if and only if $\theta_1=\theta_0$.
\item[(c)] For $q \geq p$, we can write $\mathcal{A}_{\theta_1}f_{\theta_2}(x) =M_{\theta_2}(x)g(\theta_1)$ for some measurable $p \times q$ matrix $M_{\theta_2}$ and $g=(g_1,\ldots, g_q)^\top:\Theta \rightarrow \mathbb{R}^q $ continuously differentiable for all $\theta_1,\theta_2 \in \Theta_0$, $x \in (a,b)$. Moreover, we assume that $\mathbb{E}[M_{\theta}(X)]\frac{\partial}{\partial \theta} g(\theta) $ (where $X \sim \mathbb{P}_{\theta_0}$) is invertible for all $\theta \in \Theta_0$ and that the function $\theta \mapsto \mathrm{vec}(M_{\theta}(x))$ is continuously differentiable on $\Theta_0$ for all $x \in (a,b)$.
\item[(d)] For $X \sim \mathbb{P}_{\theta_0}$, there exist two functions $F_1$, $F_2$ on $(a,b)$ with $\mathbb{E}[F_i(X)]<\infty$, $i=1,2$, such that $\Vert M_{\theta}(x) \Vert  \leq F_1(x)$ and $\Vert \frac{\partial}{\partial \theta} \mathrm{vec}(M_{\theta}(x)) \Vert \leq F_2(x)$ for all $\theta \in \Theta_0$, $x \in (a,b)$.
\end{description}
\end{Assumption}
Assumptions  \ref{assumptions_iteratice_proc}(b)--(d) introduced above are mostly equivalent to Assumptions \ref{assumptions_optimal_func}(b)--(d), although we have a slight modification in (c). Here, we require the matrix  $\mathbb{E}[M_{\theta}(X)]\frac{\partial}{\partial \theta} g(\theta)$, $X \sim \mathbb{P}_{\theta_0}$, to be invertible for all $\theta \in \Theta_0$, in contrast to \ref{assumptions_optimal_func}(c) in which this needs to be the case only for $\theta=\theta_0$, which can be difficult to verify, especially if $f_{\theta}$ is complicated.
\begin{Theorem} \label{theorem_sequence_stein_estimator}
Suppose that Assumptions \ref{assumptions_iteratice_proc}(a)--(d) hold. Then, for each sequence $\hat{\theta}_n^{(m)}$, $m \in \mathbb{N}$, satisfying \eqref{definition_iterative_procedure}, there exists a sequence of sets $A_n \subset \Omega$, $n \in \mathbb{N}$, with $\mathbb{P}(A_n) \rightarrow 1$ as $n \rightarrow \infty$ such that for each $n$ we have that on $A_n$,
$
\hat{\theta}_n^{(m)} \rightarrow \hat{\theta}_n^{\mathrm{ML}}$ as $m \rightarrow  \infty$.
\end{Theorem}

\section{Applications} \label{section_applications}

In this section, we apply Stein's method of moments to three challenging estimation problems for univariate distributions that have received interest in the literature. Here we establish small sample performance of our asymptotically efficient estimators obtained in Section~\ref{subsection_optimal_function}, and, by choosing suitable test functions, we propose alternatives to moment estimation that are as simple and improve significantly in terms of asymptotic variance.  We conclude the section with an application to a less challenging setting for which censoring can break down performance of ML and other MOM estimates, herewith demonstrating the usefulness of our approach even for regular models.  
 For all examples, we suppose that $\{X_n,n \in \mathbb{Z}\}$ is i.i.d. However, we stress that SMOM can be applied to dependent data, and we give such an application in  Example \ref{cauchy_noniid} of the Supplementary Information.

The applications presented in this section demonstrate the power and universality of Stein's method of moments, as from a single framework we are able to efficiently derive explicit estimators with desirable asymptotic properties that typically outperform competitor methods in our simulation studies. In each example, we compare to the MLE and, where appropriate, the classical moment estimators, as well as more specialist estimators that have been found to perform well for the particular distribution under consideration. It would seem that the minimum Stein discrepancy estimators developed in \cite{barp2019minimum} would be natural competitors as the discrepancy is based on the density approach Stein identity. However, we have excluded them from our simulation studies, as we found that they are outperformed for almost all parameter values in terms of bias and MSE, involve more computational effort and for certain distribution require a numerical procedure even when our Stein estimators are completely explicit; a more detailed justification is given in Section \ref{appa} of the Supplementary Information.

Further examples for the beta, Student's $t$, Lomax, Nakagami, one-sided truncated inverse-gamma distribution and generalised logistic distributions, as well as a non-i.i.d.\ example are given in the Supplementary Information. Some of these estimators also have excellent performance.


\subsection{Truncated normal distribution} \label{sec_truncnormal}
The density of the two-sided truncated Gaussian distribution on $(a,b)$ with $a,b \in \mathbb{R}$, denoted by $TN(\mu,\sigma^2)$, $\theta=(\mu,\sigma^2)$, is given by $p_\theta(x)=C_\theta\phi((x-\mu)/\sigma)$, where $C_\theta^{-1}=\sigma[\Phi(b-\mu)/\sigma)-\Phi(a-\mu)/\sigma)]$
and $\phi$ and $\Phi$ are the standard Gaussian PDF and CDF, respectively. With $\tau_{\theta}(x)=\sigma^2$ we obtain the same Stein operator as in Example \ref{example_gaussian}.
Note that the function class $\mathscr{F}_{\theta}$ differs from the one in the untruncated case.
As we have the same Stein operator as in Example \ref{example_gaussian}, we obtain for two test functions $f_1$, $f_2$ the same expressions for the Stein estimators as in the untrucated case, as given by (\ref{norest}).
Note that the normalising constant drops out and therefore completely explicit and easily computable estimators are retrieved. A natural choice seem to be the polynomials
\begin{align*}
    f_1(x)= -x^2+(a+b)x-ab, \quad
    f_2(x)= x^3-\frac{3}{2}(a+b)x^2+\frac{1}{2}(a^2+4ab+b^2)x-\frac{1}{2}(a^2b+ab^2).
\end{align*}
We denote the Stein estimator based on the latter test functions by $\hat{\theta}_n^{\mathrm{ST}}=(\hat{\mu}_n^{\mathrm{ST}},\hat{\sigma}_n^{\mathrm{ST}})$. 

The first and second moments of the truncated normal distribution take a rather complicated form involving the functions $\phi$ and $\Phi$ (see the Supplementary Information), and consequently the classical moment estimators, which we denote by  $\hat{\mu}_n^{\mathrm{MO}}$ and $\hat{\sigma}_n^{\mathrm{MO}}$, must be obtained numerically. The MLE $\hat{\theta}_n^{\mathrm{ML}}=(\hat{\mu}_n^{\mathrm{ML}},\hat{\sigma}_n^{\mathrm{ML}})$ is also not explicit, and, as for the classical moment estimator, the numerical calculation can be tedious.

A $q$-confidence region of the corresponding asymptotic normal distribution for the above estimation techniques is reported in Figure \ref{fig:truncnormal_variance} for two parameter constellations. The ellipses are plotted with respect to the two eigenvectors $v_1$ and $v_2$ of the covariance matrix and are therefore parallel to the $x$- resp.\, $y$-axis. One can see that the performance of the proposed Stein estimator essentially coincides with that of the MLE, indicating a behaviour close to efficiency. The moment estimator performs poorly, and was hence excluded from our finite sample simulation study.

\begin{figure}[!]
    \centering
\vspace{.2cm}
\begin{subfigure}{\textwidth}
   \begin{subfigure}{.49\textwidth}
\captionsetup{width=.95\textwidth}
  \centering
  \includegraphics[width=7.4cm]{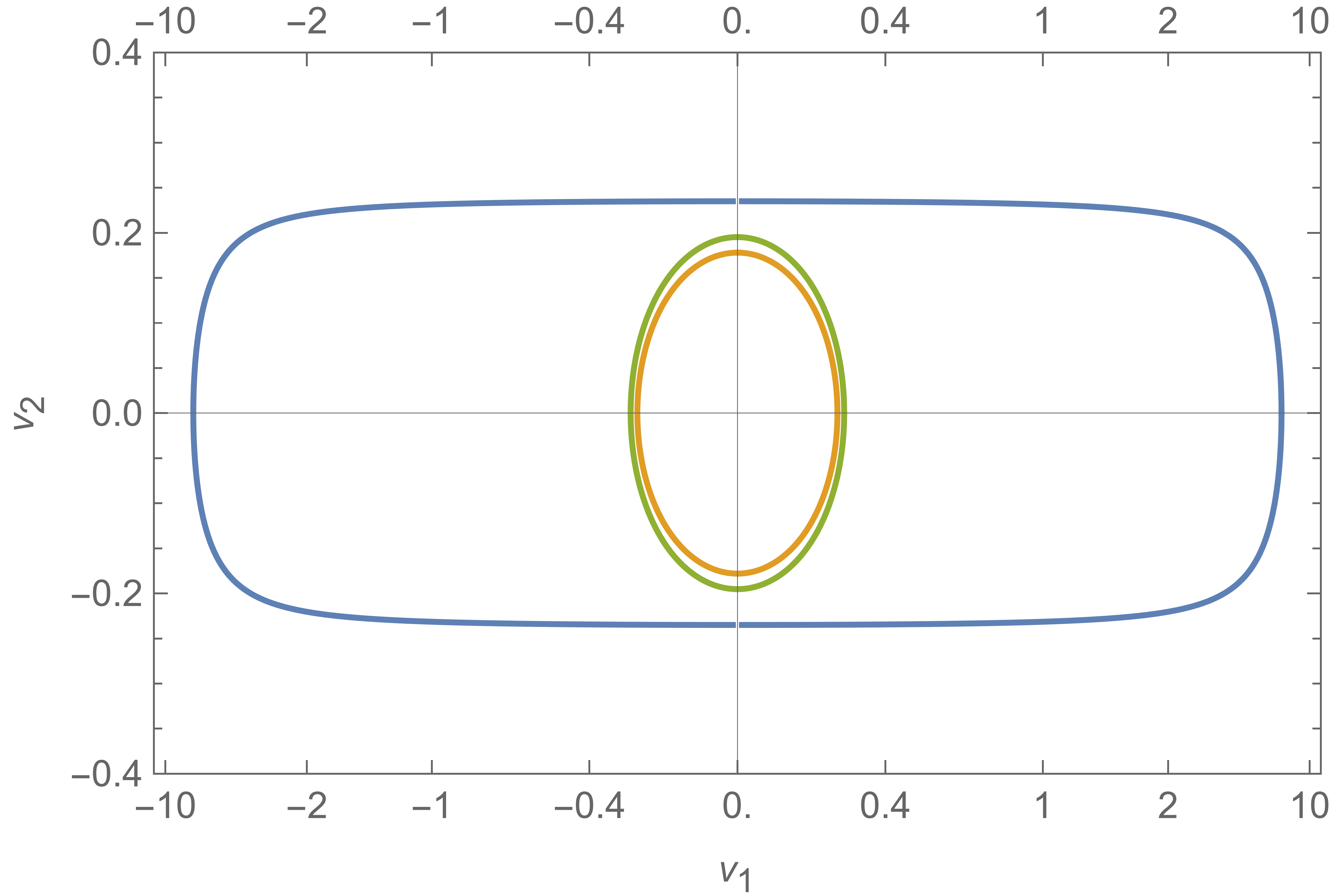}
  \caption{$\mu=0.5$, $\sigma=0.2$}
\end{subfigure}%
  \begin{subfigure}{.49\textwidth}
\captionsetup{width=.95\textwidth}
  \centering
  \includegraphics[width=7.4cm]{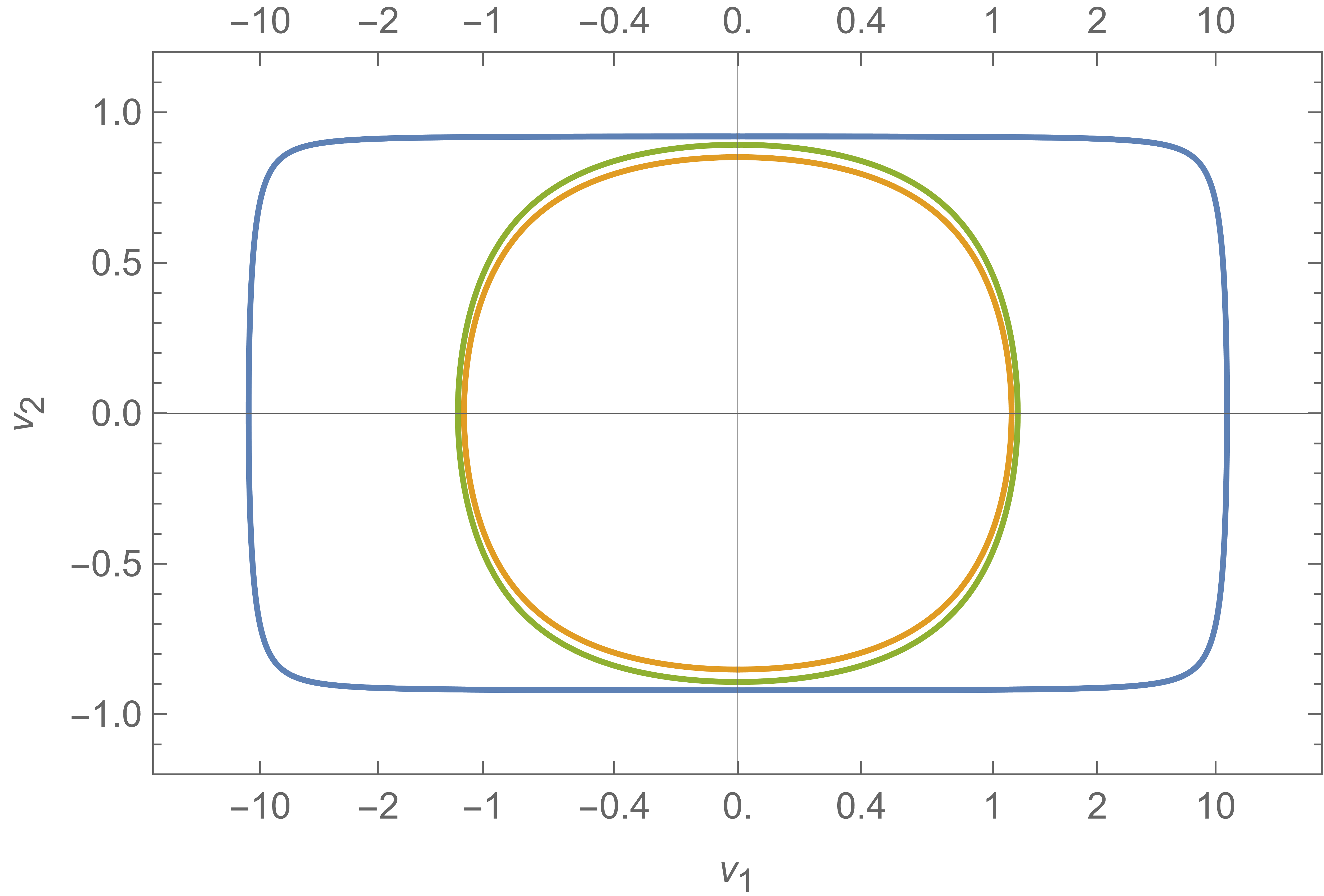}
  \caption{$\mu=0.5$, $\sigma=0.3$}
\end{subfigure}%
\end{subfigure}%
\caption{\protect\label{fig:truncnormal_variance} \it Asymptotic confidence regions for the estimators of the $TN(\mu,\sigma^2)$ distribution for $q=0.95$, $a=0$ and $b=1$ in the directions of the eigenvectors $v_1$ and $v_2$. Plotted are the MLE \includegraphics[scale=1]{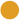}, the moment estimator \includegraphics[scale=1]{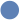} and the Stein estimator \includegraphics[scale=1]{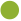}. The $x$-axis scale is transformed via $x \mapsto \arctan x$.}
\end{figure}

\cite{hegde1989estimation} showed that the MLE exists if and only if 
$
    \overline{Y}^2 < \overline{Y^2} < 1 - 2\overline{Y}/x^{\star}$ with $\coth(x^{\star})- 1/x^{\star}=\overline{Y}$,
where $Y_i=2(X_i-a)/(b-a)$.  The conditions for existence of the moment estimator seem to be difficult to work out. It is a known issue for any explicit estimator that it is possible for the estimate to lie outside of the parameter space if the latter is restricted to a certain subset of Euclidean space. This problem also applies to the Stein estimator. We added a column \textit{NE} to the tables to report the estimated relative frequency of cases in which the estimator does not exist (the relative frequency is given as a number between $0$ and $100$). These estimates are based on the same Monte Carlo samples as the estimates for bias and MSE. However, for the considered parameter constellations, existence of the estimator seems to be hardly an issue. Nevertheless, we noticed that in cases where parameter estimation for the $TN(\mu, \sigma^2)$-distribution becomes in general more difficult (for example, when $\mu$ lies outside of the truncation domain and $\sigma^2$ is large), the number of Monte Carlo samples for which the MLE and the Stein estimator does not exist grows rapidly. As can be seen in Table \ref{truncnormal_sim}, the Stein estimator yields better results in terms of bias and MSE for most parameter constellations and both sample sizes considered.

\begingroup
\setlength\tabcolsep{2pt}
\begin{table}[h] \small
\centering
\begin{tabular}{cc|cc|cc|cc|cc|cc|cc}
 $\theta_0$ & & \multicolumn{6}{|c}{n=20} & \multicolumn{6}{|c}{n=50} \\  \hline 
     & & \multicolumn{2}{|c}{Bias} & \multicolumn{2}{|c}{MSE} & \multicolumn{2}{|c}{NE} & \multicolumn{2}{|c}{Bias} & \multicolumn{2}{|c}{MSE} & \multicolumn{2}{|c}{NE} \\  \hline
 & & $\hat{\theta}_n^{\mathrm{ML}}$ & $\hat{\theta}_n^{\mathrm{ST}}$ & $\hat{\theta}_n^{\mathrm{ML}}$ & $\hat{\theta}_n^{\mathrm{ST}}$ & $\hat{\theta}_n^{\mathrm{ML}}$ & $\hat{\theta}_n^{\mathrm{ST}}$ %
 & $\hat{\theta}_n^{\mathrm{ML}}$ & $\hat{\theta}_n^{\mathrm{ST}}$ & $\hat{\theta}_n^{\mathrm{ML}}$ & $\hat{\theta}_n^{\mathrm{ST}}$ & $\hat{\theta}_n^{\mathrm{ML}}$ & $\hat{\theta}_n^{\mathrm{ST}}$\\ \hline
\multirow{2}{*}{$(0.5,0.05)$} & $\mu$  & $4.46\text{e-5}$ &  $\hl 4.3\text{\hl e-5}$ &  $1.25\text{e-4}$ &  $1.25\text{e-4}$ & \multirow{2}{*}{0} & \multirow{2}{*}{0} & $-9.54\text{e-5}$ & \hl $\hl -9.17\text{e-5}$ & \hl $\hl 4.99\text{e-5}$ & $5\text{e-5}$ & \multirow{2}{*}{0} & \multirow{2}{*}{0}
\\ & $\sigma$  & $-1.19\text{e-4}$ & \hl $\hl -1.12\text{e-4}$ & \hl $\hl6.04\text{e-7}$ & $6.1\text{e-7}$ &&& $-5.37\text{e-5}$ & \hl $\hl -5.05\text{e-5}$ & \hl $\hl 2.49\text{e-7}$ & $2.5\text{e-7}$
\\ \hline 
\multirow{2}{*}{$(0.5,0.1)$} & $\mu$  & \hl $\hl 1.54\text{e-5}$ & $4.23\text{e-5}$ & \hl $\hl 5.04\text{e-4}$ & $5.09\text{e-4}$ & \multirow{2}{*}{0} & \multirow{2}{*}{0} & \hl $\hl -7.14\text{e-5}$ & $-7.19\text{e-5}$ & \hl $\hl 1.96\text{e-4}$ & $1.99\text{e-4}$ & \multirow{2}{*}{0} & \multirow{2}{*}{0}
\\ & $\sigma$  & $-4.88\text{e-4}$ & \hl $\hl -3.67\text{e-4}$ & \hl $\hl 9.92\text{e-6}$ & $1.05\text{e-5}$ &&& $-1.79\text{e-4}$ & \hl $\hl -1.21\text{e-4}$ & \hl $\hl 3.94\text{e-6}$ & $4.12\text{e-6}$
\\ \hline 
\multirow{2}{*}{$(0.5,0.2)$} & $\mu$  & \hl $\hl -2.01\text{e-4}$ & $5.86\text{e-4}$ & 0.094 & \hl 0.021 & \multirow{2}{*}{0} & \multirow{2}{*}{0} & $ 3.28\text{e-3}$ & \hl $\hl 3.94\text{e-5}$ & 0.075 & \hl $\hl 9.79\text{e-4}$ & \multirow{2}{*}{0} & \multirow{2}{*}{0}
\\ & $\sigma$  & 0.026 & \hl $\hl 4.58\text{e-3}$ & 4.49 & \hl 0.042 &&& 0.024 & \hl $\hl 4.82\text{e-4}$ & 4.53 & \hl $\hl 1.23\text{e-4}$\\ \hline 
\multirow{2}{*}{$(0.5,0.3)$} & $\mu$  & 0.024 & \hl $\hl -1.41\text{e-3}$ & 1.35 & \hl 1.09 & \multirow{2}{*}{3} & \multirow{2}{*}{3} & $-5.56\text{e-3}$ & \hl $\hl -2.46\text{e-3}$ & 0.383 & \hl 0.025 & \multirow{2}{*}{0} & \multirow{2}{*}{0} 
\\ & $\sigma$  & 0.122 & \hl 0.117 & \hl 3.06 & 9.56 &&& 0.062 & \hl 0.024 & 5.71 & \hl 0.064\\ \hline 
\multirow{2}{*}{$(0.6,0.05)$} & $\mu$  & 0.016 & \hl $\hl -1.59\text{e-4}$ & 0.403 & \hl $\hl 1.25\text{e-4}$ & \multirow{2}{*}{0} & \multirow{2}{*}{0} & $5.25\text{e-3}$ & \hl $\hl -4.95\text{e-5}$ & 0.142 & \hl $\hl 5.07\text{e-5}$ & \multirow{2}{*}{0} & \multirow{2}{*}{0}
\\ & $\sigma$  & 0.015 & \hl $\hl -1.16\text{e-4}$ & 0.365 & \hl $\hl 6.15\text{e-7}$ &&& $4.55\text{e-3}$ & \hl $\hl -5.04\text{e-5}$ & 0.106 & \hl $\hl 2.49\text{e-7}$\\ \hline 
\multirow{2}{*}{$(0.6,0.1)$} & $\mu$  & 0.019 & \hl $\hl 2.35\text{e-4}$ & 0.477 & \hl $\hl 4.98\text{e-4}$ & \multirow{2}{*}{0} & \multirow{2}{*}{0} & 0.023 & \hl $\hl -2.64\text{e-5}$ & 0.584 & \hl $\hl 2.05\text{e-4}$ & \multirow{2}{*}{0} & \multirow{2}{*}{0}
\\ & $\sigma$  & 0.017 & \hl $\hl -4.5\text{e-4}$ & 0.435 & \hl $\hl 1.02\text{e-5}$ &&& 0.021 & \hl $\hl -1.32\text{e-4}$ & 0.515 & \hl $\hl 4.18\text{e-6}$\\ \hline 
\multirow{2}{*}{$(0.6,0.2)$} & $\mu$  & 0.03 & \hl $\hl 4.73\text{e-3}$ & 0.661 & \hl $\hl 3.51\text{e-3}$ & \multirow{2}{*}{0} & \multirow{2}{*}{0} & 0.016 & \hl $\hl 1.14\text{e-3}$ & 0.394 & \hl $\hl 1.11\text{e-3}$ & \multirow{2}{*}{0} & \multirow{2}{*}{0}
\\ & $\sigma$  & 0.045 & \hl $\hl 2.32\text{e-3}$ & 5.05 & \hl $\hl 7.93\text{e-4}$ &&& 0.015 & \hl $\hl 5.9\text{e-4}$ & 0.381 & \hl $\hl 1.37\text{e-4}$ \\ \hline 
\multirow{2}{*}{$(0.7,0.05)$} & $\mu$  &  $-6.32\text{e-5}$ & \hl $\hl -4.6\text{e-5}$ & \hl $\hl 1.2\text{e-4}$ & $1.21\text{e-4}$ & \multirow{2}{*}{0} & \multirow{2}{*}{0} & $2.81\text{e-3}$ & \hl $\hl -3.54\text{e-5}$ & 0.081 & \hl $\hl 5.11\text{e-5}$ & \multirow{2}{*}{0} & \multirow{2}{*}{0}
\\ & $\sigma$  & $-1.29\text{e-4}$ & \hl $\hl -1.21\text{e-4}$ & \hl $\hl 6\text{e-7}$ & $6.21\text{e-7}$&&& $9.96\text{e-4}$ & \hl $\hl -4.95\text{e-5}$ & 0.011 & \hl $\hl 2.57\text{e-7}$\\ \hline 
\multirow{2}{*}{$(0.7,0.1)$} & $\mu$  & $2.43\text{e-3}$ & \hl $\hl -7.98\text{e-5}$ & 0.061 & \hl $\hl 5.4\text{e-4}$ & \multirow{2}{*}{0} & \multirow{2}{*}{0}& $7.91\text{e-4}$ & \hl $\hl 1.19\text{e-4}$ & $4.18\text{e-3}$ & \hl $\hl 2.11\text{e-4}$ & \multirow{2}{*}{0} & \multirow{2}{*}{0}
\\ & $\sigma$  & $5.76\text{e-4}$ & \hl $\hl -3.48\text{e-4}$ & $9.37\text{e-3}$ & \hl $\hl 1.24\text{e-5}$ &&& \hl $\hl 7.57\text{e-5}$ & $-1.6\text{e-4}$ & $6.13\text{e-4}$ & \hl $\hl 4.89\text{e-6}$\\ \hline 
\multirow{2}{*}{$(0.7,0.2)$} & $\mu$  & $0.017$ & $0.017$ & 0.086 & \hl 0.036 & \multirow{2}{*}{0} & \multirow{2}{*}{0}& $4.71\text{e-3}$ & \hl $\hl 3.36\text{e-3}$ & 0.017 & \hl $\hl 1.78\text{e-3}$ & \multirow{2}{*}{0} & \multirow{2}{*}{0}
\\ & $\sigma$  & $6.14\text{e-3}$ & \hl $\hl 5.45\text{e-3}$ & 0.071 & \hl $\hl5.77\text{e-3}$&&& $1.29\text{e-3}$ & \hl $\hl 9.36\text{e-4}$ & $3.91\text{e-3}$ & \hl $\hl 2.04\text{e-4}$\\ \hline 
\end{tabular} 
\caption{\protect\label{truncnormal_sim} Simulation results for the $TN(\mu,\sigma)$ distribution with $a=0,b=1$ for $n\in \{20, 50\}$ and $10,000$ repetitions.}
\end{table}
\endgroup

\subsection{Cauchy distribution} \label{cauchysec}
For $\theta=(\mu, \gamma) \in \mathbb{R} \times (0,\infty)$, the density of the Cauchy distribution is given by $p_\theta(x)=(\pi\gamma)^{-1}(1+((x-\mu)/\gamma)^2)^{-1}$, $x\in\mathbb{R}$.
We fix $\tau_{\theta}=(x-\mu)^2+\gamma^2$, and obtain
$
    \mathcal{A}_{\theta}f(x)=\big( (x-\mu)^2+\gamma^2 \big) f'(x)
$
(see also \cite{schoutens2001orthogonal}). For test functions $f_1$, $f_2$ we obtain the estimators
\begin{align*} 
    \hat{\mu}_n=  \frac{\overline{f_2'(X)} \ \overline{X^2f_1'(X)} - \overline{f_1'(X)} \ \overline{X^2f_2'(X)}}{2\big[\overline{f_2'(X)} \ \overline{Xf_1'(X)} - \overline{f_1'(X)} \ \overline{Xf_2'(X)}\big]},  \:
    \hat{\gamma}_n^2= \frac{\overline{X^2f_2'(X)} \ \overline{Xf_1'(X)} -  \overline{X^2f_1'(X)} \ \overline{Xf_2'(X)} } {\overline{f_1'(X)} \ \overline{Xf_2'(X)} - \overline{f_2'(X)} \ \overline{Xf_1'(X)} } - 
   \hat{\mu}_n^2. 
\end{align*}
The CDF is $P_{\theta}(x)=\pi^{-1} \arctan\big( (x-\mu)/\gamma\big)+1/2$, and we thus obtain simple optimal functions:
\begin{align} \label{optimal_func_cauchy}
    f_{\theta}^{(1)}(x)= - \frac{1}{\gamma^2+(x-\mu)^2}  \quad \text{and} \quad
    f_{\theta}^{(2)}(x)=\frac{\mu - x}{\gamma \big(\gamma^2+(x-\mu)^2\big)}.
\end{align}
With a suitable first step estimate, we have an efficient estimator which is considerably simpler to compute than the MLE, which involves solving polynomial equations of degree $2n-1$ given by
\begin{align*}
    \sum_{i=1}^n \frac{2(X_i- \mu)}{\gamma^2 + (X_i-\mu)^2}=0 \quad \text{and} \quad \frac{n}{\gamma} - \sum_{i=1}^n \frac{2\gamma}{\gamma^2 + (X_i - \mu)^2} =0.
\end{align*}
In \cite{copas1975unimodality,gabrielsen1982unimodality} it is shown that, in the case where both parameters $\mu$ and $\gamma$ are unknown, the likelihood function is unimodal under some regularity assumptions. Clearly, moment estimation is not tractable due to the non-existence of all moments. 

Interestingly, parameter estimation for the Cauchy distribution can be difficult in the case where $\gamma$ is known and one is left with estimation of the location parameter $\mu$. We therefore now focus on the case that $\gamma$ is known. The parameter space thus reduces to $\Theta=\mathbb{R}$ with $\theta=\mu$. This estimation problem has received great attention in the literature; see \cite{zhang2010highly} for an overview of available estimation techniques. The MLE of $\mu$ with known $\gamma$ is often cited as an example of computational failure (\cite{bai1987maximum,zhang2010highly} summarise the challenges) although \cite{bai1987maximum} show that the MLE remains the asymptotically optimal estimator in the Bahadur sense. One reason for this is a multimodal likelihood function (in fact the number of local maxima is asymptotically Poisson distributed with mean $1/\pi$; see \cite{reeds1985asymptotic}). However, closed-form expressions for the MLE exist for sample sizes $3$ and $4$; see \cite{ferguson1978maximum}. Due to the difficulties concerning the MLE, other methods have been developed. In our simulation study we consider the L-estimator methods of \cite{rothenberg1964note}, \cite{bloch1966note}, \cite{chernoff1967asymptotic} and \cite{zhang2010highly}, which we denote by $\hat{\mu}_n^{\mathrm{L1}}$, $\hat{\mu}_n^{\mathrm{L2}}$, $\hat{\mu}_n^{\mathrm{L3}}$ and $\hat{\mu}_n^{\mathrm{L4}}$, respectively. We also consider the Pitman estimator of \cite{freue2007pitman}, which we denote by $\hat{\mu}_n^{\mathrm{PI}}$. The explicit forms of these estimators are given in the Supplementary Information. \cite{zhang2010highly} modified the estimator $\hat{\mu}_n^{\mathrm{L3}}$ in order to also achieve high efficiency for finite sample sizes, and so we do not include the estimator $\hat{\mu}_n^{\mathrm{L3}}$ in our simulation study.

Let us now describe a procedure based on the Stein operator. Note that if we choose one test function (since we only have to estimate $\mu$), the corresponding equation is quadratic and has in general two solutions. This is why we choose two test functions and consider the estimator $\hat{\mu}_n$ with test functions $f_\theta^{(1)}(x)$ and $f_\theta^{(2)}(x)$ as defined in (\ref{optimal_func_cauchy}), whereby $\gamma$ is now considered known. 
We take $\hat{\mu}_n^{\mathrm{L4}}$ as a first-step estimate, 
and denote the resulting estimator by $\hat{\mu}_n^{\mathrm{ST1}}$. Note that $\hat{\mu}_n^{\mathrm{ST1}}$ is not translation-invariant and we therefore consider different values of $\mu$ in our simulation. 

This slight modification of the estimation procedure still results in an asymptotically efficient estimator. We apply Theorem \ref{theorem_consistency_two_step} in the setting where both parameters are estimated (where we take $\tilde{\gamma}_n=\gamma_0$ as the first-step estimator for $\gamma$). Then the asymptotic variance of $\hat{\mu}_n^{\mathrm{ST1}}$ is the top-left element of the inverse Fisher information matrix in the case where $\gamma$ is unknown. The latter is given by the diagonal matrix $I_{\mathrm{ML}}^{-1}(\mu,\gamma)=\mathrm{diag}(2\gamma^2,2\gamma^2)$,
and we conclude that the asymptotic variance for both estimators $\hat{\mu}_n^{\mathrm{ST1}}$ and the (one-dimensional) MLE equals $2\gamma^2$. 

However, when performing the simulations we noticed a very large variance for $\hat{\mu}_n^{\mathrm{ST1}}$ for small sample sizes, which is consistent with the trade-off between small sample size and asymptotic efficiency noticed by \cite{zhang2010highly} for $\hat{\mu}_n^{\mathrm{L3}}$ and $\hat{\mu}_n^{\mathrm{L4}}$. This is why we propose a modified version of $\hat{\mu}_n^{\mathrm{ST1}}$, denoted by $\hat{\mu}_n^{\mathrm{ST2}}$. In a similar manner to the estimator $\hat{\mu}_n^{\mathrm{L1}}$, we cut off the bottom and top $p$-quantile of the sample at hand and calculate the sample means in $\hat{\mu}_n$ and $\hat{\gamma}_n^2$ by the means of the remaining observations. Pursuant to $\hat{\mu}_n^{\mathrm{L1}}$, we choose $p=0.38$ and disregard the first and last $\lfloor np \rfloor$ observations of the sorted sample. Simulation results can be found in Tables \ref{cauchy_sim_n20} and \ref{cauchy_sim_n50}. For the sample size $n=20$, the Pitman estimator $\hat{\mu}_n^{\mathrm{PI}}$ seems to be globally the best, with the modified L-estimator $\hat{\mu}_n^{\mathrm{L4}}$ close behind. The Stein estimator $\hat{\mu}_n^{\mathrm{ST2}}$ delivers good results as well, outperforming $\hat{\mu}_n^{\mathrm{L1}}$ and $\hat{\mu}_n^{\mathrm{L2}}$ for most parameter constellations. For the sample size $n=50$, $\hat{\mu}_n^{\mathrm{L4}}$, $\hat{\mu}_n^{\mathrm{PI}}$ and $\hat{\mu}_n^{\mathrm{ST1}}$ show the best performance regarding the bias and the Stein estimator $\hat{\mu}_n^{\mathrm{ST1}}$ has the lowest MSE for most parameter constellations. Further simulations results for the sample sizes $n=100$ and $n=250$ are given in the Supplementary Information. For these sample sizes, the Stein estimator $\hat{\mu}_n^{\mathrm{ST1}}$ has the lowest MSE for all parameter constellations, and the Pitman estimator performs very poorly. Indeed, our simulations  suggest that $\hat{\mu}^{\mathrm{PI}}$ is not a consistent estimator. 

\begingroup
\setlength\tabcolsep{4pt}
\begin{table} [h]\small
\centering
\begin{tabular}{cc|ccccc|ccccc} 
 $\theta_0$ & & \multicolumn{5}{|c}{Bias} & \multicolumn{5}{|c}{MSE} \\ \hline
 & & $\hat{\mu}_n^{\mathrm{L1}}$ & $\hat{\mu}_n^{\mathrm{L2}}$  & $\hat{\mu}_n^{\mathrm{L4}}$ & $\hat{\mu}_n^{\mathrm{PI}}$ & $\hat{\mu}_n^{\mathrm{ST2}}$ & $\hat{\mu}_n^{\mathrm{L1}}$ & $\hat{\mu}_n^{\mathrm{L2}}$  & $\hat{\mu}_n^{\mathrm{L4}}$ & $\hat{\mu}_n^{\mathrm{PI}}$ & $\hat{\mu}_n^{\mathrm{ST2}}$ \\ \hline
\multirow{1}{*}{$(-5,1)$} & $\mu$  & $-0.725$ & $0.131$ & $-2.54\text{e-3}$ & $ \hl 1.72\text{\hl e-4}$ & $0.085$ & $0.721$ & $0.944$ & $0.136$ & $ \hl 0.116$ & $0.158$ \\ \hline 
\multirow{1}{*}{$(-4,1.5)$} & $\mu$  & $-0.506$ & $0.193$ & $-4.22\text{e-3}$ & $ \hl -3.13\text{\hl e-5}$ & $0.122$ & $0.695$ & $1.04$ & $0.307$ & $ \hl 0.261$ & $0.351$ \\ \hline 
\multirow{1}{*}{$(-2,2)$} & $\mu$  & $-0.099$ & $0.261$ & $8.84\text{e-3}$ & $ \hl 3.88\text{\hl e-3}$ & $0.182$ & $0.782$ & $2.26$ & $0.537$ & $ \hl 0.46$ & $0.627$ \\ \hline 
\multirow{1}{*}{$(0,1)$} & $\mu$  & $0.118$ & $0.117$ & $5.1\text{e-3}$ & $ \hl 3.95\text{\hl e-3}$ & $0.091$ & $0.206$ & $0.296$ & $0.134$ & $ \hl 0.113$ & $0.161$ \\ \hline 
\multirow{1}{*}{$(0,3)$} & $\mu$  & $0.343$ & $0.394$ & $5.14\text{e-3}$ & $ \hl 3.29\text{\hl e-3}$ & $0.273$ & $1.87$ & $4.81$ & $1.22$ & $ \hl 1.02$ & $1.42$ \\ \hline 
\multirow{1}{*}{$(2,0.1)$} & $\mu$  & $0.345$ & $0.013$ & $3.05\text{e-4}$ & $ \hl 2.73\text{\hl e-4}$ & $9.01\text{e-3}$ & $0.121$ & $5.89\text{e-3}$ & $1.36\text{e-3}$ & $ \hl 1.15\text{\hl e-3}$ & $1.59\text{e-3}$ \\ \hline 
\multirow{1}{*}{$(2,0.5)$} & $\mu$  & $0.388$ & $0.065$ & $-1.29\text{e-3}$ & $ \hl -1.22\text{\hl e-3}$ & $0.043$ & $0.2$ & $0.24$ & $0.034$ & $ \hl 0.028$ & $0.039$ \\ \hline 
\multirow{1}{*}{$(4,0.8)$} & $\mu$  & $0.751$ & $0.092$ & $-4.79\text{e-3}$ & $ \hl -4.02\text{\hl e-3}$ & $0.066$ & $0.69$ & $0.325$ & $0.087$ & $ \hl 0.074$ & $0.101$ \\ \hline 
\multirow{1}{*}{$(6,2.3)$} & $\mu$  & $1.26$ & $0.29$ & $7.79\text{e-3}$ & $ \hl 3.64\text{\hl e-3}$ & $0.217$ & $2.59$ & $2.02$ & $0.693$ & $ \hl 0.59$ & $2.92$ \\ \hline 
\multirow{1}{*}{$(10,0.2)$} & $\mu$  & $1.69$ & $0.026$ & $5.34\text{e-4}$ & $ \hl 3.58\text{\hl e-4}$ & $0.018$ & $2.86$ & $0.03$ & $5.62\text{e-3}$ & $ \hl 4.49\text{\hl e-3}$ & $6.78\text{e-3}$ \\ \hline 
\end{tabular} 
\caption{\protect\label{cauchy_sim_n20} Simulation results for the $C(\mu,\gamma)$ distribution for $n=20$ and $10,000$ repetitions.}
\end{table}
\endgroup

\begingroup
\setlength\tabcolsep{3.5pt}
\begin{table}[h] \small
\centering
\begin{tabular}{cc|ccccc|ccccc}
 $\theta_0$ & & \multicolumn{5}{|c}{Bias} & \multicolumn{5}{|c}{MSE} \\ \hline
 & & $\hat{\mu}_n^{\mathrm{L1}}$ & $\hat{\mu}_n^{\mathrm{L2}}$  & $\hat{\mu}_n^{\mathrm{L4}}$ & $\hat{\mu}_n^{\mathrm{PI}}$ & $\hat{\mu}_n^{\mathrm{ST1}}$ & $\hat{\mu}_n^{\mathrm{L1}}$ & $\hat{\mu}_n^{\mathrm{L2}}$  & $\hat{\mu}_n^{\mathrm{L4}}$ & $\hat{\mu}_n^{\mathrm{PI}}$ & $\hat{\mu}_n^{\mathrm{ST1}}$ \\ \hline
\multirow{1}{*}{$(-5,1)$} & $\mu$  & $-0.378$ & 0.012 & $ \hl 1.63\text{\hl e-3}$ & $3.7\text{e-3}$ & $1.71\text{e-3}$ & 0.201 & 0.048 & 0.05 & 0.082 & \hl 0.044\\ \hline 
\multirow{1}{*}{$(-4,1.5)$} & $\mu$  & $-0.276$ & 0.015 & $1.38\text{e-3}$ & $1.5\text{e-3}$ & $ \hl 1.17\text{\hl e-3}$ & 0.206 & 0.105 & 0.11 & \hl 0.099 & \hl 0.099\\ \hline 
\multirow{1}{*}{$(-2,2)$} & $\mu$  & $-0.092$ & 0.025 & $ \hl -7.46\text{\hl e-4}$ & $7.45\text{e-3}$ & $1.33\text{e-3}$ & 0.248 & 0.195 & 0.202 & 0.554 & \hl 0.182\\ \hline 
\multirow{1}{*}{$(0,1)$} & $\mu$  & 0.034 & $7.29\text{e-3}$ & $ \hl -2.69\text{\hl e-3}$ & $3.68\text{e-3}$ & $-3.39\text{e-3}$ & 0.058 & 0.047 & 0.048 & 0.517 & \hl 0.044\\ \hline 
\multirow{1}{*}{$(0,3)$} & $\mu$  & 0.094 & 0.024 & $-0.016$ & $ \hl -9.89\text{\hl e-3}$ & $-0.01$ & 0.527 & 0.42 & 0.437 & \hl 0.384 & 0.39\\ \hline 
\multirow{1}{*}{$(2,0.1)$} & $\mu$  & 0.171 & $1.29\text{e-3}$ & $ \hl 1.9\text{\hl e-4}$ & $4.96\text{e-4}$ & $1.93\text{e-4}$ & 0.03 & $4.79\text{e-4}$ & $4.99\text{e-4}$ & $1.08\text{e-3}$ & $ \hl 4.46\text{\hl e-4}$\\ \hline 
\multirow{1}{*}{$(2,0.5)$} & $\mu$  & 0.186 & $6.37\text{e-3}$ & $ \hl 2.87\text{\hl e-4}$ & $1.14\text{e-3}$ & $1.12\text{e-3}$ & 0.049 & 0.012 & 0.012 & 0.019 & \hl 0.011\\ \hline 
\multirow{1}{*}{$(4,0.8)$} & $\mu$  & 0.364 & $8.99\text{e-3}$ & $4.44\text{e-4}$ & $-2.36\text{e-3}$ & $ \hl -2.75\text{\hl e-4}$ & 0.169 & 0.03 & 0.031 & 0.037 & \hl 0.028\\ \hline 
\multirow{1}{*}{$(6,2.3)$} & $\mu$  & 0.582 & 0.024 & $-3.22\text{e-3}$ & $ \hl -9.5\text{\hl e-4}$ & $-2.31\text{e-3}$ & 0.639 & 0.243 & 0.253 & 0.24 & \hl 0.226\\ \hline 
\multirow{1}{*}{$(10,0.2)$} & $\mu$  & 0.84 & $1.84\text{e-3}$ & $ \hl -4.34\text{\hl e-4}$ & $-2.94\text{e-3}$ & $-2.73\text{e-4}$ & 0.708 & $1.83\text{e-3}$ & $1.87\text{e-3}$ & 0.154 & $ \hl 1.7\text{\hl e-3}$\\ \hline 
\end{tabular} 
\caption{\protect\label{cauchy_sim_n50} Simulation results for the $C(\mu,\gamma)$ distribution for $n=50$ and $10,000$ repetitions.}
\end{table}
\endgroup

\subsection{Exponential polynomial models}
For $\theta=(\theta_1,\ldots,\theta_p) \in \mathbb{R}^{p-1} \times (-\infty,0)$, the density of an exponential polynomial model is given by
$
    p_{\theta}(x)=C_\theta^{-1} \exp(\theta_1 x+ \ldots + \theta_p x^p)$, $x>0$,   
where $C_\theta=\int_0^{\infty} \exp(\theta_1 x+ \ldots + \theta_p x^p)\,dx $ is the normalising constant which cannot be calculated analytically. We choose $\tau_{\theta}(x)=1$ and obtain the Stein operator
\begin{align*}
    \mathcal{A}_{\theta}f(x)=(\theta_1 + 2\theta_2x + \ldots + p\theta_px^{p-1})f(x)+f'(x).
\end{align*}
Here, we need $p$ test functions $f_1,\ldots,f_p$ and the Stein estimator is then given by
\begin{align*}
    \hat{\theta}_n = A^{-1}b,
\end{align*}
where $A$ is a $p\times p$ matrix with $(i,j)$-th entry $j\overline{X^{j-1}f_i(X)} $ and $b=(-\overline{f_1'(X)},\ldots,-\overline{f_p'(X)})^\intercal$.
We propose the test functions $f_1(x)=x,\ldots,f_p(x)=x^p$, and denote the corresponding estimator by $\hat{\theta}_n^{\mathrm{ST1}}$. We also consider $f_i(x)=x^ie^{-ix}$, for $i=0,\ldots,p-1$, and call the respective Stein estimator $\hat{\theta}_n^{\mathrm{ST2}}$. Further, we study the two-step Stein estimator, which we denote by $\hat{\theta}_n^{\mathrm{ST3}}$, whereby we take $\hat{\theta}_n^{\mathrm{ST2}}$ as a first-step estimate. This estimator is consistent and asymptotically efficient.

Let us walk through the estimation methods in the literature. \cite{hayakawa2016estimation} and \cite{nakayama2011holonomic} used the holomorphic gradient method in order to compute the MLE. For exponential polynomial models, the MLE coincides with the moment estimator. \cite{gutmann2012noise} (a refined version of one from \cite{gutmann2010noise}) proposed the noise-contrastive estimator, which we denote by $\hat{\theta}_n^{\mathrm{NC}}$. We consider the score matching approach from \cite{hyvarinen2007some} (a refined version of \cite{hyvarinen2005estimation}) and denote the score matching estimator by $\hat{\theta}_n^{\mathrm{SM}}$. The estimators $\hat{\theta}_n^{\mathrm{NC}}$ and $\hat{\theta}_n^{\mathrm{SM}}$ are computed via numerical optimisation, and the procedure for implementing them is given in the Supplementary Information. It is also natural to consider the  the minimum $\mathscr{L}^q$-estimator obtained from \cite{betsch2021minimum}, which is also motivated by a Stein characterisation; more precisely, by an expectation-based representation of the CDF. The minimum distance estimator is only explicit for a parameter space of dimension less than or equal to 2, and we thus exclude this estimator from our simulation study, since the numerical calculation turns out to be too heavy for a parameter space dimension of $3$ or higher.

In our simulation study, for the MLE, $C_\theta$ is calculated through numerical integration and optimising the log-likelihood function is performed with the Nelder-Mead algorithm. The vector $(-1,\ldots,-1)^\top \in \mathbb{R}^p$ is used as an initial guess for the optimisation procedure. This implementation seems, at least for the parameter constellations we consider, to be computationally manageable and numerically stable. For the noise-contrastive estimator $\hat{\theta}_n^{\mathrm{NC}}$ and the score matching approach $\hat{\theta}_n^{\mathrm{SM}}$, we also used the Nelder-Mead algorithm with initial guess $(-1,\ldots,-1)^\top \in \mathbb{R}^p$. For the two-step Stein estimator $\hat{\theta}_n^{\mathrm{ST3}}$, the normalising constant $C_\theta$ needs to be calculated in order to evaluate the optimal function. This is done through numerical differentiation. 

The results are reported in Tables \ref{expomodels_sim_bias} and \ref{expomodels_sim_mse}. The column \textit{NE} is interpreted as follows. First, the last element of the parameter vector $\theta_p$ has to be negative. Thus, if any estimator returns a positive value for this parameter, we count the estimator as non-existent. Secondly, we restrict the computation time for each estimator to $20$ seconds meaning that an estimator counts equally as non-existent if it requires more time to be calculated or if the numerical procedure fails completely. Concerning $\hat{\theta}_n^{\mathrm{ST3}}$, we also used the parameter vector $(-1,\ldots,-1)^\top \in \mathbb{R}^p$ as a first-step estimate if $\hat{\theta}_n^{\mathrm{ST2}}$ was not available for a Monte Carlo sample. The sample size for this simulation was chosen to be larger than for the other simulation studies, since we are concerned with an estimation problem in which the variance of the estimator becomes typically large as the dimension of the parameter space grows. This makes it difficult to compare estimators for small sample sizes in the case of parameter dimensions of $3$ and $4$. Additionally, the Stein estimators $\hat{\theta}_n^{\mathrm{ST1}}$ and $\hat{\theta}_n^{\mathrm{ST2}}$ often return positive values for $\theta_p$, which makes a comparison even more difficult since the number of samples on which the bias and MSE are based is in truth lower than the number of Monte Carlo repetitions. However, for rather small sample sizes of $n=20$ or $n=50$, we found that our simulation results are reliable for a parameter space dimension of $2$ with similar results as described below, which is why we did not include a separate table for these results. Therefore, we chose the sample size $n=1000$, where we feel comfortable in drawing conclusions out of the study. Overall, we observe a solid performance of the Stein estimators. For example, the explicit Stein estimators $\hat{\theta}_n^{\mathrm{ST1}}$ and $\hat{\theta}_n^{\mathrm{ST2}}$ outperform all other methods for the parameter vector $(-2,0.1,3,-2)^\top$ in terms of bias and MSE. The two-step Stein estimator $\hat{\theta}_n^{\mathrm{ST3}}$ together with $\hat{\theta}_n^{\mathrm{ML}}$ and $\hat{\theta}_n^{\mathrm{NC}}$ seem to be globally the best. Moreover, one observes that $\hat{\theta}_n^{\mathrm{ST3}}$ can often improve in terms of bias and MSE with respect to the first-step estimator $\hat{\theta}_n^{\mathrm{ST2}}$ (although there are some exceptions). In the end, we advise to use $\hat{\theta}_n^{\mathrm{ST3}}$, the MLE or the noise-contrastive estimator $\hat{\theta}_n^{\mathrm{NC}}$, while the explicit Stein estimators can serve as a reliable initial guess, if they exist.

\begin{table}[h] \small
\centering
\begin{tabular}{cc|cccccc}
 $\theta_0$ &  & \multicolumn{6}{|c}{Bias} \\ \hline
 & & $\hat{\theta}_n^{\mathrm{ML}}$ & $\hat{\theta}_n^{\mathrm{NC}}$ & $\hat{\theta}_n^{\mathrm{SM}}$ & $\hat{\theta}_n^{\mathrm{ST1}}$ & $\hat{\theta}_n^{\mathrm{ST2}}$ & $\hat{\theta}_n^{\mathrm{ST3}}$  \\ \hline
\multirow{2}{*}{$(1,-2)$} & $\theta_1$  & $0.02$ & $0.019$ & $0.085$ & $0.03$ & $0.03$ & $\hl 0.017$ \\ & $\theta_2$ & $-0.017$ & $-0.017$ & $-0.055$ & $-0.024$ & $-0.024$ & $\hl -0.015$ \\ \hline 
\multirow{2}{*}{$(-2,-1)$} & $\theta_1$  & $0.025$ & $0.026$ & $0.13$ & $0.035$ & $0.035$ & $ \hl 0.023$ \\ & $\theta_2$ & $-0.031$ & $-0.032$ & $-0.106$ & $-0.04$ & $-0.04$ & $\hl -0.029$ \\ \hline 
\multirow{3}{*}{$(1,2,-3)$} & $\theta_1$  & $-0.025$ & $-0.027$ & $-0.477$ & $-0.122$ & $-0.122$ & $ \hl -0.011$ \\ & $\theta_2$ & $0.073$ & $0.079$ & $0.737$ & $0.244$ & $0.244$ & $ \hl 0.042$ \\ & $\theta_{3}$ & $-0.051$ & $-0.054$ & $-0.34$ & $-0.134$ & $-0.134$ & $ \hl -0.033$ \\ \hline 
\multirow{3}{*}{$(-3,5,-1)$} & $\theta_1$  & $2.6$ & $2.59$ & $-4.83$ & $ \hl -0.078$ & $ \hl -0.078$ & $2.82$ \\ & $\theta_2$ & $-0.92$ & $-0.923$ & $1.62$ & $ \hl 0.048$ & $ \hl 0.048$ & $-0.986$ \\ & $\theta_{3}$ & $0.107$ & $0.107$ & $-0.18$ & $ \hl -7.71\text{\hl e-3}$ & $ \hl -7.71\text{\hl e-3}$ & $0.113$ \\ \hline 
\multirow{3}{*}{$(0.2,-0.8,-2)$} & $\theta_1$  & $ \hl -0.048$ & $-0.054$ & $-0.923$ & $-0.208$ & $-0.189$ & $-0.093$ \\ & $\theta_2$ & $ \hl 0.156$ & $0.176$ & $1.83$ & $0.532$ & $0.491$ & $0.252$ \\ & $\theta_{3}$ & $ \hl -0.127$ & $-0.142$ & $-1.04$ & $-0.359$ & $-0.336$ &  $-0.183$ \\ \hline 
\multirow{3}{*}{$(3,0.5,-0.5)$} & $\theta_1$  & $ \hl -0.018$ & $0.019$ & $-0.332$ & $-0.101$ & $-0.101$ & $0.052$ \\ & $\theta_2$ & $0.026$ & $ \hl -5.63\text{\hl e-4}$ & $0.205$ & $0.075$ & $0.075$ & $-0.026$ \\ & $\theta_{3}$ & $-7.56\text{e-3}$ & $ \hl -1.87\text{\hl e-3}$ & $-0.04$ & $-0.016$ & $-0.016$ & $3.63\text{e-3}$ \\ \hline 
\multirow{3}{*}{$(0.1,2,-3)$} & $\theta_1$  & $-0.036$ & $-0.04$ & $-0.544$ & $-0.136$ & $-0.136$ & $ \hl -0.022$ \\ & $\theta_2$ & $0.102$ & $0.112$ & $0.909$ & $0.296$ & $0.295$ & $ \hl 0.067$ \\ & $\theta_{3}$ & $-0.07$ & $-0.077$ & $-0.445$ & $-0.172$ & $-0.171$ & $ \hl -0.049$ \\ \hline 
\multirow{3}{*}{$(3,0,-4)$} & $\theta_1$  & $-0.018$ & $-0.021$ & $-0.745$ & $-0.164$ & $-0.16$ & $ \hl -5.05\text{\hl e-3}$ \\ & $\theta_2$ & $0.088$ & $0.093$ & $1.38$ & $0.391$ & $0.382$ & $ \hl 0.049$ \\ & $\theta_{3}$ & $-0.082$ & $-0.085$ & $-0.763$ & $-0.258$ & $-0.253$ & $ \hl -0.054$ \\ \hline 
\multirow{4}{*}{$(1,2,0.5,-2)$} & $\theta_1$  & $-0.881$ & $-0.94$ & $3.42$ & $1.16$ & $0.888$ & $ \hl 0.58$ \\ & $\theta_2$ & $2.19$ & $2.34$ & $-7.61$ & $-3.14$ & $-2.43$ & $ \hl -1.65$ \\ & $\theta_{3}$ & $-2.01$ & $-2.16$ & $6.73$ & $3.15$ & $2.48$ & $ \hl 1.71$ \\ & $\theta_{4}$ & $0.614$ & $0.66$ & $-2.06$ & $-1.06$ & $-0.842$ & $ \hl -0.591$ \\ \hline 
\multirow{4}{*}{$(-2,0.1,3,-2)$} & $\theta_1$  & $1.2$ & $1.57$ & $1.9$ & $0.382$ & $0.344$ & $ \hl 0.139$ \\ & $\theta_2$ & $-3.21$ & $-4.2$ & $-4.51$ & $-1.19$ & $-1.08$ & $ \hl -0.43$ \\ & $\theta_{3}$ & $3.04$ & $3.97$ & $4.01$ & $1.26$ & $1.14$ & $ \hl 0.458$ \\ & $\theta_{4}$ & $-0.939$ & $-1.22$ & $-1.19$ & $-0.421$ & $-0.386$ & $ \hl -0.159$ \\ \hline 
\end{tabular}
\caption{\protect\label{expomodels_sim_bias} Simulation results regarding the bias for the exponential polynomial models for $n=1000$ and $10,000$ repetitions.}
\end{table}

\vspace{1cm}

\begingroup
\setlength\tabcolsep{3.5pt}
\begin{table}[h] \small
\centering
\begin{tabular}{cc|cccccc|cccccc}
 $\theta_0$ &  & \multicolumn{6}{|c}{MSE} & \multicolumn{6}{|c}{NE} \\ \hline
 & & $\hat{\theta}_n^{\mathrm{ML}}$ & $\hat{\theta}_n^{\mathrm{NC}}$ & $\hat{\theta}_n^{\mathrm{SM}}$ & $\hat{\theta}_n^{\mathrm{ST1}}$ & $\hat{\theta}_n^{\mathrm{ST2}}$ & $\hat{\theta}_n^{\mathrm{ST3}}$  & $\hat{\theta}_n^{\mathrm{ML}}$ & $\hat{\theta}_n^{\mathrm{NC}}$ & $\hat{\theta}_n^{\mathrm{SM}}$ & $\hat{\theta}_n^{\mathrm{ST1}}$ & $\hat{\theta}_n^{\mathrm{ST2}}$ & $\hat{\theta}_n^{\mathrm{ST3}}$  \\ \hline
\multirow{2}{*}{$(1,-2)$} & $\theta_1$  & $ \hl 0.073$ & $0.079$ & $0.28$ & $0.09$ & $0.09$ & $ \hl 0.073$ & \multirow{2}{*}{$0$} & \multirow{2}{*}{$0$} & \multirow{2}{*}{$0$} & \multirow{2}{*}{$0$} & \multirow{2}{*}{$0$} & \multirow{2}{*}{$0$} \\ & $\theta_2$ & $ \hl 0.039$ & $0.042$ & $0.109$ & $0.047$ & $0.047$ & $ \hl 0.039$ \\ \hline 
\multirow{2}{*}{$(-2,-1)$} & $\theta_1$  & $ \hl 0.084$ & $0.092$ & $0.289$ & $0.09$ & $0.09$ & $ \hl 0.084$ & \multirow{2}{*}{$0$} & \multirow{2}{*}{$0$} & \multirow{2}{*}{$0$} & \multirow{2}{*}{$0$} & \multirow{2}{*}{$0$} & \multirow{2}{*}{$0$} \\ & $\theta_2$ & $0.069$ & $0.074$ & $0.171$ & $0.074$ & $0.074$ & $ \hl 0.068$ \\ \hline 
\multirow{3}{*}{$(1,2,-3)$} & $\theta_1$  & $ \hl 0.687$ & $0.743$ & $4.83$ & $1.15$ & $1.15$ & $ \hl 0.687$ & \multirow{3}{*}{$0$} & \multirow{3}{*}{$0$} & \multirow{3}{*}{$0$} & \multirow{3}{*}{$0$} & \multirow{3}{*}{$0$} & \multirow{3}{*}{$0$} \\ & $\theta_2$ & $ \hl 2.08$ & $2.26$ & $10.7$ & $3.4$ & $3.4$ & $ \hl 2.08$ \\ & $\theta_{3}$ & $0.533$ & $0.576$ & $2.12$ & $0.834$ & $0.834$ & $ \hl 0.532$ \\ \hline 
\multirow{3}{*}{$(-3,5,-1)$} & $\theta_1$  & $6.74$ & $ \hl 6.73$ & $44.8$ & $39.6$ & $39.6$ & $29.4$ & \multirow{3}{*}{$0$} & \multirow{3}{*}{$0$} & \multirow{3}{*}{$0$} & \multirow{3}{*}{$0$} & \multirow{3}{*}{$0$} & \multirow{3}{*}{$0$} \\ & $\theta_2$ & $ \hl 0.873$ & $0.882$ & $5.29$ & $4.99$ & $4.99$ & $3.7$ \\ & $\theta_{3}$ & $ \hl 0.013$ & $ \hl 0.013$ & $0.069$ & $0.068$ & $0.068$ & $0.051$ \\ \hline 
\multirow{3}{*}{$(0.2,-0.8,-2)$} & $\theta_1$  & $0.675$ & $0.725$ & $3.8$ & $0.839$ & $0.871$ & $ \hl 0.62$ & \multirow{3}{*}{$0$} & \multirow{3}{*}{$0$} & \multirow{3}{*}{$9$} & \multirow{3}{*}{$3$} & \multirow{3}{*}{$3$} & \multirow{3}{*}{$1$} \\ & $\theta_2$ & $3.46$ & $3.72$ & $14.8$ & $4.43$ & $4.57$ & $ \hl 3.16$ \\ & $\theta_{3}$ & $1.42$ & $1.51$ & $4.85$ & $1.82$ & $1.86$ & $ \hl 1.31$ \\ \hline 
\multirow{3}{*}{$(3,0.5,-0.5)$} & $\theta_1$  & $0.873$ & $0.769$ & $3.82$ & $1.38$ & $1.38$ & $ \hl 0.67$ & \multirow{3}{*}{$0$} & \multirow{3}{*}{$0$} & \multirow{3}{*}{$0$} & \multirow{3}{*}{$0$} & \multirow{3}{*}{$0$} & \multirow{3}{*}{$0$} \\ & $\theta_2$ & $0.402$ & $0.345$ & $1.35$ & $0.565$ & $0.565$ & $ \hl 0.292$ \\ & $\theta_{3}$ & $0.018$ & $0.015$ & $0.047$ & $0.023$ & $0.023$ & $ \hl 0.013$ \\ \hline 
\multirow{3}{*}{$(0.1,2,-3)$} & $\theta_1$  & $ \hl 0.636$ & $0.681$ & $4.23$ & $0.967$ & $0.968$ & $ \hl 0.636$ & \multirow{3}{*}{$0$} & \multirow{3}{*}{$0$} & \multirow{3}{*}{$1$} & \multirow{3}{*}{$0$} & \multirow{3}{*}{$0$} & \multirow{3}{*}{$0$} \\ & $\theta_2$ & $ \hl 2.31$ & $2.48$ & $11.1$ & $3.48$ & $3.49$ & $ \hl 2.31$ \\ & $\theta_{3}$ & $0.684$ & $0.73$ & $2.53$ & $0.996$ & $0.997$ & $ \hl 0.681$ \\ \hline 
\multirow{3}{*}{$(3,0,-4)$} & $\theta_1$  & $ \hl 1.08$ & $1.17$ & $6.92$ & $1.77$ & $1.78$ & $ \hl 1.08$ & \multirow{3}{*}{$0$} & \multirow{3}{*}{$0$} & \multirow{3}{*}{$2$} & \multirow{3}{*}{$0$} & \multirow{3}{*}{$0$} & \multirow{3}{*}{$0$} \\ & $\theta_2$ & $4.75$ & $5.18$ & $22.4$ & $7.55$ & $7.6$ & $ \hl 4.71$ \\ & $\theta_{3}$ & $1.78$ & $1.94$ & $6.59$ & $2.71$ & $2.73$ & $ \hl 1.77$ \\ \hline 
\multirow{4}{*}{$(1,2,0.5,-2)$} & $\theta_1$  & $1.34$ & $ \hl 1.27$ & $33.2$ & $6.71$ & $7.23$ & $3.31$ & \multirow{4}{*}{$0$} & \multirow{4}{*}{$0$} & \multirow{4}{*}{$16$} & \multirow{4}{*}{$13$} & \multirow{4}{*}{$9$} & \multirow{4}{*}{$1$} \\ & $\theta_2$ & $9.3$ & $ \hl 8.96$ & $161$ & $45.2$ & $48.7$ & $21.6$ \\ & $\theta_{3}$ & $10.2$ & $ \hl 10$ & $128$ & $44.2$ & $47.2$ & $21.6$ \\ & $\theta_{4}$ & $1.29$ & $ \hl 1.28$ & $12.3$ & $4.93$ & $5.22$ & $2.51$ \\ \hline 
\multirow{4}{*}{$(-2,0.1,3,-2)$} & $\theta_1$  & $2.61$ & $3.2$ & $13.8$ & $ \hl 1.94$ & $2.02$ & $2.54$ & \multirow{4}{*}{$0$} & \multirow{4}{*}{$0$} & \multirow{4}{*}{$9$} & \multirow{4}{*}{$2$} & \multirow{4}{*}{$1$} & \multirow{4}{*}{$0$} \\ & $\theta_2$ & $19.9$ & $24.2$ & $78$ & $ \hl 16.4$ & $17.1$ & $21.4$ \\ & $\theta_{3}$ & $18.9$ & $22.9$ & $62$ & $ \hl 17$ & $17.7$ & $22$ \\ & $\theta_{4}$ & $1.9$ & $2.31$ & $5.54$ & $ \hl 1.83$ & $1.9$ & $2.34$ \\ \hline 
\end{tabular} 
\caption{\protect\label{expomodels_sim_mse} Simulation results regarding the MSE and existence for the exponential polynomial models for $n=1000$ and $10,000$ repetitions.}
\end{table}
\endgroup


\subsection{Nakagami distribution for censored data: a real data application}\label{section_discussionnaka} 

The Nakagami distribution, also known as the $m$-distribution  \cite{nakagami1960m}, is a continuous probability distribution on the positive real numbers. Its probability  density function is given by  
\begin{align*}
p_{\theta}(x)=\frac{2m^m}{\Gamma(m)O^m} x^{2m-1} \exp\bigg(-\frac{m}{O} x^2 \bigg), \quad x >0
\end{align*}
where the parameter  $\theta = (m, O)$ has strictly positive  components:  $m$ (the shape parameter)  and $O$ (a scale parameter). This distribution, which is  the distribution of  the square root of a gamma variable,  is designed to  model phenomena characterized by fading and variability and its applications span various fields, including wireless communications, hydrology, mining, medical imaging; see, for example,  \cite{baugci11nakagami, kolar2004estimator, kumar2024q, miyoshi2015downlink, reyes2020nakagami, tegos2021distribution}, among many others.     

 Although this is a regular family with straightforward theoretical properties (and thus MLE is the best all-round estimator), the problem of estimating the parameters of a Nakagami distribution has attracted much attention because of the importance of this distribution for modeling purposes.  
The MLE of the scale $O$ is 
\begin{equation}
    \label{NakaOmegahat}
     \hat{O}_n^{\mathrm{ML}} = \overline{X^2},
\end{equation}
and this cannot be improved upon. The MLE  of  the shape $m$ requires solving the likelihood equations
\begin{equation}
    \label{NakaMLmhat}
     \log  \hat{m}_n^{\mathrm{ML}} - \psi( \hat{m}_n^{\mathrm{ML}}) - \log \hat{O}_n^{\mathrm{ML}}  + 2  \overline{\log X} = 0
\end{equation}
with $\psi(\cdot)$ the digamma function.  See e.g.\ \cite{kolar2004estimator} for a comparative study of various algorithms computing roots of digamma functions numerically;   see \cite{schwartz2013improved} for a bias-corrected version. Starting points for the optimization are given by  the MOM estimators; the MOM estimate for $O$ is the same as  above  but for $m$ can only be computed through a numerical solver and  has high asymptotic variance.  In \cite{artyushenko2019nakagami} the  modified   MOM estimator 
 \begin{equation}
    \label{nakagami_modified_moment_estimators22}
\hat{m}_n^{\mathrm{MO2}}= \frac{{(\overline{X^2}})^2 }{\overline{X^4} - {(\overline{X^2})^2}}, \quad  \hat O_n^{\mathrm{MO2}} = \overline{X^2} 
\end{equation}
is proposed and it is argued through simulations (the  asymptotic properties are not studied) that this  estimator is a more suitable first-step estimator.  More recently  \cite{zhao2021closed} use a generalized Nakagami distribution to derive  another  closed-form moment type estimator 
\begin{align} \label{nakagami_modified_moment_estimators2}
\hat{m}_n^{\mathrm{MO3}}= \frac12 \frac{{\overline{X^2}} }{\overline{X^2 \log(X) } - {\overline{X^2} \, \overline{\log X}}}, \quad  \hat \hat{O}_n^{\mathrm{MO3}} = \overline{X^2}
\end{align}
whose  asymptotic variance is obtained and is showed  to be  very close to that of the MLE.

Setting up our SMOM estimators for Nakagami distributions is simple. A Stein operator is known for the Nakagami  distribution (and is also easy to obtain e.g.\ through the density approach) and  is given by 
\begin{align*}
\mathcal{A}_{\theta}f(x)=2m(O - x^2)f(x) +x O f'(x).
\end{align*}
As the approach  conceals no surprise, we postpone the details of the computations to the Supplementary Material  (Section \ref{subsection_nakagami_distribution}).    In particular one immediately sees that both previous  modified MOM estimators  fall directly within  our general SMOM estimation procedure,  with  \eqref{nakagami_modified_moment_estimators22} obtained through  $f_1(x) = 1$ and $f_2(x) = x$, and   \eqref{nakagami_modified_moment_estimators2} through  $f_1 (x) = 1$ and $f_2(x)  = \log (x)$. Our Theorem \ref{theorem_asymptotic_normality} immediately yields the asymptotic variance of these estimators, confirming   \cite{zhao2021closed} in that particular case.  All expected behaviors are illustrated in our simulation study detailed in Section \ref{subsection_nakagami_distribution} from the Supplementary Material.

Despite their excellent performances, the above $\hat{m}_n^{\mathrm{ML}}$ and $\hat{m}_n^{\mathrm{MO2}}$ estimators  are nevertheless plagued by numerical instability whenever the data contains 0's, as can happen for instance when the data is rounded to the first decimal. This is a classical issue, see e.g.\  \cite{wilks1990maximum}, and such data does occur in many real-life scenarios, as we shall illustrate below. Here the  flexibility of our  Stein-MOM estimators can be exploited to design explicit estimators with low asymptotic variance which are not sensitive to the presence of 0's in the data. Some explorations lead us to propose the compromise  $f_1(x) = 1$ and $f_2(x) = x$ which yields the new estimator
\begin{align} \label{nakagami_stein_estimators}
\hat{m}_n^{\mathrm{ST}}= \frac{1}{2} \frac{\overline{X^2} \ \overline{X}}{\overline{X^3} - \overline{X} \ \overline{X^2}}, \quad \hat{O}_n^{\mathrm{ST}} = \overline{X^2}
\end{align} 
which is explicit, immediate to compute, does not  suffer from the  numerical instability of \eqref{nakagami_modified_moment_estimators2} and has a better variance (since it involves lower moments) than \eqref{nakagami_modified_moment_estimators22}.
 We illustrate the various behaviours in Table~\ref{tab:NA} where all the estimators are applied to the same samples rounded to the first decimal over a variety of problematic parameter ranges (the MLE is started at initial estimates provided by the new estimator $\hat{m}_n^{\mathrm{ST}}$). When the simulation is performed over a small sample e.g. $n=50$, the MLE and $\hat{m}_n^{\mathrm{MO3}}$ estimators have overall good properties ; increasing the  sample size increases the volume of zeros and MLE and $\hat{m}_n^{\mathrm{MO3}}$  incur  heightened bias because of this, hence requiring more sophisticated approaches for these estimators. $\hat{m}_n^{\mathrm{ST}}$ remains good irrespective of the parameter values, range, or sample size.

\begingroup
\setlength\tabcolsep{3.5pt}
\begin{table}[h]
\centering \small
\begin{tabular}{cc|c|cccc|cccc}
 $\theta_0$ & & Sample Size & \multicolumn{4}{c}{Bias} & \multicolumn{4}{|c}{MSE} \\ \hline
 & & $n$ & $\hat{m}_n^{\mathrm{ML}}$ & $\hat{m}_n^{\mathrm{MO2}}$ & $\hat{m}_n^{\mathrm{MO3}}$ & $\hat{m}_n^{\mathrm{ST}}$ & $\hat{m}_n^{\mathrm{ML}}$ & $\hat{m}_n^{\mathrm{MO2}}$ & $\hat{m}_n^{\mathrm{MO3}}$ & $\hat{m}_n^{\mathrm{ST}}$ \\ \hline
 $(0.7,1)$ & $m$ & $50$ & $0.095$ & $0.095$ & $0.088$ & $\hl{0.059}$ & $\hl{0.025}$ & $0.055$ & $0.026$ & $0.032$ \\ 
 & & $500$ & $NaN$ & $0.005$ & $NaN$ & $\hl{0.002}$ & $NaN$ & $0.002$ & $NaN$ & $\hl{0.001}$ \\ \hline
  $(1,0.5)$ & $m$ & $50$ & $0.084$ & $0.109$ & $0.081$ & $\hl{0.07}$ & $\hl{0.044}$ & $0.093$ & $0.047$ & $0.061$ \\ 
 & & $500$ & $0.039$ & $0.002$ & $0.031$ & $\hl{-0.001}$ & $0.003$ & $0.004$ & $0.003$ & $\hl{0.002}$ \\ \hline
 $(1,1)$ & $m$ & $50$ & $0.067$ & $0.108$ & $\hl{0.066}$ & $0.071$ & $\hl{0.041}$ & $0.09$ & $0.043$ & $0.059$ \\ 
 & & $500$ & $0.019$ & $0.005$ & $0.014$ & $\hl{0.001}$ & $0.002$ & $0.004$ & $\hl{0.002}$ & $0.002$ \\ \hline
 $(1,5)$ & $m$ & $50$ & $\hl{0.056}$ & $0.112$ & $0.058$ & $0.074$ & $\hl{0.04}$ & $0.091$ & $0.043$ & $0.06$ \\ 
 & & $500$ & $0.007$ & $0.006$ & $0.006$ & $\hl{0.003}$ & $\hl{0.002}$ & $0.004$ & $0.002$ & $0.002$ \\ \hline
 $(1,10)$ & $m$ & $50$ & $\hl{0.053}$ & $0.111$ & $0.055$ & $0.073$ & $\hl{0.039}$ & $0.092$ & $0.042$ & $0.06$ \\ 
 & & $500$ & $0.005$ & $0.006$ & $0.004$ & $\hl{0.004}$ & $\hl{0.002}$ & $0.004$ & $0.002$ & $0.003$ \\ \hline
 $(5,5)$ & $m$ & $50$ & $\hl{0.265}$ & $0.329$ & $0.268$ & $0.287$ & $\hl{1.233}$ & $1.505$ & $1.247$ & $1.337$ \\ 
 & & $500$ & $-0.004$ & $\hl{0.001}$ & $-0.003$ & $-0.002$ & $\hl{0.047}$ & $0.06$ & $0.047$ & $0.052$ \\ \hline
\end{tabular} 
\caption{\label{tab:NA} Bias and MSE over 10,000 simulations of samples of size $n=50$ and $n=500$, respectively, with the prescribed parameters; each time the samples are rounded to the first decimal.}
\end{table}
\endgroup

To illustrate our approach on a real data set, consider  rainfall data   from 1961-01-01 to 2017-12-31 available from  the national river flow  archive\footnote{https://nrfa.ceh.ac.uk/data/}. It is well-documented that Nakagami distributions provides an excellent fit for such  data sets \cite{baugci11nakagami}. Considering monthly rainfall data (i.e. summing the values over each month throughout the timespan), all parameter estimation methods lead to similar values. The resulting estimates are used to produce the plot in Figure~\ref{fig:histogram-last-ten-years}. If, instead  we extend  to weekly data over the full period, the presence of 0's makes $\hat{m}_n^{\mathrm{MO3}}$ and MLE methods break down, and only $\hat{m}_n^{\mathrm{MO2}}$ and $\hat{m}_n^{\mathrm{ST}}$ provide reasonable estimates. The corresponding histogram and fitted densities are reported in Figure \ref{fig:histogram-last-five-years}. Goodness-of-fit tests demonstrate an excellent fit of the Nakagami distribution to the monthly data, but the fit is less convincing  for the weekly rainfall data. Alternative distributions may be considered, including compound gamma distributions as studied recently in \cite{ebner2024independent} (again  via  a version of Stein's method).

\begin{figure}[ht!]
    \centering
     \begin{subfigure}[b]{0.48\textwidth}
        \centering
        \includegraphics[width=\textwidth]{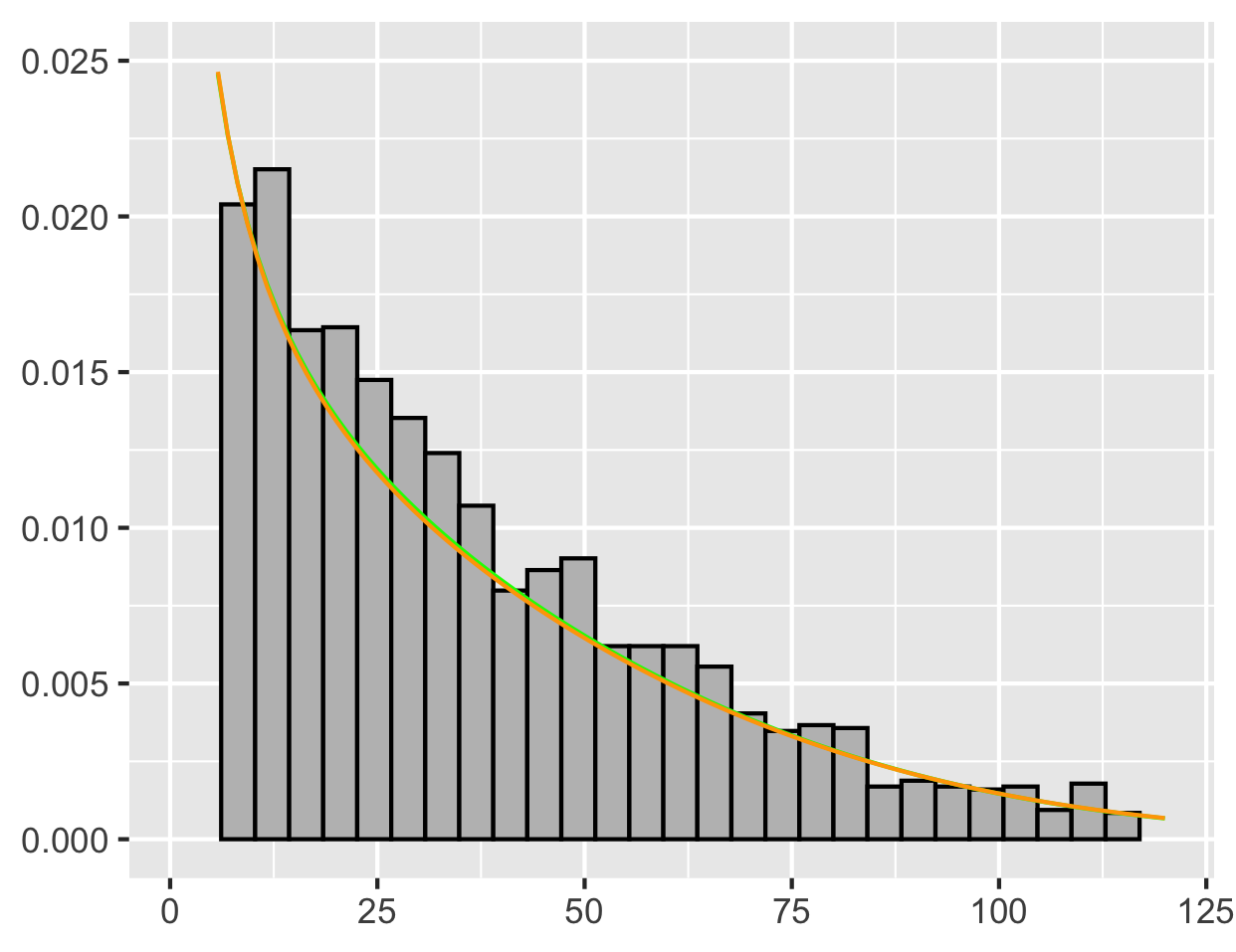}
        \caption{Weekly rainfall data}
        \label{fig:histogram-last-five-years}
    \end{subfigure}
    \begin{subfigure}[b]{0.48\textwidth}
        \centering
        \includegraphics[width=\textwidth]{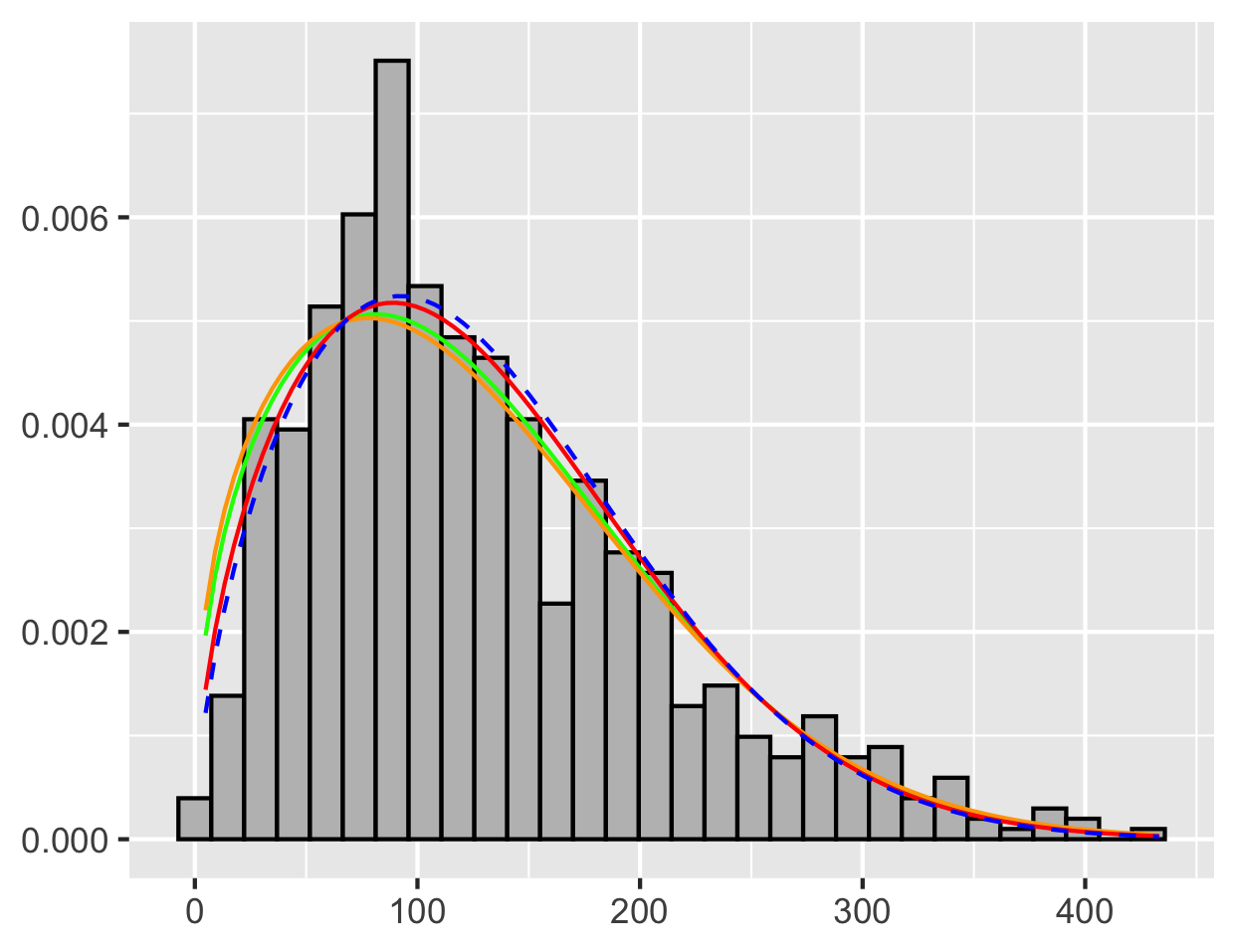}
        \caption{Monthly rainfall data}
        \label{fig:histogram-last-ten-years}
    \end{subfigure}
    \hfill
    \caption{Histograms  from Rhymney at Bargoed station (grid reference ST1559698381) with weekly data (a) and monthly data (b) from  January 1961 to December 2018. The orange curve is the density with parameters estimated through ${\mathrm{MO2}}$, the green curve through ${\mathrm{ST}}$, the blue dashed curve through MLE and the red curve through ${\mathrm{MO3}}$. The last two are only available for the monthly rainfall data because of the presence of 0's in the weekly data. }
    \label{fig:rainfall-histograms}
\end{figure}

It may also be interesting to investigate GOF or CP analysis with  our approach. A more detailed study of such data sets will be the topic of a future publication.  

\section{Discussion}\label{section_discussion} In this paper, we have developed Stein's method of moments in the context of univariate continuous probability distributions which can be characterised in a tractable manner via the density approach. Restricting ourselves to this setting has allowed us to develop a detailed asymptotic theory and analysis of `optimal functions' and to carefully assess performance via simulations; however, many directions for research remain. Whilst we have treated a number of important univariate continuous distributions, our treatment is not comprehensive. We refer the reader to the recent references \cite{wang2023new} and \cite{nw23} for applications to the Lindley, exponential and inverse Gaussian distributions, as well as the discrete negative binomial distribution. In this direction, it would be interesting to develop a general theory for univariate discrete distributions akin to our detailed treatment of the continuous case. The density approach generalises in a natural manner to multivariate continuous distributions (see \cite{mijoule2023stein}), and SMOM has recently been applied in a multivariate setting to truncated multivariate distributions \cite{fgs23truncated} and to the notoriously difficult problem of parameter estimation on the sphere by \cite{fgs23}. Finally, a number of important univariate continuous distributions do not have simple characterisations via the density approach, and characterisations are instead based on higher order differential operators (for example, the variance-gamma distribution (\cite{gaunt14})) or fractional operators (for example, stable distributions (\cite{xu19})). It would therefore be interesting to extend Stein's method of moments beyond the current density method setting.

\section*{Acknowledgements} The second and fifth authors are funded in part by ARC Consolidator grant from ULB and FNRS Grant CDR/OL J.0200.24. The second author is in addition in part funded by EPSRC Grant EP/T018445/1. The third author is funded in part by EPSRC grant EP/Y008650/1 and EPSRC grant UKRI068.

\appendix

\section{Further details for Example \ref{gammaexample5}}\label{appa}

\subsection{(Strong) consistency and asymptotic efficiency}

Let us study the two-step Stein estimator from Example \ref{gammaexample5}. For this purpose, we recall that the CDF of the gamma distribution is given by
\begin{align*}
P_{\theta}(x)=\frac{\gamma(\alpha,\beta x)}{\Gamma(\alpha)},
\end{align*}
where $\gamma(\cdot,\cdot)$ is the lower incomplete gamma function. With this formula at hand, we can calculate the optimal functions, which are given by 
\begin{align*}
& f_{\theta}^{(1)}(x)= \frac{e^{\beta  x}}{(\beta  x)^{\alpha}} \Big( \gamma (\alpha ,x \beta )\big( \log (\beta  x)-\psi(\alpha ) \big) -  \frac{(\beta x)^{\alpha}}{\alpha^2} \, _2F_2(\alpha,\alpha;1+\alpha,1+\alpha;-\beta x) \Big), \\
& f_{\theta}^{(2)}(x)=1/\beta.
\end{align*} We give some details regarding the computation of $f_{\theta}^{(1)}$. First note that
\begin{align*}
  \frac{\partial}{\partial \alpha}  \frac{\gamma(\alpha,\beta x)}{\Gamma(\alpha)} & = \frac{1}{\Gamma(\alpha)} \bigg( \int_0^{\beta x} \log(t) t^{\alpha -1} e^{-t}\,dt - \psi(\alpha) \gamma(\alpha,\beta x) \bigg).
\end{align*}
The integral in the expression above evaluates to
\begin{align*}
& \sum_{k=0}^{\infty} \frac{(-1)^k}{k!} \int_0^{\beta x} \log(t) t^{\alpha+k-1}\, dt \\
  &= \sum_{k=0}^{\infty} \frac{(-1)^k}{k!} \frac{(\beta x)^{\alpha + k}\big((\alpha + k ) \log(\beta x) -1 \big)} {(\alpha+k)^2} \\
&=\log(\beta x) \sum_{k=0}^{\infty} \frac{(-1)^k}{k!} \frac{(\beta x)^{\alpha + k}} {\alpha+k} - \sum_{k=0}^{\infty} \frac{(-1)^k}{k!} \frac{(\beta x)^{\alpha + k}} {(\alpha+k)^2} \\ 
&=\log(\beta x) \int_0^{\beta x} \sum_{k=0}^{\infty} \frac{(-1)^k}{k!} t^{\alpha + k-1}\,dt  - \frac{(\beta x)^{\alpha}}{\alpha^2} \, _2F_2(\alpha,\alpha;1+\alpha,1+\alpha;-\beta x) \\
&= \log(\beta x) \gamma(\alpha,\beta x)  -  \frac{(\beta x)^{\alpha}}{\alpha^2} \, _2F_2(\alpha,\alpha;1+\alpha,1+\alpha;-\beta x).
\end{align*}
We consider a consistent first-step estimator which satisfies Assumption \ref{assumptions_optimal_func}(a) and Assumption (i) of Theorem \ref{theorem_efficiency_optimal_functions} (for example \eqref{gamma_log_estimators}). By studying the asymptotic behaviour of $f_{\theta}^{(1)}(x)$ and using its continuity in $x$, one easily observes that $f_{\theta}^{(1)} \in \mathscr{F}$ and that Assumption (iii) of Theorem \ref{theorem_efficiency_optimal_functions} is satisfied (see \cite[Section 16.11]{olver2010nist} for the asymptotic behaviour of generalised hypergeometric functions of large argument). Regarding Assumption \ref{assumptions_optimal_func}(c), we have $g(\theta)=(\alpha,\beta,1)^\top$ and
\begin{align*}
    M_{\theta}(x)= 
    \begin{pmatrix}
        f_{\theta}^{(1)}(x) & -xf_{\theta}^{(1)}(x) & x(f_{\theta}^{(1)})'(x) \\
        f_{\theta}^{(2)}(x) & -xf_{\theta}^{(2)}(x) & x(f_{\theta}^{(2)})'(x)
    \end{pmatrix}.
\end{align*} 
Therefore, one obtains
\begin{align*}
    M_{\theta}(x)\frac{\partial}{\partial \theta}g(\theta)= 
     \begin{pmatrix}
        f_{\theta}^{(1)}(x) & -xf_{\theta}^{(1)}(x) \\
        f_{\theta}^{(2)}(x) & -xf_{\theta}^{(2)}(x)
    \end{pmatrix}.
\end{align*}
Notice that for $X \sim \mathbb{P}_{\theta_0}$, we have that $\mathbb{E}[\frac{\partial}{\partial \theta} \log p_{\theta}(X) \vert_{\theta_0}]=0$, which is equivalent to 
\begin{align*}
    \begin{pmatrix}
        \mathbb{E}\big[f_{\theta_0}^{(1)}(X)\big] & - \mathbb{E}\big[Xf_{\theta_0}^{(1)}(X)\big] \\
         \mathbb{E}\big[f_{\theta_0}^{(2)}(X)\big] & - \mathbb{E}\big[Xf_{\theta_0}^{(2)}(X)\big]
    \end{pmatrix}
    \begin{pmatrix}
         \alpha_0 \\
         \beta_0
    \end{pmatrix}
    =
    \begin{pmatrix}
         -\mathbb{E}\big[X(f_{\theta_0}^{(1)})'(X)\big] \\
         -\mathbb{E}\big[X\big(f_{\theta_0}^{(2)}\big)'(X)\big]
    \end{pmatrix},
\end{align*}
and implies that $\mathbb{E}[M_{\theta_0}(X)\frac{\partial}{\partial \theta}g(\theta)\vert_{\theta_0}]$ is invertible. Additionally, we notice that the map $\theta \mapsto  M_{\theta}(x)$ is continuously differentiable for all $x>0$. Let us now tackle Assumption \ref{assumptions_optimal_func}(d). Note that it suffices to consider the absolute value of each matrix entry of $M_{\theta}(x)$ resp. $\frac{\partial}{\partial \theta} \mathrm{vec}(M_{\theta}(x))$. We start by looking for a dominating function for $f_{\theta}^{(1)}$ with respect to some neighbourhood of $\theta_0$. We observe $f_{\theta}^{(1)}(x) = O(\log x)$ as $x \rightarrow 0$, and calculate that, for any $\theta \in \Theta$,
\begin{align*}
    \lim_{x \downarrow 0} \frac{f_{\theta}^{(1)}(x)}{\log x} = \frac{1}{\alpha} \quad \text{as well as} \quad \lim_{x \rightarrow \infty} \frac{f_{\theta}^{(1)}(x)}{\log x} = 0.
\end{align*}
One can see from the series expansion of the incomplete lower gamma function
\begin{align*}
    \gamma(\alpha,\beta x)=(\beta x)^{\alpha} \sum_{k=0}^{\infty} \frac{(-1)^k \beta^k}{k!(\alpha + k)}x^k
\end{align*}
that
\begin{align*}
    f_{\theta}^{(1)}(x)=e^{\beta x} \bigg(\big(\log(\beta x) - \psi(\alpha) \big)  \sum_{k=0}^{\infty} \frac{(-1)^k \beta^k}{k!(\alpha + k)}x^k - \sum_{k=0}^{\infty} \frac{(-1)^k \beta^k}{k!(\alpha + k)^2}x^k\bigg),
\end{align*}
whereby the two sums above decrease faster than $\exp(-\beta x)$. Therefore, with the joint continuity of the map $(x,\theta) \mapsto f_{\theta}^{(1)}(x)$ and the compactification $\iota:(0,\infty) \rightarrow (0,1)$, $x \mapsto 1/(1+x^2)$ we see that there is a $K >0$ such that, for any compact neighbourhood $\Theta' \subset \Theta$ of $\theta_0$,
\begin{align*}
    \sup_{x>0} \sup_{\theta \in \Theta'} \bigg\vert  \frac{f_{\theta}^{(1)}(x)}{\log x} \bigg\vert \leq  \sup_{x \in [0,1]} \sup_{\theta \in \Theta'}  \bigg\vert \frac{f_{\theta}^{(1)}(\iota^{-1}(x))}{\log(\iota^{-1}(x))} \bigg\vert \leq K
\end{align*}
(in the expression on the right-hand side we take the limit for $x=0$ or $x=1$). This implies that, for each $x>0$,
\begin{align*}
    \sup_{\theta \in \Theta'} \vert f_{\theta}^{(1)}(x) \vert =  \sup_{\theta \in \Theta'} \vert\log x \vert \bigg\vert \frac{f_{\theta}^{(1)}(x)}{\log x} \bigg\vert \leq \vert \log x \vert K,
\end{align*}
which is integrable with respect to any gamma distribution. The reasoning for all other entries of $M_{\theta}(x)$ is the same, for completeness we give the corresponding functions, derivatives and limits for $f_{\theta}^{(1)}$; the respective arguments for $f_{\theta}^{(2)}$ are trivial. We have
\begin{align*}
\frac{\partial}{\partial \alpha} f_{\theta}^{(1)}(x) &= \frac{e^{\beta  x}}{(\beta  x)^{\alpha}} \Big( 2(\beta x)^{\alpha}/\alpha^3 \, _3F_3(\alpha,\alpha,\alpha;1+\alpha,1+\alpha,1+\alpha;-\beta x)- \psi'(\alpha) \gamma (\alpha ,x \beta ) \\
& -\big( \log (\beta  x)-\psi(\alpha ) \big)   (\beta x)^{\alpha} \, _2F_2(\alpha,\alpha;1+\alpha,1+\alpha;-\beta x)/\alpha^2  \Big), \\
\frac{\partial}{\partial \beta} f_{\theta}^{(1)}(x) &= \frac{1}{\alpha^2 \beta} \Big( 
e^{\beta x} (\alpha - \beta x) \, _2F_2(\alpha,\alpha;1+\alpha,1+\alpha;-\beta x)  \\
& + \alpha^2 (\log(\beta x) - \psi(\alpha)) \big( 1 - e^{\beta x} (\beta x)^{-\alpha} (\alpha- \beta x)  \gamma(\alpha, \beta x)  \big)  \Big), 
\end{align*}
with
\begin{align*}
    \lim_{x \downarrow 0} \frac{\frac{\partial}{\partial \alpha}f_{\theta}^{(1)}(x)}{\log x} = -\frac{1}{\alpha^2}, \quad \lim_{x \rightarrow \infty} \frac{\frac{\partial}{\partial \alpha}f_{\theta}^{(1)}(x)}{\log x} = 0, \quad \lim_{x \downarrow 0} \frac{\partial}{\partial \beta}f_{\theta}^{(1)}(x) = \frac{1}{\alpha\beta}, \quad \lim_{x \rightarrow \infty} \frac{\partial}{\partial \beta}f_{\theta}^{(1)}(x) = 0
\end{align*}
(the limits can be obtained as before by considering the appropriate series expansions; we omit the details). Further, we calculate the derivative with respect to $x$:
\begin{align*}
    \big(f_{\theta}^{(1)}\big)'(x)= & \frac{1}{\alpha^2 x} \Big( 
e^{\beta x} (\alpha - \beta x) \, _2F_2(\alpha,\alpha;1+\alpha,1+\alpha;-\beta x)  \\
& + \alpha^2 (\log(\beta x) - \psi(\alpha)) \big( 1 - e^{\beta x} (\beta x)^{-\alpha} (\alpha- \beta x)  \gamma(\alpha, \beta x)  \big)  \Big) ,
\end{align*}
and have
\begin{align*}
    \lim_{x \downarrow 0} x\big(f_{\theta}^{(1)}\big)'(x) = \frac{1}{\alpha} \quad \text{and} \quad \lim_{x \rightarrow \infty} x\big(f_{\theta}^{(1)}\big)'(x) = 0.
\end{align*}
Moreover, for the derivatives with respect to the parameters we have
\begin{align*}
    x\frac{\partial}{\partial \alpha} \big(f_{\theta}^{(1)}\big)'(x) &= -\psi'(x)+e^{\beta  x}\bigg(\frac{1}{\alpha^3} \Big( \alpha \, _2F_2(\alpha,\alpha;1+\alpha,1+\alpha;-\beta x)  \big(1+ (\alpha -\beta x)\log(\beta x)  \\
    & - (\alpha -\beta x) \psi(\alpha) \big)-2 (\alpha -\beta x) \, _3F_3(\alpha,\alpha,\alpha;1+\alpha,1+\alpha,1+\alpha;-\beta x) \Big) \\
    & + (x\beta)^{-\alpha} \gamma(\alpha, \beta x) \big((\alpha -\beta x) \psi'(\alpha) + \psi(\alpha) - \log(\beta x) \big) \bigg), \\
    x\frac{\partial}{\partial \beta} \big(f_{\theta}^{(1)}\big)'(x) &= \frac{1}{ \beta} \bigg(1- 
    \frac{e^{\beta x}}{\alpha^2} \big( \beta x + (\alpha - \beta x)^2 \big) \, _2F_2(\alpha,\alpha;1+\alpha,1+\alpha;-\beta x) - (\alpha - \beta x) \log(\beta x)   \\
    & + e^{\beta x}(\beta x)^{-\alpha}\big( \beta x + (\alpha - \beta x)^2 \big) \gamma(\alpha, \beta x)  (\log(\beta x) - \psi(\alpha)) +\psi(\alpha) (\alpha- \beta x)  \bigg), 
\end{align*}
with corresponding asymptotic behaviour
\begin{align*}
    \lim_{x \downarrow 0} x\frac{\partial}{\partial \alpha}\big(f_{\theta}^{(1)}\big)'(x) &=- \frac{1}{\alpha^2}, \quad \lim_{x \rightarrow \infty} x\frac{\partial}{\partial \alpha}\big(f_{\theta}^{(1)}\big)'(x) = 0, \\
    \lim_{x \downarrow 0} x\frac{\partial}{\partial \beta}\big(f_{\theta}^{(1)}\big)'(x) &= 0, \quad \lim_{x \rightarrow \infty} x\frac{\partial}{\partial \beta}\big(f_{\theta}^{(1)}\big)'(x) = 0.
\end{align*}
Hence, Assumption \ref{assumptions_optimal_func}(d) holds. Note that we can express the derivative of $f_{\theta}^{(1)}$ in terms of the function itself:
\begin{align} 
    \big(f_{\theta}^{(1)}\big)'(x)&=\frac{p_{\theta}(x)\tau_{\theta}(x) \big(\frac{\partial}{\partial \alpha} p_{\theta}(x) \big) - \big(\frac{\partial}{\partial \alpha} P_{\theta}(x) \big) \big(p_{\theta}(x)\tau_{\theta}(x)\big)'}{ \big(p_{\theta}(x)\tau_{\theta}(x)\big)^2} \nonumber\\
    &=\frac{\frac{\partial}{\partial \alpha} \log p_{\theta}(x)}{\tau(x)} - \big(\log(p_{\theta}(x) \tau_{\theta}(x)) \big)'f_{\theta}^{(1)}(x)  \nonumber\\
\label{deriv_opt_func_gamma}    &=\frac{\log \beta - \psi(\alpha) + \log x}{x} -  \bigg(\frac{\alpha}{x}-\beta \bigg)f_{\theta}^{(1)}(x).
\end{align}
Here, we take $\hat{\theta}_n^{\mathrm{LOG}}$ as a first-step estimate and denote the resulting two-step estimator by $\hat{\theta}_n^{\mathrm{ST}}$. 
Taking into account the results from the previous paragraph, we are able to apply Theorems \ref{theorem_consistency_two_step} and \ref{theorem_efficiency_optimal_functions} and deduce (strong) consistency and asymptotic efficiency of $\hat{\theta}_n^{\mathrm{ST}}$.

\subsection{Simulation study}

In addition to the earlier defined estimators, we note that the MLE $\hat{\theta}_n^{\mathrm{ML}}=(\hat{\alpha}_n^{\mathrm{ML}},\hat{\beta}_n^{\mathrm{ML}})$ is defined through the equations
\begin{align*}   \log\big(\hat{\alpha}_n^{\mathrm{ML}}\big) - \psi\big(\hat{\alpha}_n^{\mathrm{ML}}\big)  = \log\big( \overline{X} \big) - \overline{\log X} \mbox{ and }
    \hat{\beta}_n^{\mathrm{ML}}  = \frac{\hat{\alpha}_n^{\mathrm{ML}}}{\overline{X}},
\end{align*}
where $\psi(\cdot)$ denotes the digamma function. It is not difficult to see that a unique solution to the latter equations exists almost surely. We performed a competitive simulation study implemented in R, in which we included the moment estimator $\hat{\theta}_n^{\mathrm{MO}}$, the MLE $\hat{\theta}_n^{\mathrm{ML}}$, the logarithmic estimator $\hat{\theta}_n^{\mathrm{LOG}}$, as well as the Stein estimator $\hat{\theta}_n^{\mathrm{ST}}$. Concerning the computation of the MLE, we used the Nelder-Mead algorithm implemented in the R function \texttt{optim} to calculate the maximum of the log-likelihood function, while the logarithmic estimator served as an initial guess. The two-step Stein estimator requires an evaluation of the function $f_{\theta}^{(1)}$. For the sake of the latter task, it turns out to be more efficient to calculate the derivative of the CDF with respect to the shape parameter $\alpha$ numerically instead of approximating the (analytical continuation) of the generalised hypergeometric function. We employed the R function \texttt{grad} of the R package \textit{numDeriv} \cite{gilbert2019numDeriv}. Note that through \eqref{deriv_opt_func_gamma}, we are able to avoid the numerically heavy task of evaluating the derivative $\big(f_{\theta}^{(1)}\big)'$. Simulation results for sample sizes $n=20$ and $n=50$ can be found in Tables \ref{gamma_sim_n20} and \ref{gamma_sim_n50}. We compared the estimation procedures in terms of bias and MSE. The logarithmic estimator behaves very similarly to the MLE (as already described in \cite{ye2017closed}) and the Stein estimator shows a performance almost identical to the MLE, which is as expected since the first-step estimator used already gives an estimate close to the true value. \par

We also considered the minimum Stein discrepancy estimators developed in \cite{barp2019minimum} since they seem to be natural competitors as the discrepancy is based on the density approach Stein identity. We chose the estimator obtained through a reproducing kernel Hilbert space with the Gaussian kernel $(s,t) \mapsto \exp\big(-\vert s-t \vert^2/2\big)$ which yields the estimator $\hat{\theta}_n^{\mathrm{MSK}}=(\hat{\alpha}_n^{\mathrm{MSK}},\hat{\beta}_n^{\mathrm{MSK}})$ given through
\begin{align*}
	\hat{\theta}_n^{\mathrm{MSK}}= &\underset{\alpha,\beta>0}{\mathrm{argmin}} \, \frac{1}{n(n-1)} \sum_{i \neq j} \exp\big(-\vert X_i-X_j \vert^2/2\big) \\
	 & \times \Big(X_i^2 (2X_j^2-\alpha) -X_i^3 X_j + \alpha (\alpha -X_j^2 -X_j \beta) + X_i (X_j (1+2\alpha + \beta^2) - X_j^3 - \alpha \beta ) \Big).
\end{align*}
However, a simulation study showed that the latter estimator does not seem to be competitive and the results are therefore not reported in Tables \ref{gamma_sim_n20} and \ref{gamma_sim_n50}. More precisely, $\hat{\theta}_n^{\mathrm{MSK}}$ is outperformed throughout almost all parameter values in terms of bias and MSE by most estimators that were included in the simulation and does not seem to exist in many cases for certain parameter constellations. We also emphasise that $\hat{\theta}_n^{\mathrm{MSK}}$ requires more computational effort given the double sums in the formula above. Moreover, for certain distributions, the minimum Stein discrepancy estimators necessitate a numerical procedure in cases where the Stein estimator is still explicit (e.g.\ for the Nakagami distribution, see Section \ref{subsection_nakagami_distribution}). This is why we decided to not include the minimum Stein discrepancy estimator in the simulation studies in this paper.

\begin{table}[h] \small
\centering
\begin{tabular}{cc|cccc|cccc}
 $\theta_0$ & & \multicolumn{4}{|c}{Bias} & \multicolumn{4}{|c}{MSE} \\ \hline  & & $\hat{\theta}_n^{\mathrm{MO}}$ & $\hat{\theta}_n^{\mathrm{ML}}$ & $\hat{\theta}_n^{\mathrm{LOG}}$ & $\hat{\theta}_n^{\mathrm{ST}}$ & $\hat{\theta}_n^{\mathrm{MO}}$ & $\hat{\theta}_n^{\mathrm{ML}}$ & $\hat{\theta}_n^{\mathrm{LOG}}$ & $\hat{\theta}_n^{\mathrm{ST}}$ \\ \hline
\multirow{2}{*}{$(1,1)$} & $\alpha$  & 0.274 & 0.141 & 0.149 & \hl 0.14 & 0.327 & \hl 0.159 & 0.167 & \hl 0.159\\ & $\beta$ & 0.341 & 0.2 & 0.208 & \hl 0.199 & 0.516 & \hl 0.282 & 0.294 & \hl 0.282\\ \hline 
\multirow{2}{*}{$(0.5,1)$} & $\alpha$  & 0.184 & 0.061 & 0.068 & \hl 0.06 & 0.109 & \hl 0.031 & 0.034 & \hl 0.031\\ & $\beta$ & 0.529 & 0.252 & 0.268 & \hl 0.251 & 1.02 & 0.422 & 0.446 & \hl 0.421\\ \hline 
\multirow{2}{*}{$(1,2)$} & $\alpha$  & 0.279 & 0.147 & 0.154 & \hl 0.146 & 0.321 & \hl 0.164 & 0.169 & \hl 0.164\\ & $\beta$ & 0.697 & 0.415 & 0.432 & \hl 0.414 & 2.03 & \hl 1.14 & 1.18 & \hl 1.14\\ \hline 
\multirow{2}{*}{$(2,1)$} & $\alpha$  & 0.462 & \hl 0.319 & 0.328 & \hl 0.319 & 1.08 & \hl 0.729 & 0.749 & \hl 0.729\\ & $\beta$ & 0.263 & \hl 0.19 & 0.194 & \hl 0.19 & 0.35 & \hl 0.246 & 0.252 & \hl 0.246\\ \hline 
\multirow{2}{*}{$(4,5)$} & $\alpha$  & 0.781 & \hl 0.624 & 0.634 & \hl 0.624 & 3.62 & \hl 2.92 & 2.97 & \hl 2.92\\ & $\beta$ & 1.05 & \hl 0.853 & 0.867 & \hl 0.853 & 6.49 & \hl 5.27 & 5.35 & \hl 5.27\\ \hline 
\multirow{2}{*}{$(3,6)$} & $\alpha$  & 0.624 & 0.481 & 0.488 & \hl 0.48 & 2.19 & \hl 1.69 & 1.72 & \hl 1.69\\ & $\beta$ & 1.38 & \hl 1.08 & 1.1 & \hl 1.08 & 10.4 & \hl 8.15 & 8.27 & \hl 8.15\\ \hline 
\multirow{2}{*}{$(7,0.5)$} & $\alpha$  & 1.32 & \hl 1.17 & 1.18 & \hl 1.17 & 11 & \hl 9.79 & 9.86 & \hl 9.79\\ & $\beta$ & 0.098 & \hl 0.088 & \hl 0.088 & \hl0.088 & 0.06 & \hl 0.054 & \hl 0.054 & \hl 0.054\\ \hline 
\multirow{2}{*}{$(0.2,8)$} & $\alpha$  & 0.121 & \hl 0.021 & 0.025 & \hl 0.021 & 0.035 & $3.96\text{e-3}$ & $4.38\text{e-3}$ & $ \hl 3.94\text{\hl e-3}$\\ & $\beta$ & 9.19 & 3.84 & 4.05 & \hl 3.83 & 299 & 93.6 & 98.6 & \hl 93.4\\ \hline 
\multirow{2}{*}{$(1,9)$} & $\alpha$  & 0.282 & 0.15 & 0.157 & \hl 0.149 & 0.333 & \hl 0.166 & 0.173 & \hl 0.166\\ & $\beta$ & 3.12 & \hl 1.88 & 1.95 & \hl 1.88 & 41.7 & \hl 23.9 & 24.6 & \hl 23.9\\ \hline 
\multirow{2}{*}{$(0.2,0.1)$} & $\alpha$  & 0.123 & 0.021 & 0.025 & \hl 0.02 & 0.036 & $4.01\text{e-3}$ & $4.46\text{e-3}$ & $ \hl 3.99\text{\hl e-3}$\\ & $\beta$ & 0.112 & \hl 0.045 & 0.048 & \hl 0.045 & 0.046 & \hl 0.014 & 0.015 & \hl 0.014\\ \hline 
\end{tabular} 
\caption{\protect\label{gamma_sim_n20} Simulation results for the $\Gamma(\alpha,\beta)$ distribution for $n=20$ and $10,000$ repetitions.}
\end{table}

\begin{table}[h] \small
\centering
\begin{tabular}{cc|cccc|cccc}
 $\theta_0$ & & \multicolumn{4}{|c}{Bias} & \multicolumn{4}{|c}{MSE} \\ \hline  & & $\hat{\theta}_n^{\mathrm{MO}}$ & $\hat{\theta}_n^{\mathrm{ML}}$ & $\hat{\theta}_n^{\mathrm{LOG}}$ & $\hat{\theta}_n^{\mathrm{ST}}$ & $\hat{\theta}_n^{\mathrm{MO}}$ & $\hat{\theta}_n^{\mathrm{ML}}$ & $\hat{\theta}_n^{\mathrm{LOG}}$ & $\hat{\theta}_n^{\mathrm{ST}}$ \\ \hline 
\multirow{2}{*}{$(1,1)$} & $\alpha$  & 0.114 & 0.055 & 0.058 &\hl 0.054 & 0.095 &\hl 0.042 & 0.044 &\hl 0.042\\ & $\beta$ & 0.135 & \hl 0.074 & 0.078 &\hl 0.074 & 0.134 &\hl 0.07 & 0.074 &\hl 0.07\\ \hline 
\multirow{2}{*}{$(0.5,1)$} & $\alpha$  & 0.081 &\hl 0.023 & 0.026 &\hl 0.023 & 0.034 & $8.83\text{e-3}$ & $9.53\text{e-3}$ & $\hl 8.81\text{\hl e-3}$\\ & $\beta$ & 0.212 & 0.091 & 0.098 & \hl 0.09 & 0.23 &\hl 0.097 & 0.101 &\hl 0.097\\ \hline 
\multirow{2}{*}{$(1,2)$} & $\alpha$  & 0.115 & 0.055 & 0.058 &\hl 0.054 & 0.093 &\hl 0.041 & 0.043 &\hl 0.041\\ & $\beta$ & 0.276 &\hl 0.153 & 0.16 &\hl 0.153 & 0.518 & 0.283 & 0.292 &\hl 0.282\\ \hline 
\multirow{2}{*}{$(2,1)$} & $\alpha$  & 0.183 &\hl 0.117 & 0.121 &\hl 0.117 & 0.302 &\hl 0.189 & 0.195 &\hl 0.189\\ & $\beta$ & 0.1 &\hl 0.067 & 0.069 &\hl 0.067 & 0.091 &\hl 0.061 & 0.062 &\hl 0.061\\ \hline 
\multirow{2}{*}{$(4,5)$} & $\alpha$  & 0.306 &\hl 0.249 & 0.25 &\hl 0.249 & 1.07 &\hl 0.831 & 0.846 &\hl 0.831\\ & $\beta$ & 0.406 &\hl 0.335 & 0.336 &\hl 0.335 & 1.88 &\hl 1.49 & 1.52 &\hl 1.49\\ \hline 
\multirow{2}{*}{$(3,6)$} & $\alpha$  & 0.263 &\hl 0.197 & 0.201 &\hl 0.197 & 0.632 &\hl 0.456 & 0.466 &\hl 0.456\\ & $\beta$ & 0.573 & 0.44 & 0.447 &\hl 0.439 & 2.93 &\hl 2.19 & 2.23 &\hl 2.19\\ \hline 
\multirow{2}{*}{$(7,0.5)$} & $\alpha$  & 0.51 &\hl 0.454 &\hl 0.454 &\hl 0.454 & 2.93 &\hl 2.56 & 2.57 &\hl 2.56\\ & $\beta$ & 0.038 &\hl 0.034 &\hl 0.034 &\hl 0.034 & 0.016 &\hl 0.014 &\hl 0.014 &\hl 0.014\\ \hline 
\multirow{2}{*}{$(0.2,8)$} & $\alpha$  & 0.055 & $7.78\text{e-3}$ & $9.56\text{e-3}$ & $\hl 7.69\text{\hl e-3}$ & 0.01 & $\hl 1.18\text{\hl e-3}$ & $1.28\text{e-3}$ & $\hl 1.18\text{\hl e-3}$\\ & $\beta$ & 3.33 &\hl 1.21 & 1.29 &\hl 1.21 & 43.1 &\hl 14.3 & 14.9 &\hl 14.3\\ \hline 
\multirow{2}{*}{$(1,9)$} & $\alpha$  & 0.116 & \hl 0.055 & 0.058 &\hl 0.055 & 0.096 &\hl 0.043 & 0.045 &\hl 0.043\\ & $\beta$ & 1.25 & 0.69 & 0.721 & \hl 0.688 & 10.9 & 5.95 & 6.15 & \hl 5.94\\ \hline 
\multirow{2}{*}{$(0.2,0.1)$} & $\alpha$  & 0.055 & $7.4\text{e-3}$ & $9.08\text{e-3}$ & $\hl 7.3\text{\hl e-3}$ & 0.01 & $\hl 1.12\text{\hl e-3}$ & $1.22\text{e-3}$ & $\hl 1.12\text{\hl e-3}$\\ & $\beta$ & 0.042 &\hl 0.015 & 0.016 &\hl 0.015 & $7.1\text{e-3}$ & $\hl 2.34\text{\hl e-3}$ & $2.44\text{e-3}$ & $\hl 2.34\text{\hl e-3}$\\ \hline 
\end{tabular} 
\caption{\protect\label{gamma_sim_n50} Simulation results for the $\Gamma(\alpha,\beta)$ distribution for $n=50$ and $10,000$ repetitions.}
\end{table}

\section{Further details for Section \ref{section_applications}}

\subsection{Truncated normal distribution}

Here we state the formulas for the first and second moments of the truncated normal distribution, from which moment estimators $\hat{\mu}_n^{\mathrm{MO}}$ and $\hat{\sigma}_n^{\mathrm{MO}}$ are derived. For $X \sim TN(\mu, \sigma^2)$,
\begin{align*}
    \mathbb{E}[X] &= \mu-\sigma \frac{\phi\big( \frac{b-\mu}{\sigma} \big) - \phi\big( \frac{a-\mu}{\sigma} \big) }{\Phi\big( \frac{b-\mu}{\sigma} \big) -\Phi\big( \frac{a-\mu}{\sigma} \big) }, \\
     \mathbb{E}[X^2] &= \mu^2+ \sigma^2 \bigg( 1- \frac{\frac{b-\mu}{\sigma} \phi\big( \frac{b-\mu}{\sigma} \big) - \frac{a-\mu}{\sigma} \phi\big( \frac{a-\mu}{\sigma} \big) }{\Phi\big( \frac{b-\mu}{\sigma} \big) -\Phi\big( \frac{a-\mu}{\sigma} \big) } \bigg)- 2\mu\sigma \frac{\phi\big( \frac{b-\mu}{\sigma} \big) - \phi\big( \frac{a-\mu}{\sigma} \big) }{\Phi\big( \frac{b-\mu}{\sigma} \big) -\Phi\big( \frac{a-\mu}{\sigma} \big) }.
\end{align*}

\subsection{Cauchy distribution}

In this section, we give explicit forms for the estimators $\hat{\mu}_n^{\mathrm{L1}}$, $\hat{\mu}_n^{\mathrm{L2}}$, $\hat{\mu}_n^{\mathrm{L3}}$, $\hat{\mu}_n^{\mathrm{L4}}$ and $\hat{\mu}_n^{\mathrm{PI}}$ from Section \ref{cauchysec}. We also provide some additional simulation results.

Let $X_{(1)},\ldots,X_{(n)}$ be the order statistics of $X_1,\ldots,X_n$. In \cite{rothenberg1964note}, the authors showed that it can be more efficient to use only a part of the observations due to the heavy tails of the Cauchy distribution, and proposed the estimator
\begin{align*}
    \hat{\mu}_n^{\mathrm{L1}}= \frac{1}{n-2r} \sum_{i=r+1}^{n-r} X_{(i)},
\end{align*}
where $r=\lfloor np \rfloor$, $0<p<0.5$, with $p=0.38$ recommended in order to achieve a minimal variance. In \cite{bloch1966note}, the author developed an L-estimator based on only $5$ order statistics, defined by
\begin{align*}
     \hat{\mu}_n^{\mathrm{L2}}=  \sum_{i=1}^{5} \omega_i X_{(\lfloor np_i \rfloor)},
\end{align*}
and found 
\begin{align*}(p_1,p_2,p_3,p_4,p_5)& = (0.13,0.4,0.5,0.6,0.87),\\ 
(\omega_1,\omega_2,\omega_3,\omega_4,\omega_5)&= (-0.052,0.3485,0.407,0.3485,-0.052)
\end{align*}
to be optimal. Another asymptotically efficient estimator was developed in \cite{chernoff1967asymptotic}:
\begin{align*}
     \hat{\mu}_n^{\mathrm{L3}}= \frac{1}{n}  \sum_{i=1}^{n} J\bigg( \frac{i}{n+1} \bigg) X_{(i)},
\end{align*}
where $J(u)=\sin\big(4\pi(u-0.5) \big)/\tan\big(\pi(u-0.5) \big)$. Furthermore, \cite{zhang2010highly} proposed to modify the latter estimator in order to also achieve a high efficiency for finite sample sizes and obtained the estimator $ \hat{\mu}_n^{\mathrm{L4}}= \sum_{i=1}^{n} \omega_i X_{(i)}$, where $\omega_i=c^{-1} \big(\cos \big(\pi\big( \frac{i-0.5}{n}-0.5\big) \big) \big)^{2+\epsilon_n} \cos \big(2\pi\big( \frac{i-0.5}{n}-0.5\big) \big) $, $c\in \mathbb{R}$, are such that $\sum_{i=1}^n \omega_i=1$. The constant $\epsilon_n$ needs to be chosen such that the estimator is efficient for the corresponding sample size $n$. The authors give propositions for certain values of $n$; in fact, $\epsilon_n$ should converge to $0$ as $n$ grows. If we choose $\epsilon_n=0$, we retrieve $\hat{\mu}_n^{\mathrm{L3}}$, which is asymptotically normal and therefore $\hat{\mu}_n^{\mathrm{L4}}$ is asymptotically normal as well. We also consider the Pitman estimator \cite{freue2007pitman}, given by
\begin{align*}
     \hat{\mu}_n^{\mathrm{PI}}= \sum_{i=1}^{n}  \frac{\mathrm{Re}(\omega_i)}{\sum_{j=1}^n \mathrm{Re}(\omega_j)}  X_i,
\end{align*}
where
\begin{align*}
    \omega_k=\prod_{j \neq k} \bigg( \frac{1}{(X_k-X_j)^2+4\gamma^2} \bigg) \bigg(1- \frac{2\gamma}{X_k-X_j}\mathrm{i} \bigg).
\end{align*}
Note that in the latter formula, $\mathrm{i}$ represents the imaginary unit rather than an index. 

Now, we provide additional simulations for the Cauchy distribution $C(\mu,\gamma)$ with the parameter $\gamma$ known for sample sizes $n=100$ and $n=250$.

\begingroup
\setlength\tabcolsep{3.5pt}
\begin{table}[h] \small
\centering
\begin{tabular}{cc|ccccc|ccccc}
 $\theta_0$ & & \multicolumn{5}{|c}{Bias} & \multicolumn{5}{|c}{MSE} \\ \hline
 & & $\hat{\mu}_n^{\mathrm{L1}}$ & $\hat{\mu}_n^{\mathrm{L2}}$  & $\hat{\mu}_n^{\mathrm{L4}}$ & $\hat{\mu}_n^{\mathrm{PI}}$ & $\hat{\mu}_n^{\mathrm{ST1}}$ & $\hat{\mu}_n^{\mathrm{L1}}$ & $\hat{\mu}_n^{\mathrm{L2}}$  & $\hat{\mu}_n^{\mathrm{L4}}$ & $\hat{\mu}_n^{\mathrm{PI}}$ & $\hat{\mu}_n^{\mathrm{ST1}}$ \\ \hline
\multirow{1}{*}{$(-5,1)$} & $\mu$  & $-0.193$ & $-9.97\text{e-3}$ & $-2.39\text{e-3}$ & $-0.215$ & \hl $\hl -2.33\text{e-3}$ & $0.063$ & $0.022$ & $0.023$ & $607$ & \hl $\hl 0.021$ \\ \hline 
\multirow{1}{*}{$(-4,1.5)$} & $\mu$  & $-0.141$ & $-0.011$ & \hl $\hl -2.04\text{e-4}$ & $6.52\text{e-4}$ & $7.74\text{e-4}$ & $0.078$ & $0.05$ & $0.053$ & $23.7$ & \hl $\hl 0.048$ \\ \hline 
\multirow{1}{*}{$(-2,2)$} & $\mu$  & $-0.053$ & $-0.018$ & $-4.75\text{e-3}$ & $-0.057$ & \hl $\hl -2.79\text{e-3}$ & $0.106$ & $0.087$ & $0.094$ & $44.4$ & \hl $\hl 0.082$ \\ \hline 
\multirow{1}{*}{$(0,1)$} & $\mu$  & $0.018$ & $-6.89\text{e-3}$ & \hl $\hl 4.05\text{e-4}$ & $-0.043$ & $9.06\text{e-4}$ & $0.026$ & $0.022$ & $0.023$ & $20.2$ & \hl $\hl 0.021$ \\ \hline 
\multirow{1}{*}{$(0,3)$} & $\mu$  & $0.055$ & $-0.021$ & $2.25\text{e-3}$ & $3.95\text{e-3}$ & \hl $\hl 1.74\text{e-3}$ & $0.233$ & $0.198$ & $0.211$ & $63.6$ & \hl $\hl 0.189$ \\ \hline 
\multirow{1}{*}{$(2,0.1)$} & $\mu$  & $0.085$ & $-6.74\text{e-4}$ & $4.88\text{e-5}$ & $-0.011$ & \hl $\hl 1.05\text{e-4}$ & $7.5\text{e-3}$ & $2.17\text{e-4}$ & $2.31\text{e-4}$ & $33$ & \hl $\hl 2.06\text{e-4}$ \\ \hline 
\multirow{1}{*}{$(2,0.5)$} & $\mu$  & $0.093$ & $-2.86\text{e-3}$ & \hl $\hl 8.8\text{e-4}$ & $-0.024$ & $1.14\text{e-3}$ & $0.015$ & $5.39\text{e-3}$ & $5.89\text{e-3}$ & $10.1$ & \hl $\hl 5.14\text{e-3}$ \\ \hline 
\multirow{1}{*}{$(4,0.8)$} & $\mu$  & $0.18$ & $-6.57\text{e-3}$ & $-8.02\text{e-4}$ & $0.162$ & \hl $\hl -5.24\text{e-4}$ & $0.049$ & $0.014$ & $0.015$ & $155$ & \hl $\hl 0.013$ \\ \hline 
\multirow{1}{*}{$(6,2.3)$} & $\mu$  & $0.29$ & $-0.019$ & \hl $\hl -6.86\text{e-4}$ & $0.165$ & $-7.25\text{e-4}$ & $0.22$ & $0.116$ & $0.125$ & $215$ & \hl $\hl 0.109$ \\ \hline 
\multirow{1}{*}{$(10,0.2)$} & $\mu$  & $0.42$ & $-1.42\text{e-3}$ & \hl $\hl 1.98\text{e-4}$ & $-0.344$ & $2.42\text{e-4}$ & $0.178$ & $9.06\text{e-4}$ & $9.68\text{e-4}$ & $908$ & \hl $\hl 8.53\text{e-4}$ \\ \hline 
\end{tabular} 
\caption{\protect\label{cauchy_sim_n100} Simulation results for the $C(\mu,\gamma)$ distribution for $n=100$ and $10,000$ repetitions.}
\end{table}
\endgroup

\begingroup
\setlength\tabcolsep{3.5pt}
\begin{table}[h] \small
\centering
\begin{tabular}{cc|ccccc|ccccc}
 $\theta_0$ & & \multicolumn{5}{|c}{Bias} & \multicolumn{5}{|c}{MSE} \\ \hline
 & & $\hat{\mu}_n^{\mathrm{L1}}$ & $\hat{\mu}_n^{\mathrm{L2}}$  & $\hat{\mu}_n^{\mathrm{L4}}$ & $\hat{\mu}_n^{\mathrm{PI}}$ & $\hat{\mu}_n^{\mathrm{ST1}}$ & $\hat{\mu}_n^{\mathrm{L1}}$ & $\hat{\mu}_n^{\mathrm{L2}}$  & $\hat{\mu}_n^{\mathrm{L4}}$ & $\hat{\mu}_n^{\mathrm{PI}}$ & $\hat{\mu}_n^{\mathrm{ST1}}$ \\ \hline
\multirow{1}{*}{$(-5,1)$} & $\mu$  & $-0.077$ & $5.75\text{e-4}$ & \hl $\hl -4.1\text{e-4}$ & $0.317$ & $-4.66\text{e-4}$ & $0.016$ & $8.56\text{e-3}$ & $9.29\text{e-3}$ & $1710$ & \hl $\hl 8.13\text{e-3}$ \\ \hline 
\multirow{1}{*}{$(-4,1.5)$} & $\mu$  & $-0.058$ & \hl $\hl 4.46\text{e-5}$ & $-1.71\text{e-3}$ & $0.419$ & $-2.17\text{e-3}$ & $0.024$ & $0.019$ & $0.02$ & $1398$ & \hl $\hl 0.018$ \\ \hline 
\multirow{1}{*}{$(-2,2)$} & $\mu$  & $-0.021$ & $1.73\text{e-3}$ & $-1.18\text{e-3}$ & $-0.131$ & \hl $\hl -4.47\text{e-4}$ & $0.038$ & $0.034$ & $0.036$ & $193$ & \hl $\hl 0.032$ \\ \hline 
\multirow{1}{*}{$(0,1)$} & $\mu$  & $8.51\text{e-3}$ & $2.77\text{e-3}$ & \hl $\hl 1.7\text{e-3}$ & $-0.011$ & $1.83\text{e-3}$ & $9.64\text{e-3}$ & $8.53\text{e-3}$ & $9.19\text{e-3}$ & $35.6$ & \hl $\hl 8.07\text{e-3}$ \\ \hline 
\multirow{1}{*}{$(0,3)$} & $\mu$  & $0.023$ & $6.67\text{e-3}$ & \hl $\hl 2.75\text{e-3}$ & $-0.048$ & $4.35\text{e-3}$ & $0.088$ & $0.077$ & $0.084$ & $0.375$ & \hl $\hl 0.073$ \\ \hline 
\multirow{1}{*}{$(2,0.1)$} & $\mu$  & $0.034$ & $1.35\text{e-4}$ & \hl $\hl -2.92\text{e-6}$ & $\ast$ & $4.97\text{e-5}$ & $1.25\text{e-3}$ & $8.63\text{e-5}$ & $9.28\text{e-5}$ & $\ast$ & \hl $\hl 8.29\text{e-5}$ \\ \hline 
\multirow{1}{*}{$(2,0.5)$} & $\mu$  & $0.037$ & $7.34\text{e-4}$ & $3.19\text{e-4}$ & $0.044$ & \hl $\hl -1.7\text{e-5}$ & $3.77\text{e-3}$ & $2.13\text{e-3}$ & $2.31\text{e-3}$ & $65.6$ & \hl $\hl 2.04\text{e-3}$ \\ \hline 
\multirow{1}{*}{$(4,0.8)$} & $\mu$  & $0.071$ & \hl $-2.14\text{e-4}$ & $-1.51\text{e-3}$ & $0.033$ & $-1.07\text{e-3}$ & $0.011$ & $5.51\text{e-3}$ & $5.91\text{e-3}$ & $34.9$ & \hl $\hl 5.25\text{e-3}$ \\ \hline 
\multirow{1}{*}{$(6,2.3)$} & $\mu$  & $0.115$ & $2.07\text{e-3}$ & \hl $\hl -5.43\text{e-4}$ & $-2.11$ & $-1.31\text{e-3}$ & $0.063$ & $0.045$ & $0.048$ & $4.07\text{e4}$ & \hl $\hl 0.042$ \\ \hline 
\multirow{1}{*}{$(10,0.2)$} & $\mu$  & $0.168$ & $2.04\text{e-4}$ & $-2.75\text{e-5}$ & $0.058$ & \hl $\hl -1.18\text{e-5}$ & $0.029$ & $3.41\text{e-4}$ & $3.68\text{e-4}$ & $535$ & \hl $\hl 3.27\text{e-4}$ \\ \hline 
\end{tabular} 
\caption{\protect\label{cauchy_sim_n250} Simulation results for the $C(\mu,\gamma)$ distribution for $n=250$ and $10,000$ repetitions. Here, $\ast$ stands for \texttt{NaN} type results.}
\end{table}
\endgroup

\subsection{Exponential polynomial models}

In this section, we provide additional details on the competitor estimation procedures. The noise-contrastive estimator developed in \cite{gutmann2012noise} (a refined version of one from \cite{gutmann2010noise}) is defined as follows. We choose the exponential distribution $E(\lambda)$ with parameter $\lambda=1/\overline{X}$ and density $p_{\lambda}(x)=\lambda e^{-\lambda x}$, $x >0$, as the noise distribution and generate an i.i.d.\ sample, denoted by $Y_1,\ldots,Y_d \sim E(\lambda)$. Moreover, we take the tuning parameter $\nu=10$, which gives $d=10n$ for the sample size of the noise distribution. Letting $G_{\theta,\lambda}(x)=\log p_{\theta}(x) - \log p_{\lambda}(x)$ and $r_{\nu}(x)=(1+\nu\exp(-x))^{-1}$, the noise-contrastive estimator maximises the quantity
\begin{align*}
    J(\theta)=\frac{1}{n}\bigg(\sum_{i=1}^n \log r_{\nu}(G_{\theta,\lambda}(X_i)) + \sum_{i=1}^d\log\big(1- r_{\nu}(G_{\theta,\lambda}(Y_i)\big) \bigg).
\end{align*}
We write $\hat{\theta}_n^{\mathrm{NC}}$ for the noise-contrastive estimator. We remark that $\hat{\theta}_n^{\mathrm{NC}}$ is not explicit and needs to be computed via numerical optimisation. Moreover, we consider the score matching approach from \cite{hyvarinen2007some} (a refined version of \cite{hyvarinen2005estimation}). This boils down to finding the minimum of
\begin{align*}
    JJ(\theta)&=\frac{1}{n}\sum_{i=1}^n \bigg( 2X_i \frac{p_{\theta}'(X_i)}{p_{\theta}(X_i)} + \frac{p_{\theta}''(X_i)}{p_{\theta}(X_i)}X_i^2-\frac{1}{2}\frac{p_{\theta}'(X_i)^2}{p_{\theta}(X_i)^2}X_i^2 \bigg) \\
    &=\frac{1}{n}\sum_{i=1}^n \bigg(\sum_{j=1}^p j(j+1)\theta_jX_i^j + \frac{1}{2}  \bigg(\sum_{j=1}^p j\theta_jX_i^j \bigg)^2 \bigg).
\end{align*}
We write $\hat{\theta}_n^{\mathrm{SM}}$ for the score matching estimator. Although it is possible to work out a unique stationary point of the target function for $\hat{\theta}_n^{\mathrm{SM}}$, we will minimise the function numerically. In \cite{hayakawa2016estimation} and \cite{nakayama2011holonomic}, the holomorphic gradient method is used in order to compute the MLE. Note that, for exponential polynomial models, the MLE coincides with the moment estimator. To conclude, we implement the minimum $\mathscr{L}^q$-estimator obtained from \cite{betsch2021minimum}. The latter are also motivated by a Stein characterisation; more precisely, by an expectation-based representation of the CDF. This estimator, which we denote by $\hat{\theta}_n^{\mathrm{MD}}$, is calculated as follows. Let
\begin{align*}
    \eta_n(t,\theta)&=-\frac{1}{n}\sum_{i=1}^n \frac{p_{\theta}'(X_i)}{p_{\theta}(X_i)} \min(t,X_i) - \frac{1}{n}\sum_{i=1}^n 1\{X_i \leq t \} \\
    &=-\frac{1}{n}\sum_{i=1}^n \min(t,X_i) \sum_{j=1}^p j \theta_jx^{j-1} - \frac{1}{n}\sum_{i=1}^n 1\{X_i \leq t \}.
\end{align*}
Then we define $\hat{\theta}_n^{\mathrm{MD}}=\argmin \{\Vert \eta_n(\cdot,\theta) \Vert_{\mathscr{L}^q} \, \vert \, \theta \in \Theta \}$, where $\Vert \cdot \Vert_{\mathscr{L}^q}$, $1 \leq q < \infty$, is the $\mathscr{L}^q$-norm defined for a function $f:\mathbb{R}_{+} \rightarrow \mathbb{R}$ by
\begin{align*}
    \Vert f \Vert_{\mathscr{L}^q}= \bigg( \int_0^\infty \vert f(t) \vert^q w(t) \,dt \bigg)^{1/q},
\end{align*}
for a positive and integrable weight function $w:\mathbb{R}_{+} \rightarrow \mathbb{R}$. The authors recommend $q=2$,  $w(t)=e^{-at}$ and the tuning parameter $a=1$ seems to produce the best results based on their simulations. The minimum-distance estimator $\hat{\theta}_n^{\mathrm{MD}}$ is only explicit for a parameter space of dimension less than or equal to $2$.

\section{Further applications}

In this section, we complement the work of Section \ref{section_applications} with a number of further applications. 
We stress that for all examples, if not explicitly stated differently, we suppose that $\{X_n,n \in \mathbb{Z}\}$ is i.i.d. However, with Example~\ref{cauchy_noniid} we also consider a case of dependent data. In Section~\ref{lomax_section}, we work out the conditions for the existence of the MLE and show that the assumptions from Theorem~\ref{theorem_sequence_stein_estimator} are satisfied. We emphasise that all Stein operators in the next section are obtained with the density approach \eqref{definition_Stein_kernel_operator} and thus Assumption (ii) of Theorem \ref{theorem_efficiency_optimal_functions} is always satisfied.

\subsection{Beta distribution}
The density of the beta distribution $B(\alpha,\beta)$ with parameter $\theta=(\alpha,\beta)$, $\alpha, \beta >0$, is given by
\begin{align*}
p_{\theta}(x)=\frac{1}{B(\alpha,\beta)}x^{\alpha-1}(1-x)^{\beta -1}, \quad 0 < x < 1,
\end{align*}
where $B(\cdot,\cdot)$ is the beta function. We apply the Stein kernel approach and take $\tau_{\theta}(x)=x(1-x)$. With $\mathbb{E}[X]=\alpha/(\alpha+\beta)$ we conclude that a Stein operator is given by
\begin{align*}
\mathcal{A}_{\theta}f(x)=x(1-x)f'(x)+(\alpha - (\alpha+\beta)x )f(x)
\end{align*}
(see also \cite{dobler2015stein,goldstein2013stein}). We choose two test functions $f_1$, $f_2$ and retrieve
\begin{align*}
\begin{cases}
\overline{Xf_1(X)} \beta +\overline{(X-1)f_1(X)}\alpha = \overline{X(1-X)f_1'(X)} \\
\overline{Xf_2(X)} \beta +\overline{(X-1)f_2(X)}\alpha = \overline{X(1-X)f_2'(X)}.
\end{cases}
\end{align*}
This leaves us with the estimators
\begin{align*}
\hat{\alpha}_n&=\frac{\overline{Xf_1(X)} \ \overline{X(1-X)f_2'(X)}-\overline{Xf_2(X)} \ \overline{X(1-X)f_1'(X)}}{\overline{Xf_1(X)} \ \overline{(X-1)f_2(X)}-\overline{Xf_2(X)} \  \overline{(X-1)f_1(X)}}, \\
\hat{\beta}_n&=\frac{\overline{(X-1)f_2(X)} \ \overline{X(1-X)f_1'(X)}-\overline{(X-1)f_1(X)} \ \overline{X(1-X)f_2'(X)}}{\overline{Xf_1(X)} \ \overline{(X-1)f_2(X)}-\overline{Xf_2(X)} \  \overline{(X-1)f_1(X)}}.
\end{align*}
By choosing $f_1(x)=1$ and $f_2(x)=x$ we obtain the moment estimators
\begin{align*}
\hat{\alpha}_n^{\mathrm{MO}}=\frac{\overline{X}(\overline{X}-\overline{X^2})}{\overline{X^2}-\overline{X}^2} \quad \text{and} \quad 
\hat{\beta}_n^{\mathrm{MO}}=\frac{(1-\overline{X})(\overline{X}-\overline{X^2})}{\overline{X^2}-\overline{X}^2}.
\end{align*}
Since the values of the $X_i$ are always between $0$ and $1$, the moment estimators return positive values with probability $1$ and therefore exist almost surely. In addition, by choosing $f_1(x)=1$ and $f_2(x)=\log(x/(1-x))$, we recover the explicit estimators
\begin{align*}
\hat{\alpha}_n^{\mathrm{LOG}}&=\frac{\overline{X}}{\overline{X\log(X/(1-X))}-\overline{X} \ \overline{\log(X/(1-X))}}, \\
\hat{\beta}_n^{\mathrm{LOG}}&=\frac{(1-\overline{X})}{\overline{X\log(X/(1-X))}-\overline{X} \ \overline{\log(X/(1-X))}} 
\end{align*}
proposed in \cite{papadatos2022point}, which show a behaviour close to asymptotic efficiency. Note that the latter estimates are always positive due to a multivariate version of Jensen's inequality and the convexity of the function $(x,y) \mapsto xy$. We remark that the estimator of \cite{papadatos2022point} was also obtained through a version of Stein's method based on a covariance identity. The chosen test functions were motivated by the logarithmic estimator for the gamma distribution \eqref{gamma_log_estimators}. \par
We futher consider the MLE $\hat{\theta}_n^{\mathrm{ML}}=(\hat{\alpha}_n^{\mathrm{ML}},\hat{\beta}_n^{\mathrm{ML}})$ defined as the solution to the following system of equations:
\begin{gather*}
    \psi\big( \hat{\alpha}_n^{\mathrm{ML}}\big) - \psi\big( \hat{\alpha}_n^{\mathrm{ML}}+\hat{\beta}_n^{\mathrm{ML}}\big) = \overline{\log X}, \\
    \psi\big( \hat{\beta}_n^{\mathrm{ML}}\big) - \psi\big( \hat{\alpha}_n^{\mathrm{ML}}+\hat{\beta}_n^{\mathrm{ML}}\big) = \overline{\log(1- X)}.
\end{gather*}
It is easy to see that the likelihood function of the $B(\alpha,\beta)$ distribution is strictly convex and therefore has a unique maximum characterised by the equations above. However, for some parameter constellations, the MLE can be difficult to compute due to the return of non-finite values for the likelihood function. As in Section \ref{appa}, we study the two-step Stein estimator. The CDF of the Beta distribution is given through
\begin{align*}
P_{\theta}(x)=\frac{B_x(\alpha,\beta)}{B(\alpha,\beta)},
\end{align*}
where $B_x$ is the incomplete beta function. We recover
\begin{align*}
f_{\theta}^{(1)}(x)&=x^{-\alpha } (1-x)^{-\beta } \bigg(B_x(\alpha ,\beta ) (\psi (\alpha +\beta )-\psi (\alpha )+\log (x)) \\
&\quad- \frac{x^{\alpha } \, _3F_2(\alpha ,\alpha ,1-\beta ;\alpha +1,\alpha +1;x)}{\alpha
   ^2}\bigg),
\end{align*}
and
\begin{align*}
f_{\theta}^{(2)}(x)&=x^{-\alpha } \bigg(\frac{\, _3F_2(1-\alpha ,\beta ,\beta ;\beta +1,\beta +1;1-x)}{\beta ^2} \\
&\quad- (1-x)^{-\beta } B_{1-x}(\beta ,\alpha ) (\psi(\alpha +\beta )-\psi(\beta )+\log
   (1-x))\bigg).
\end{align*}
In this case,
\begin{align*}
    M_{\theta}(x)= 
    \begin{pmatrix}
        (x-1)f_{\theta}^{(1)}(x)  & xf_{\theta}^{(1)}(x) & x(x-1)(f_{\theta}^{(1)})'(x) \\
         (x-1)f_{\theta}^{(2)}(x) & xf_{\theta}^{(2)}(x) & x(x-1)(f_{\theta}^{(2)})'(x)
    \end{pmatrix}
\end{align*}
and $g(\theta)=(\alpha,\beta,1)^\top$. The conditions for Theorems \ref{theorem_consistency_two_step} and \ref{theorem_efficiency_optimal_functions} can be verified in a similar way as in the previous section, so that we can conclude (strong) consistency and asymptotic efficiency for a suitable first-step estimator which we choose to be the logarithmic estimator $\hat{\theta}_n^{\mathrm{LOG}}$. The resulting two-step estimator will be denoted by $\hat{\theta}_n^{\mathrm{ST}}$. As in Section \ref{appa}, we are able to give a formula for the derivatives of the optimal functions. We obtain
\begin{align*}
    \big(f_{\theta}^{(1)}\big)'(x)=\frac{\psi(\alpha + \beta) - \psi(\alpha) + \log x}{x(1-x)}-\bigg(\frac{\alpha}{x}-\frac{\beta}{1-x} \bigg)f_{\theta}^{(1)}(x)
\end{align*}
and
\begin{align*}
    \big(f_{\theta}^{(2)}\big)'(x)=\frac{\psi(\alpha + \beta) - \psi(\beta) + \log(1-x)}{x(1-x)}-\bigg(\frac{\alpha}{x}-\frac{\beta}{1-x} \bigg)f_{\theta}^{(2)}(x).
\end{align*}
The results of a competitive simulation study can be found in Tables \ref{beta_sim_n20} and \ref{beta_sim_n50}. The MLE is calculated equivalently to the MLE of the gamma distribution with the logarithmic estimator as initial guess. The simulation results are akin to the results for the gamma distribution with the logarithmic estimator, the two-step Stein estimator and the MLE showing a similar behaviour. However, the moment estimator seems to yield good results regarding the bias, whereby it is outperformed in terms of the MSE for most parameter values.

\begin{table}[h] \small
\centering
\begin{tabular}{cc|cccc|cccc}
 $\theta_0$ & & \multicolumn{4}{|c}{Bias} & \multicolumn{4}{|c}{MSE} \\ \hline  & & $\hat{\theta}_n^{\mathrm{MO}}$ & $\hat{\theta}_n^{\mathrm{ML}}$ & $\hat{\theta}_n^{\mathrm{LOG}}$ & $\hat{\theta}_n^{\mathrm{ST}}$ & $\hat{\theta}_n^{\mathrm{MO}}$ & $\hat{\theta}_n^{\mathrm{ML}}$ & $\hat{\theta}_n^{\mathrm{LOG}}$ & $\hat{\theta}_n^{\mathrm{ST}}$ \\ \hline
\multirow{2}{*}{$(1,1)$} & $\alpha$  & \hl 0.131 & 0.155 & 0.145 & 0.156 & 0.194 & 0.182 & \hl 0.18 & 0.183\\ & $\beta$ & \hl 0.129 & 0.152 & 0.143 & 0.153 & 0.191 & 0.178 & \hl 0.176 & 0.178\\ \hline 
\multirow{2}{*}{$(2,1)$} & $\alpha$  & 0.329 & 0.346 & \hl 0.328 & 0.347 & 0.949 & 0.85 & \hl 0.838 & 0.851\\ & $\beta$ & 0.148 & 0.149 & \hl 0.145 & 0.15 & 0.193 & \hl 0.165 & \hl 0.165 & 0.166\\ \hline 
\multirow{2}{*}{$(0.2,0.5)$} & $\alpha$  & 0.024 & \hl 0.021 & \hl 0.021 & \hl 0.021 & $9.9\text{e-3}$ & $ \hl 4.44\text{\hl e-3}$ & $4.66\text{e-3}$ & $4.47\text{e-3}$\\ & $\beta$ & 0.097 & 0.104 & \hl 0.095 & 0.104 & 0.184 & \hl 0.091 & \hl 0.091 & \hl 0.091\\ \hline 
\multirow{2}{*}{$(0.5,5)$} & $\alpha$  & 0.128 & \hl 0.06 & 0.064 & \hl 0.06 & 0.084 & 0.032 & 0.033 & \hl 0.031\\ & $\beta$ & 1.81 & \hl 1.11 & \hl 1.11 & \hl 1.11 & 17.7 & 8.88 & 9.16 & \hl 8.87\\ \hline 
\multirow{2}{*}{$(3,0.4)$} & $\alpha$  & 1.09 & 0.729 & \hl 0.709 & 0.729 & 7.11 & \hl 3.62 & 3.66 & \hl 3.62\\ & $\beta$ & 0.095 & \hl 0.048 & 0.05 & \hl 0.048 & 0.052 & \hl 0.02 & 0.021 & \hl 0.02\\ \hline 
\multirow{2}{*}{$(5,7)$} & $\alpha$  & \hl 0.809 & 0.85 & 0.834 & 0.851 & \hl 4.93 & 4.96 & \hl 4.93 & 4.96\\ & $\beta$ & \hl 1.12 & 1.18 & 1.15 & 1.18 & \hl 9.81 & 9.88 & \hl 9.81 & 9.88\\ \hline 
\multirow{2}{*}{$(8,4)$} & $\alpha$  & \hl 1.37 & 1.4 & 1.38 & 1.4 & 14.1 & 13.8 & \hl 13.7 & 13.8\\ & $\beta$ & \hl 0.646 & 0.661 & 0.649 & 0.661 & 3.2 & 3.14 & \hl 3.12 & 3.14\\ \hline 
\multirow{2}{*}{$(0.9,0.3)$} & $\alpha$  & 0.191 & 0.192 & \hl 0.175 & 0.192 & 0.379 & 0.269 & \hl 0.262 & 0.269\\ & $\beta$ & 0.042 & \hl 0.034 & \hl 0.034 & 0.035 & 0.021 & \hl 0.011 & \hl 0.011 & \hl 0.011\\ \hline 
\multirow{2}{*}{$(6,5)$} & $\alpha$  & \hl 0.952 & 1.01 & 0.989 & 1.01 & \hl 7.4 & 7.46 & 7.42 & 7.46\\ & $\beta$ & \hl 0.792 & 0.84 & 0.822 & 0.84 & \hl 5.02 & 5.06 & 5.03 & 5.06\\ \hline 
\multirow{2}{*}{$(4,4)$} & $\alpha$  & \hl 0.62 & 0.668 & 0.65 & 0.668 & \hl 3.08 & 3.11 & 3.09 & 3.11\\ & $\beta$ & \hl 0.63 & 0.678 & 0.66 & 0.679 & \hl 3.19 & 3.22 & 3.2 & 3.22\\ \hline 
\end{tabular} 
\caption{\protect\label{beta_sim_n20} Simulation results for the $B(\alpha,\beta)$ distribution for $n=20$ and $10,000$ repetitions.}
\end{table}

\begin{table}[h] \small
\centering
\begin{tabular}{cc|cccc|cccc}
 $\theta_0$ & & \multicolumn{4}{|c}{Bias} & \multicolumn{4}{|c}{MSE} \\ \hline  & & $\hat{\theta}_n^{\mathrm{MO}}$ & $\hat{\theta}_n^{\mathrm{ML}}$ & $\hat{\theta}_n^{\mathrm{LOG}}$ & $\hat{\theta}_n^{\mathrm{ST}}$ & $\hat{\theta}_n^{\mathrm{MO}}$ & $\hat{\theta}_n^{\mathrm{ML}}$ & $\hat{\theta}_n^{\mathrm{LOG}}$ & $\hat{\theta}_n^{\mathrm{ST}}$ \\ \hline
\multirow{2}{*}{$(1,1)$} & $\alpha$  & \hl 0.046 & 0.054 & 0.051 & 0.055 & 0.051 & \hl 0.045 & \hl 0.045 & \hl 0.045\\ & $\beta$ & \hl 0.046 & 0.055 & 0.051 & 0.056 & 0.052 & \hl 0.045 & \hl 0.045 & \hl 0.045\\ \hline 
\multirow{2}{*}{$(2,1)$} & $\alpha$  & \hl 0.116 & 0.126 & 0.119 & 0.126 & 0.257 & \hl 0.227 & \hl 0.227 & \hl 0.227\\ & $\beta$ & \hl 0.051 & 0.053 & \hl 0.051 & 0.053 & 0.055 & \hl 0.044 & 0.045 & \hl 0.044\\ \hline 
\multirow{2}{*}{$(0.2,0.5)$} & $\alpha$  & $9.1\text{e-3}$ & $8.11\text{e-3}$ & $ \hl 8.02\text{\hl e-3}$ & $8.34\text{e-3}$ & $3.08\text{e-3}$ & $ \hl 1.25\text{\hl e-3}$ & $1.32\text{e-3}$ & $ \hl 1.25\text{\hl e-3}$\\ & $\beta$ & \hl 0.03 & 0.035 & 0.032 & 0.035 & 0.024 & \hl 0.017 & 0.018 & \hl 0.017\\ \hline 
\multirow{2}{*}{$(0.5,5)$} & $\alpha$  & 0.049 & \hl 0.023 & 0.024 & \hl 0.023 & 0.023 & $8.66\text{e-3}$ & $9.17\text{e-3}$ & $ \hl 8.65\text{\hl e-3}$\\ & $\beta$ & 0.66 & 0.422 & \hl 0.418 & 0.421 & 3.81 & \hl 2.15 & 2.18 & \hl 2.15\\ \hline 
\multirow{2}{*}{$(3,0.4)$} & $\alpha$  & 0.363 & 0.253 & \hl 0.243 & 0.253 & 1.42 & \hl 0.822 & 0.829 & \hl 0.822\\ & $\beta$ & 0.035 & \hl 0.018 & \hl 0.018 & \hl 0.018 & 0.015 & $ \hl 5.26\text{\hl e-3}$ & $5.57\text{e-3}$ & $ \hl 5.26\text{\hl e-3}$\\ \hline 
\multirow{2}{*}{$(5,7)$} & $\alpha$  & \hl 0.277 & 0.293 & 0.287 & 0.293 & 1.24 & \hl 1.23 & \hl 1.23 & \hl 1.23\\ & $\beta$ & \hl 0.398 & 0.421 & 0.412 & 0.421 & 2.53 & \hl 2.5 & \hl 2.5 & \hl 2.5\\ \hline 
\multirow{2}{*}{$(8,4)$} & $\alpha$  & \hl 0.466 & 0.484 & 0.472 & 0.484 & 3.49 & 3.37 & \hl 3.36 & 3.37\\ & $\beta$ & \hl 0.226 & 0.234 & 0.229 & 0.235 & 0.825 & \hl 0.789 & \hl 0.789 & \hl 0.789\\ \hline 
\multirow{2}{*}{$(0.9,0.3)$} & $\alpha$  & 0.067 & 0.069 & \hl 0.063 & 0.069 & 0.081 & \hl 0.059 & 0.06 & \hl 0.059\\ & $\beta$ & 0.016 & \hl 0.013 & \hl 0.013 & \hl 0.013 & $6.44\text{e-3}$ & $ \hl 2.95\text{\hl e-3}$ & $3.11\text{e-3}$ & $2.96\text{e-3}$\\ \hline 
\multirow{2}{*}{$(6,5)$} & $\alpha$  & \hl 0.336 & 0.358 & 0.35 & 0.358 & 1.86 & \hl 1.85 & \hl 1.85 & \hl 1.85\\ & $\beta$ & \hl 0.278 & 0.296 & 0.289 & 0.296 & 1.29 & 1.28 & \hl 1.27 & 1.28\\ \hline 
\multirow{2}{*}{$(4,4)$} & $\alpha$  & \hl 0.24 & 0.26 & 0.252 & 0.261 & 0.858 & 0.852 & \hl 0.85 & 0.852\\ & $\beta$ & \hl 0.238 & 0.259 & 0.251 & 0.259 & 0.857 & 0.852 & \hl 0.849 & 0.853\\ \hline 
\end{tabular} 
\caption{\protect\label{beta_sim_n50} Simulation results for the $B(\alpha,\beta)$ distribution for $n=50$ and $10,000$ repetitions.}
\end{table}

\subsection{Student's $t$-distribution}
The density of Student's $t$-distribution $t_{\mu}$ with $\theta=\mu >0$ is given by
\begin{align*}
p_{\theta}(x)=\frac{1}{\sqrt{\mu}B(\mu/2,1/2)}\bigg(\frac{\mu}{\mu+x^2} \bigg)^{(\mu+1)/2}, \quad x \in \mathbb{R}.
\end{align*}
The Stein kernel is given by $\tau_{\theta}(x)=x^2+\mu$ and we have $\mathbb{E}[X]=0$ whereby we remark that here, $\mathbb{E}[X^k]$ for $X\sim t_\mu$ only exists if $k<\mu$. Hence, a Stein operator is given by
\begin{align*}
\mathcal{A}_{\theta}f(x)=(x^2+\mu) f'(x) - (\mu-1)x f(x)
\end{align*}
(see also \cite{schoutens2001orthogonal}). 
We choose a test function $f_1$ which results in the equation
\begin{align*}
(\overline{Xf_1(X)}-\overline{f_1'(X)})\mu  = \overline{X^2f_1'(X)}+\overline{Xf_1(X)}, 
\end{align*}
through which we obtain the estimator
\begin{align*}
\hat{\mu}_n&=\frac{\overline{X^2f_1'(X)}+\overline{Xf_1(X)}}{ \overline{Xf_1(X)}-\overline{f_1'(X)}}.
\end{align*}
Supposing the necessary assumptions hold for the test function $f_1$, we can apply Theorem \ref{theorem_asymptotic_normality} and calculate the asymptotic variance of $\hat{\mu}_n$. It is given by
\begin{align*}
\frac{ \mathbb{E}\big[ \big(\frac{x^2+\mu}{\mu-1}f_1'(x) - x f_1(x)\big)^2\big]}{\mathbb{E} \big[\frac{-x^2-1}{(\mu-1)^2}f_1'(X_1)\big]^2}.
\end{align*}
By choosing $f_1(x)=x$ we obtain the moment-type estimator
\begin{align} \label{t_moment_estimator}
\hat{\mu}_n^{\mathrm{MO}}&=\frac{2\overline{X^2}}{\overline{X^2}-1},
\end{align}
which is only consistent for $\mu>2$. Our objective here is to propose a completely explicit alternative to the moment estimator that exists also for small values of $\mu$. We recommend to use the test function $f_1(x)=x/(\kappa +x^2)$ with a tuning parameter $\kappa \in \mathbb{R}$, which gives
\begin{align} \label{t_stein_estimators}
\hat{m}_n^{\mathrm{ST}}(\kappa) = \frac{\sum_{i=1}^n \Big( \frac{\kappa X_i^2 - X_i^4}{(\kappa+X_i^2)^2} + \frac{X_i^2}{\kappa+X_i^2} \Big) } {\sum_{i=1}^n \Big( \frac{X_i^2 }{\kappa+X_i^2} + \frac{\kappa-X_i^2}{(\kappa+X_i^2)^2} \Big) }.
\end{align}
In Figure \ref{fig:studentt_variance}, the asymptotic variances of the estimators $\hat{m}_n^{\mathrm{MO}}$, $\hat{m}_n^{\mathrm{ST1}}$ and the MLE $\hat{m}_n^{\mathrm{ML}}$ (which is defined as the parameter value for $\mu$ that maximises the log-likelihood function for a given sample) are plotted. Student's $t$-distribution is a regular probability distribution such that the MLE is asymptotically efficient. 
We observe that with $\hat{m}_n^{\mathrm{ST}}$ and $\kappa=10$, we achieve a performance which is close to efficiency for values $0<\mu<4$. However, for large degree of freedoms, larger values for $\kappa$ seem to be more suitable. As it can be seen in Figure \ref{fig:studentt_variance}, the Stein estimator $\hat{m}_n^{\mathrm{ST}}$ will eventually be outperformed by the moment estimator \eqref{t_moment_estimator} as $\mu$ grows. Here, we forgo a simulation study for small sample performance, which is due to the difficulty of estimating the parameter $\mu$ for large degrees of freedoms. As it is widely known, Student's $t$-distribution converges to the standard normal distribution as $\mu \rightarrow \infty$ and therefore, the densities merely differ for large values of $\mu$. This can also be seen by the large asymptotic variance of all estimators in the right image of Figure \ref{fig:studentt_variance}. As a consequence, the finite sample variance of any estimator is very large and makes it difficult to draw concrete conclusions about the performance of estimators. We also mention the possibility to implement the asymptotically efficient two-step Stein estimator. The function $f_{\theta}^{(1)}(x)$ is complicated since the CDF of Student's $t$-distribution involves the generalised hypergeometric function, which is why we do not give the formula here explicitly. However, we can express the derivative of the latter in terms of the function itself. We obtain
\begin{align*}
    \big(f_{\theta}^{(1)}\big)'(x)&=\frac{x^2-1+(x^2+\mu)\Big(\psi\big(\frac{\mu}{2}+\frac{1}{2}\big)+\gamma+\log\big(\frac{\mu}{x^2+\mu}\big)\Big)-(x^2+\mu)\Big(\psi\big(\frac{\mu}{2}\big)+\gamma\Big)}{2(x^2+\mu)^2} \\
    &\quad- \frac{x(1-\mu)}{x^2+\mu} f_{\theta}^{(1)}(x),
\end{align*}
where  $\gamma$ is the Euler-Mascheroni constant. The necessary assumptions of Theorems \ref{theorem_consistency_two_step} and \ref{theorem_efficiency_optimal_functions} can be verified as in Section \ref{appa}.

\begin{figure}[!]
    \centering
\vspace{.2cm}
   \begin{subfigure}{.49\textwidth}
\captionsetup{width=.95\textwidth}
  \centering
  \includegraphics[width=7.4cm]{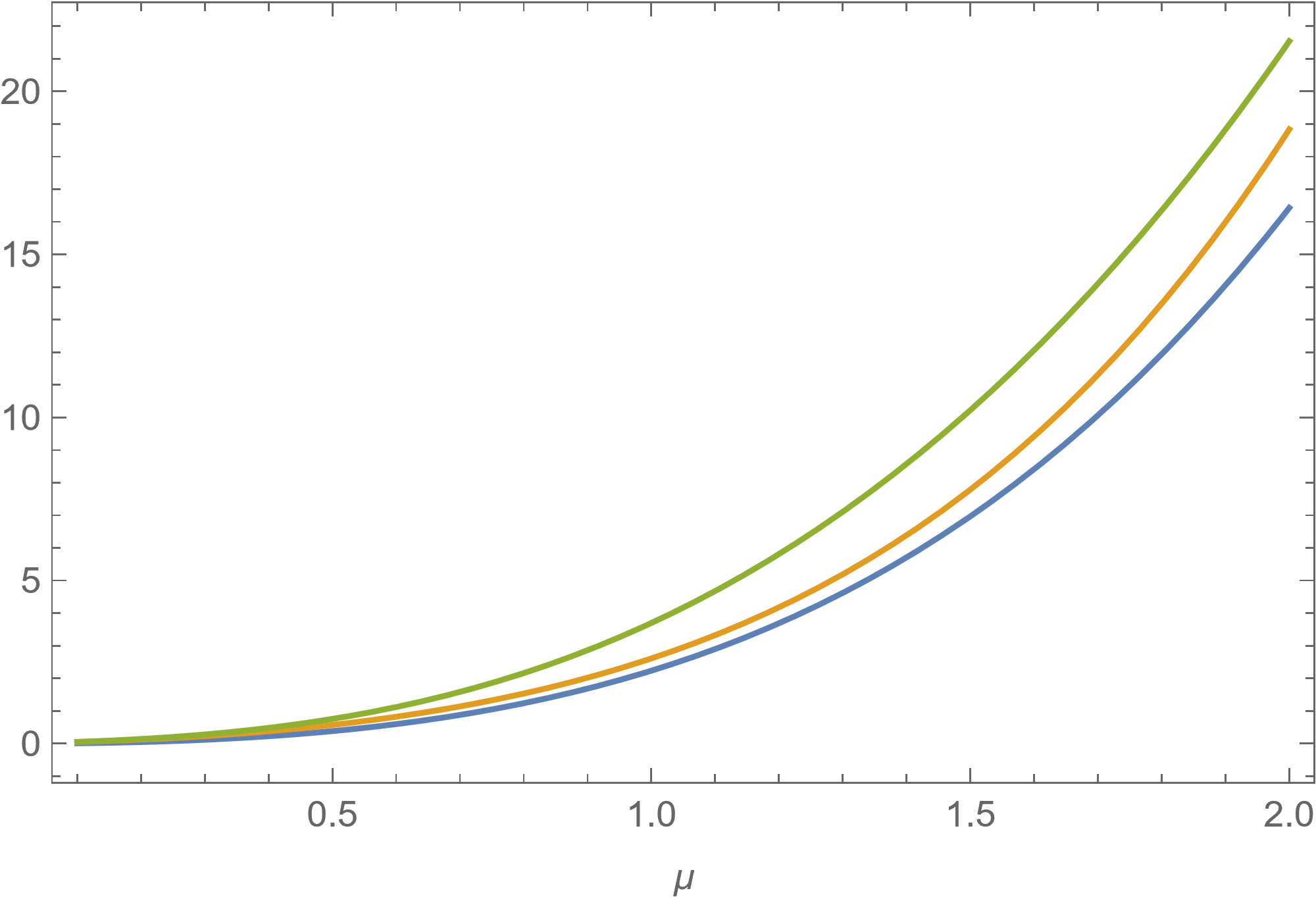}
\end{subfigure}%
  \begin{subfigure}{.49\textwidth}
\captionsetup{width=.95\textwidth}
  \centering
  \includegraphics[width=7.4cm]{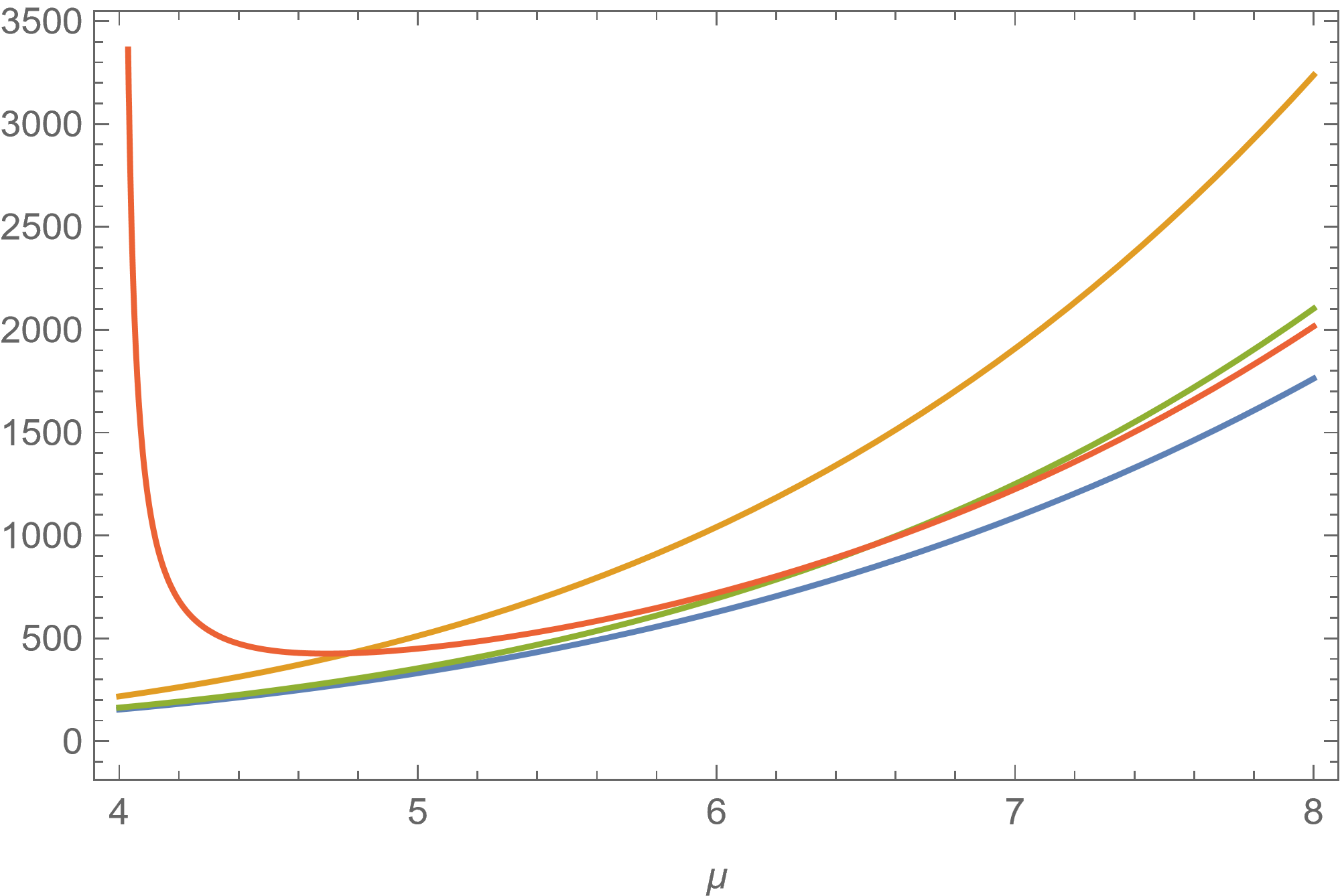}
\end{subfigure}
\caption{\protect\label{fig:studentt_variance} \it The images report the asymptotic variances of the MLE $\hat{\mu}_n^{\mathrm{ML}}$ \includegraphics[scale=1]{blue}, the moment estimator $\hat{\mu}_n^{\mathrm{MO}}$ \includegraphics[scale=1]{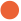} \eqref{t_moment_estimator} and the Stein estimator $\hat{\mu}_n^{\mathrm{ST}}$ \eqref{t_stein_estimators} with tuning parameter $\kappa=10$ \includegraphics[scale=1]{orange} resp.\, $\kappa=100$ \includegraphics[scale=1]{green}. } 
\end{figure}

\subsection{Lomax distribution} \label{lomax_section}
The density of the Lomax distribution $LM(\alpha,\lambda)$ with $\theta=(\alpha,\lambda)$, $\alpha, \lambda >0$, is given by
\begin{align*}
p_{\theta}(x)=\frac{\alpha}{\lambda} \bigg(1+\frac{x}{\lambda} \bigg)^{-(\alpha+1)}, \quad x >0.
\end{align*}
We find that the choice $\tau_{\theta}(x)=x+\lambda$ yields a simple Stein operator. Note that here $\mathbb{E}[X^k]$ for $X\sim LM(\alpha,\lambda)$ exists only if $k>\alpha$. Hence, a Stein operator is given by
\begin{align*}
\mathcal{A}_{\theta}f(x)=(x+\lambda)f'(x) -\alpha f(x).
\end{align*}
Our Stein estimators for test functions $f_1$, $f_2$ are given by
\begin{align*}
\hat{\alpha}_n&=\frac{\overline{f_2'(X)} \ \overline{Xf_1'(X)}-\overline{f_1'(X)} \ \overline{Xf_2'(X)}}{\overline{f_1(X)} \ \overline{f_2'(X)}- \overline{f_1'(X)}) \ \overline{f_2(X)}},  \\
\hat{\lambda}_n&=\frac{\overline{f_2(X)} \ \overline{Xf_1'(X)} -\overline{f_1(X)} \ \overline{Xf_2'(X)}}{\overline{f_1(X)} \ \overline{f_2'(X)}- \overline{f_1'(X)}) \ \overline{f_2(X)}}.
\end{align*}
The purpose of this section is to give an example for an application of Theorem \ref{theorem_sequence_stein_estimator}, that is an iterative procedure that generates a sequence of Stein estimators which converges to the MLE. Thereto, we examine the MLE, which we denote by $ \hat{\theta}_n^{\mathrm{ML}} =(\hat{\alpha}_n^{\mathrm{ML}}, \hat{\lambda}_n^{\mathrm{ML}})$, and which is given through the equations
\begin{gather} \label{lomax_MLE_lambda}
    1=\bigg(1 + \bigg( \frac{1}{n} \sum_{i=1}^n \log\big( 1 + X_i/\hat{\lambda}_n^{\mathrm{ML}} \big) \bigg)^{-1} \bigg) \bigg(\frac{1}{n} \sum_{i=1}^n \frac{X_i}{\hat{\lambda}_n^{\mathrm{ML}} + X_i}  \bigg), \\
    \label{lomax_MLE_alpha}
    \hat{\alpha}_n^{\mathrm{ML}}= \bigg( \frac{1}{n} \sum_{i=1}^n \log\big( 1 + X_i/\hat{\lambda}_n^{\mathrm{ML}} \big) \bigg)^{-1}.
\end{gather}
An explicit solution does not exist. We show in the next lemma that the MLE only exists under certain assumptions.
\begin{Lemma} \label{lemma_lomax_existence_mle}
 The MLE for the Lomax distribution $L(\alpha,\lambda)$ exists if and only if $\overline{X^2}-2\overline{X}^2>0.$
\end{Lemma}
\begin{proof}
    In \cite{giles2013bias}, it was shown that the likelihood function of the Lomax distribution with respect to an i.i.d.\ sample is strictly concave. Therefore, it admits at most one global maximum, and since the likelihood function is differentiable, this global maximum (if it exists) is characterised by \eqref{lomax_MLE_lambda} and \eqref{lomax_MLE_alpha}. Since $\hat{\alpha}_n^{\mathrm{ML}}$ can be expressed explicitly as a function of $\hat{\lambda}_n^{\mathrm{ML}}$, it is clear that the MLE exists if and only if there is a $\hat{\lambda}_n^{\mathrm{ML}}$ that satisfies equation \eqref{lomax_MLE_lambda}. In order to examine the latter condition, we rearrange the equation and get
    \begin{align} \label{lomax_proof_mle_exist_lambda_eq}
       \frac{1}{n} \sum_{i=1}^n \log\big( 1 + X_i/\hat{\lambda}_n^{\mathrm{ML}} \big) =  \bigg(\frac{1}{n} \sum_{i=1}^n \frac{X_i}{\hat{\lambda}_n^{\mathrm{ML}} + X_i}  \bigg) \bigg( \frac{1}{n} \sum_{i=1}^n \log\big( 1 + X_i/\hat{\lambda}_n^{\mathrm{ML}} \big) + 1 \bigg).
    \end{align}
    We define
    \begin{align*}
        \ell_1(\gamma)=  \frac{1}{n} \sum_{i=1}^n \log\big( 1 + \gamma  X_i \big)  \quad \text{and} \quad \ell_2(\gamma)= \bigg(\frac{1}{n} \sum_{i=1}^n \frac{\gamma X_i}{1 + \gamma X_i}  \bigg) \bigg( \frac{1}{n} \sum_{i=1}^n \log\big( 1 + \gamma X_i \big) + 1 \bigg),
    \end{align*}
    and can rewrite \eqref{lomax_proof_mle_exist_lambda_eq} by $\ell_1(1/\hat{\lambda}_n^{\mathrm{ML}})=\ell_2(1/\hat{\lambda}_n^{\mathrm{ML}})$. We know that for $\gamma>0$ large enough, we will eventually have $\ell_2(\gamma)>\ell_1(\gamma)$. Moreover, it is an easy task to compute the limits $\lim_{\gamma \rightarrow 0} \ell_{1,2}(\gamma)=0$ and $\lim_{\gamma \rightarrow \infty} \ell_{1,2}(\gamma)=\infty$. Furthermore, $\lim_{\gamma \rightarrow 0} \ell_{1,2}'(\gamma)=\overline{X}$. With the concavity of the likelihood function, it is clear that $\ell_1$ and $\ell_2$ intersect at most once for $\gamma>0$. By considering the aforementioned limits, this is the case if and only if $\lim_{\gamma \rightarrow 0}\ell_{1}''(\gamma)>\lim_{\gamma \rightarrow 0}\ell_{2}''(\gamma)$. Tedious calculations yield $\lim_{\gamma \rightarrow 0}\ell_{1}''(\gamma)=-\overline{X^2}$ and $\lim_{\gamma \rightarrow 0}\ell_{2}''(\gamma)=-2\overline{X^2}+2\overline{X}^2$, which concludes the proof.
\end{proof}
In \cite{giles2013bias}, the authors also report difficulties when performing MLE, which is in accordance with Lemma \ref{lemma_lomax_existence_mle}. It is interesting to see that the MLE exists if and only if the moment estimator $\hat{\lambda}_n^{\mathrm{MO}}$ is positive. We show that the condition for the existence of the MLE above is asymptotically satisfied and therefore complies with Assumption \ref{assumptions_iteratice_proc}(a).
\begin{Lemma}
 If $X_1,\ldots,X_n \sim L(\alpha,\lambda)$, then $\overline{X^2}-2\overline{X}^2>0$ with probability converging to one.
\end{Lemma}
\begin{proof}
Note that the event $\{\omega \in \Omega \,\vert \, \overline{X^2}-2\overline{X}^2>0\}$ is independent of $\lambda>0$, and we can therefore assume without loss of generality that $\lambda=1$. We first treat the case $\alpha>2$. In this case,
\begin{align*}
    \mathbb{P}(\overline{X^2}-2\overline{X}^2>0)&= \mathbb{P}\Big(\overline{X^2}-2\overline{X}^2 - \big( \mathbb{E}[X^2] - 2\mathbb{E}[X]^2 \big) >- \big( \mathbb{E}[X^2] - 2\mathbb{E}[X]^2 \big) \Big) \\
    &=\mathbb{P}\Big( \mathbb{E}[X^2] - 2\mathbb{E}[X]^2 - \big(\overline{X^2}-2\overline{X}^2 \big) < \mathbb{E}[X^2] - 2\mathbb{E}[X]^2 \Big)  \\
    &\geq\mathbb{P}\bigg( \big\vert \mathbb{E}[X^2] - 2\mathbb{E}[X]^2 - \big(\overline{X^2}-2\overline{X}^2 \big) \big\vert < \frac{2}{(\alpha-1)^2(\alpha-2)} \bigg),
\end{align*}
where $X \sim L(\alpha,1)$. The last expression converges to $1$ due to the strong law of large numbers. Now let $0<\alpha<2$, and let $\tilde{X}_1,\ldots\tilde{X}_n \sim L(\tilde{\alpha},1)$ be i.i.d.\ with $\tilde{\alpha}>2$. By examining the CDF of the Lomax distribution, one realises that $X_i$ and $\tilde{X}_i$ are stochastically ordered, i.e.\ $X_i \prec \tilde{X}_i$ for each $i=1,\ldots,n$ with respect to the standard ordering
\begin{align*}
    X \prec Y \quad \text{if and only if} \quad \mathbb{P}(X>x) \leq \mathbb{P}(Y>x) \mbox{ for all } x \in \mathbb{R}, 
\end{align*}
where $X$ and $Y$ are two real-valued random variables. We define the function
\begin{align*}
    F:\mathbb{R}_+^n \rightarrow \mathbb{R}, \: (x_1,\ldots,x_n) \mapsto \frac{1}{n} \sum_{i=1}^n x_i^2 - 2\bigg( \frac{1}{n} \sum_{i=1}^n x_i \bigg)^2,
\end{align*}
and have
\begin{align*}
    \frac{\partial }{\partial x_j}F(x_1,\ldots,x_n)= \frac{2x_j}{n} \bigg(1-  \frac{2}{n} \sum_{i=1}^n x_i \bigg).
\end{align*}
It is clear that for each $x \in \mathbb{R}_+^n$ that satisfies $1-2n^{-1}\sum_{i=1}^n x_i<0$, the function $F$ is monotonically decreasing in each component. We now define the sequence of sets $A_n$, $n \in \mathbb{N}$ by $A_n=\{ \omega \in \Omega \, \vert \, 1-2\overline{X}<0 \} $. Then we know that for the conditional distributions we get
\begin{align*}
     \frac{1}{n} \sum_{i=1}^n X_i^2 - 2 \bigg( \frac{1}{n} \sum_{i=1}^n X_i \bigg)^2 \, \bigg\vert \, A_n \succ \frac{1}{n} \sum_{i=1}^n \tilde{X}_i^2 - 2 \bigg( \frac{1}{n} \sum_{i=1}^n \tilde{X}_i \bigg)^2 \, \bigg\vert \, A_n.
\end{align*}
With the strong law of large numbers, $1-2\overline{X}$ will converge almost surely to $(\alpha-2)/(\alpha-1)<0$ if $1<\alpha<2$ and will diverge to $-\infty$ if $0<\alpha<1$, as $n \rightarrow \infty$. Hence, we conclude that $\mathbb{P}(A_n) \rightarrow 1$. Now, let $B_n=\{ \omega \in \Omega \, \vert \,\overline{X^2}-2\overline{X}^2 >0 \}$. Then 
\begin{align*}
    \mathbb{P}(B_n)&= \mathbb{P}(B_n \, \vert \, A_n)\mathbb{P}(A_n) + \mathbb{P}(B_n \, \vert \, \bar{A}_n)\mathbb{P}(\bar{A}_n) \\
    &\geq \mathbb{P}\bigg(\frac{1}{n} \sum_{i=1}^n \tilde{X}_i^2 - 2 \bigg( \frac{1}{n} \sum_{i=1}^n \tilde{X}_i \bigg)^2 >0 \, \bigg\vert \, A_n \bigg)\mathbb{P}(A_n) + \mathbb{P}(B_n \, \vert \, \bar{A}_n)\mathbb{P}(\bar{A}_n), 
\end{align*}
which converges to $1$ by the first part of the proof.
\end{proof}
The CDF of the Lomax distribution is given by
\begin{align*}
P_{\theta}(x)=1-\bigg(1+\frac{x}{\lambda} \bigg)^{-\alpha}, \quad x>0.
\end{align*}
The optimal functions are therefore
\begin{align*}
f_{\theta}^{(1)}(x)= \frac{ \log \big(1+\frac{x}{\lambda} \big) }{\alpha} \quad \text{and} \quad
f_{\theta}^{(2)}(x)=-\frac{x}{\lambda (x+ \lambda)}.
\end{align*}
It remains to verify Assumptions \ref{assumptions_iteratice_proc}(c), (b) and (d), of which the latter two can be checked easily. For condition (c), we have to show that $\mathbb{E}[M_{\theta}(X)]\frac{\partial }{\partial \theta} g(\theta)$ is invertible on $\Theta_0$. We compute
\begin{align*}
    M_{\theta}(x)\frac{\partial}{\partial \theta}g(\theta)= 
     \begin{pmatrix}
        -f_{\theta}^{(1)}(x) & \big(f_{\theta}^{(1)}\big)'(x) \\
        -f_{\theta}^{(2)}(x) & \big(f_{\theta}^{(2)}\big)'(x)
    \end{pmatrix}
\end{align*}
and show that the expectation with respect to $\mathbb{P}_{\theta_0}$ of the latter matrix is invertible for all $(\alpha_0,\lambda_0) \in \Theta$ and $(\alpha,\lambda) \in \Theta_0$. It suffices to show that the determinant
\begin{align} \label{lomax_det}
    \det\bigg( \mathbb{E} \bigg[ M_{\theta}(X)\frac{\partial}{\partial \theta}g(\theta) \bigg]\bigg)=\mathbb{E}\big[\big(f_{\theta}^{(1)}\big)'(X) \big] \mathbb{E}\big[f_{\theta}^{(2)}(X)\big] - \mathbb{E}\big[\big(f_{\theta}^{(2)}\big)'(X) \big] \mathbb{E}\big[f_{\theta}^{(1)}(X)\big],
\end{align}
where $X \sim \mathbb{P}_{\theta_0}$, is always positive. Note first that the derivatives of the optimal functions are given by
\begin{align*}
    \big(f_{\theta}^{(1)}\big)'(x)= \frac{1}{\alpha} \frac{1}{x+\lambda} \quad \text{and} \quad
    \big(f_{\theta}^{(2)}\big)'(x)=-\frac{1}{ (x+ \lambda)^2}.
\end{align*}
Then, by applying dominated and monotone convergence to each expectation in \eqref{lomax_det}, we obtain that the determinant diverges to $\infty$ as $\lambda \rightarrow 0$ and converges to $0$ as $\lambda \rightarrow \infty$. After examining the signs of each expectation, we conclude that the determinant is monotonically decreasing in $\lambda$. Moreover, the sign of the determinant is independent of $\alpha$, since we have dependence only through $f_{\theta}^{(1)}$, $\big(f_{\theta}^{(1)}\big)'$ and $\alpha>0$. The latter reasoning is true for all $(\alpha_0,\lambda_0)\in \Theta$, and thus Assumption \ref{assumptions_iteratice_proc}(c) is verified. In summary, we can apply Theorem \ref{theorem_sequence_stein_estimator} and state that the iteratively defined sequence of Stein estimators is converging to the MLE with probability converging to $1$ as the sample size grows. For the sake of completeness, we also give the asymptotic variance of the MLE, which is the inverse of the Fisher information matrix and given by
\begin{align*}
I_{\mathrm{ML}}^{-1}(\alpha,\lambda)=
(\alpha +1)
\begin{pmatrix}
(\alpha +1)\alpha^2 & \lambda(\alpha +2)\alpha \\
\lambda(\alpha +2)\alpha & \frac{\lambda^2(\alpha +1)(\alpha +2)}{\alpha}
\end{pmatrix}.
\end{align*}
We briefly consider the moment estimators (which are only consistent if $\alpha>2$ and are obtained through the test functions $f_1(x)= x$ and $f_2(x)=x^2$) given by
\begin{align*}
    \hat{\alpha}_n^{\mathrm{MO}}=\frac{2\overline{X}^2-2\overline{X^2}}{2\overline{X}^2-\overline{X^2}}  \quad \text{and} \quad 
    \hat{\lambda}_n^{\mathrm{MO}}=\frac{\overline{X} \, \overline{X^2}}{\overline{X^2}-2\overline{X}^2}.
\end{align*}
It can be seen directly from Lemma \ref{lemma_lomax_existence_mle} that the MLE exists if and only if $ \hat{\lambda}_n^{\mathrm{MO}}$ is positive. The asymptotic variance of $\hat{\theta}_n^{\mathrm{MO}}$ can be computed explicitly and is given by
\begin{align*}
\frac{\alpha(\alpha -1)^2}{(\alpha -3)(\alpha -4)}
\begin{pmatrix}
(\alpha -2)(6+(\alpha -1)\alpha) & \lambda(4+(\alpha -2)\alpha) \\
\lambda(4+(\alpha -2)\alpha) & \frac{\lambda^2(4+(\alpha -3)\alpha) }{\alpha-2}
\end{pmatrix}.
\end{align*}
Note that the latter formula is only valid for $\alpha>4$. A further discussion on fitting the Lomax distribution is available in \cite{labban20192}, although the authors perform simulation studies with contaminated data.

\subsection{Nakagami distribution} \label{subsection_nakagami_distribution}
The density of the Nakagami distribution $NG(m,O)$ with $\theta=(m,O)$, $m, O >0$, is given by
\begin{align*}
p_{\theta}(x)=\frac{2m^m}{\Gamma(m)O^m} x^{2m-1} \exp\bigg(-\frac{m}{O} x^2 \bigg), \quad x >0.
\end{align*}
We take $\tau_{\theta}(x)=O x$ which yields the Stein operator
\begin{align*}
\mathcal{A}_{\theta}f(x)=2m(O - x^2)f(x) +x O f'(x)
\end{align*}
(see \cite{gauntlaplace2021}). The Stein estimators for test functions $f_1$, $f_2$ are given by
\begin{align*}
\hat{m}_n&= \frac{1}{2} \frac{\overline{X^2f_2(X)} \ \overline{Xf_1'(X)}-\overline{X^2f_1(X)} \ \overline{Xf_2'(X)}}{\overline{f_2(X)} \ \overline{X^2f_1(X)}- \overline{f_1(X)} \ \overline{X^2f_2(X)}}, \\
\hat{O}_n&=\frac{\overline{X^2f_2(X)} \ \overline{Xf_1'(X)}-\overline{X^2f_1(X)} \ \overline{Xf_2'(X)}}{\overline{f_2(X)} \ \overline{Xf_1'(X)}- \overline{f_1(X)}) \ \overline{Xf_2'(X)}}.
\end{align*}
Regarding other estimation methods, we consider the moment estimator, which we denote by $\hat{\theta}_n^{\mathrm{MO1}}=(\hat{m}_n^{\mathrm{MO1}},\hat{O}_n^{\mathrm{MO1}})$, and give the first two moments
\begin{align} \label{nakagami_moment_estimator}
\begin{split}
    \mathbb{E}[X]= \frac{\Gamma(m+\frac{1}{2})}{\Gamma(m)} \bigg(\frac{O}{m} \bigg)^{1/2} \quad \text{and} \quad
    \mathbb{E}[X^2]=O.
\end{split}
\end{align}
As it can be easily seen from \eqref{nakagami_moment_estimator}, $\hat{m}_n^{\mathrm{MO1}}$ is not explicit and requires a numerical procedure in order to solve the non-linear equation. Note that we are not able to retrieve these estimators with specific test functions. This can be readily seen since our estimators are always explicit, regardless of the choice of test functions. However, the asymptotic variance of the moment estimator can be calculated explicitly and is given by the matrix $V_{\mathrm{MO1}}(\theta)$ with entries
\begin{gather*}
\big(V_{\mathrm{MO1}}(\theta)\big)_{11}= \frac{m \big(4 m^2+(m+1) ((m)_{\nicefrac{1}{2}}){}^2-4 (m)_{\nicefrac{1}{2}} (m)_{\nicefrac{3}{2}}\big)}{((m)_{\nicefrac{1}{2}}){}^2
   \big(2 m \psi(m)-2 m \psi(m+\frac{1}{2})+1\big)^2}, \quad \big(V_{\mathrm{MO1}}(\theta)\big)_{22}=\frac{(m+1) O ^2}{m}, \\
  \big(V_{\mathrm{MO1}}(\theta)\big)_{12}=\big(V_{\mathrm{MO1}}(\theta)\big)_{21}= \frac{m O }{-2 m \psi(m)+2 m \psi (m+\frac{1}{2})-1}, 
\end{gather*}
where $(a)_{b}=\Gamma(a+1)/\Gamma(a-b+1)$ denotes the Pochhammer symbol. In \cite{artyushenko2019nakagami}, the moment estimators 
\begin{align} \label{nakagami_modified_moment_estimators}
\begin{split}
\hat{m}_n^{\mathrm{MO2}}= \frac{(\overline{X^2})^2 }{\overline{X^4} - (\overline{X^2})^2} \quad \text{and} \quad
\hat{O}_n^{\mathrm{MO2}}=\overline{X^2}
\end{split}
\end{align}
are proposed (which are obtained through the test functions $f_1(x)=1$, $f_2(x)=x^2$ in our approach). However, the authors do not calculate the asymptotic variance, which is easily computed through Theorem \ref{theorem_asymptotic_normality}, and is given by
\begin{align*}
V_{\mathrm{MO2}}(\theta)=
\begin{pmatrix}
2m(m+1) & 0 \\
0 & \frac{O^2}{m}
\end{pmatrix}.
\end{align*}
In  \cite{zhao2021closed}  a generalized Nakagami distribution is used to  derive the  moment type estimator 
\begin{align} \label{nakagami_modified_moment_estimators2}
\hat{m}_n^{\mathrm{MO3}}= \frac12 \frac{{\overline{X^2}} }{\overline{X^2 \log(X) } - {\overline{X^2} \, \overline{\log X}}}, \quad  \hat O_n = \overline{X^2}
\end{align}
whose  asymptotic variance is 
\begin{align*}
    V_{\mathrm{MO3}}(\theta) = \begin{pmatrix}
        m^2(1 + m \psi'(m+1)) & 0 \\
        0 & \frac{O^2}{m}
    \end{pmatrix}.
\end{align*}  
The MLE $\hat{\theta}_n^{\mathrm{ML}}=(\hat{m}_n^{\mathrm{ML}},\hat{O}_n^{\mathrm{ML}})$ is described  by
\begin{align} \label{nakagami_MLE}
\begin{split}
     \log \hat{m}_n^{\mathrm{ML}} - \psi(\hat{m}_n^{\mathrm{ML}})&=\log \overline{X^2} - 2 \overline{\log X}, \\
    \hat{O}_n^{\mathrm{ML}}&= \overline{X^2},
\end{split}
\end{align}
and is unique and exists almost surely. To see that, note that the likelihood function admits exactly one critical point at $(\hat{m}_n^{\mathrm{ML}},\hat{O}_n^{\mathrm{ML}})$, which can be identified to be a local maximum by the second derivative test. Although the likelihood is not necessarily concave, we know that for each $m>0$ the function $O \mapsto \sum_{i=1}^n \log p_{\theta}(X_i)$ has a global maximum at $O =\hat{O}_n^{\mathrm{ML}} $ and the function $m \mapsto \sum_{i=1}^n \log p_{(m,\hat{O}_n^{\mathrm{ML}})}(X_i)$ has a global maximum at $m =\hat{m}_n^{\mathrm{ML}} $, which yields that the local maximum of the likelihood function is also a global maximum. In \cite{kolar2004estimator}, multiple algorithms for solving the likelihood equation are compared, which boils down to a comparison of approximations of the digamma function. Since we deal with a regular probability distribution, we know that the asymptotic covariance matrix of the MLE is given by the inverse of the Fisher information matrix, which is equal to
\begin{align*}
I_{\mathrm{ML}}^{-1}(\theta)=
\begin{pmatrix}
\frac{m}{m\psi(m)-1} & 0 \\
0 & \frac{O^2}{m}
\end{pmatrix}.
\end{align*}
Regarding the Stein estimator,  two of the immediate choices of test functions are already covered:  $f_1(x) = 1$ and $f_2(x) = x^2$ leads to \eqref{nakagami_modified_moment_estimators}  and $f_1(x) = 1$ and $f_2(x) = \log(x)$ leads to \eqref{nakagami_modified_moment_estimators2}. In light of the discussion to be had in  Section \ref{section_discussionnaka},     
 we propose to use the test functions $f_1(x)= 1$ and $f_2(x)= x$, which yield the new estimators
\begin{align} \label{nakagami_stein_estimators}
\hat{m}_n^{\mathrm{ST}}= \frac{1}{2} \frac{\overline{X^2} \ \overline{X}}{\overline{X^3} - \overline{X} \ \overline{X^2}} \quad \text{and} \quad
\hat{O}_n^{\mathrm{ST}}=\overline{X^2}.
\end{align}
An application of a multivariate version of Jensen's inequality gives that both estimates in \eqref{nakagami_stein_estimators} are always positive, whereas the (modified) moment estimators \eqref{nakagami_modified_moment_estimators} may not exist for small sample sizes. With Theorem \ref{theorem_asymptotic_normality}, we compute the asymptotic variance of the Stein estimator, and obtain $V_{\mathrm{ST}}$ with entries
\begin{align*}
\big(V_{\mathrm{ST}}(\theta)\big)_{11}&=\frac{m(5+4m)\Gamma(1+m)^2}{\Gamma(\nicefrac{1}{2}+m)^2}-m(1+2m)^2, \\
\big(V_{\mathrm{ST}}(\theta)\big)_{12}&=\big(V_{\mathrm{ST}}(\theta)\big)_{21}=0, \quad \big(V_{\mathrm{ST}}(\theta)\big)_{22}=\frac{O^2}{m}.
\end{align*}
Lastly, we consider the two-step Stein estimator, which we denote by $\hat{\theta}_n^{\mathrm{ST2}}=(\hat{m}_n^{\mathrm{ST2}},\hat{O}_n^{\mathrm{ST2}})$. The optimal functions are given by
\begin{align*}
    f_{\theta}^{(1)}(x)&= \frac{1}{2m^3x^{2m}} \bigg(\frac{O}{m}\bigg)^{m-1} \bigg( \bigg( \frac{mx^2}{O} \bigg)^m \bigg( m- \exp\bigg(\frac{mx^2}{O} \bigg) \, _2F_2\bigg(m ,m ;m +1,m +1;-\frac{mx^2}{O}\bigg)  \bigg) \\
    & \quad+m^2 \exp\bigg(\frac{mx^2}{O} \bigg) \gamma\bigg(m,\frac{mx^2}{O} \bigg)  \bigg( \log\bigg(\frac{mx^2}{O} \bigg) - \psi(m) \bigg) \bigg) , \\
    f_{\theta}^{(2)}(x)&=-\frac{1}{2O^2}.
\end{align*}
Once again the derivatives can be expressed in terms of the original functions. We obtain
\begin{align*}
    \big(f_{\theta}^{(1)}\big)'(x)=\frac{1-\frac{x^2}{O}+\log\big(\frac{mx^2}{O}\big)-\psi(m)}{O x}-m\bigg(\frac{2}{x}-\frac{2x}{O} \bigg)f_{\theta}^{(1)}(x).
\end{align*}
Analogous to the gamma distribution in Section \ref{appa}, the necessary assumptions for Theorems \ref{theorem_consistency_two_step} and \ref{theorem_efficiency_optimal_functions} can be verified and the latter two theorems are applicable for a suitable first-step estimator, which we choose to be the explicit Stein estimator $\hat{\theta}_n^{\mathrm{ST}}$. By studying the corresponding asymptotic covariance matrix, it is evident that the moment estimators \eqref{nakagami_moment_estimator} perform poorly, which is why we excluded the latter below. In Figure \ref{fig:nakagami_variance}, the asymptotic variances of the (modified) moment estimators \eqref{nakagami_modified_moment_estimators} and \eqref{nakagami_modified_moment_estimators2}, the MLE and the Stein estimator $\hat{m}_n^{\mathrm{ST}}$ of $m$ are plotted for a range of values of $m_0$. Note that the estimators for $O$ coincide for all considered estimation methods. We observe that we improve on the (modified) moment estimator $\hat{\theta}_n^{\mathrm{MO2}}$ in terms of asymptotic variance but not on $\hat{\theta}_n^{\mathrm{MO3}}$, and neither do we  reach efficiency. We also performed a simulation study, whose results are reported in Tables \ref{nakagami_sim_n20} and \ref{nakagami_sim_n50}. Here, we also included the asymptotically efficient two-step Stein estimator $\hat{\theta}_n^{\mathrm{ST2}}$ (with first-step estimator $\hat{\theta}_n^{\mathrm{ST}}$). The (modified) moment estimator  $\hat{\theta}_n^{\mathrm{MO2}}$ served as an initial guess for the MLE. We notice that the MLE seems to be globally the best, followed by the modified moment estimator $\hat{\theta}_n^{\mathrm{MO3}}$ for small samples and $\hat{\theta}_n^{\mathrm{ST2}}$ for large samples, as expected. Our new estimators perform similarly well, although  we find that the estimator $\hat{m}_n^{\mathrm{ST2}}$ seems to break down completely for $(m_0,O_0)=(2,5)$ in the small sample case.  More interesting observations are made in Section \ref{section_discussionnaka} from the main text. 

\begin{table} \small
\centering
\begin{tabular}{cc|ccccc|ccccc}
 $\theta_0$ & & \multicolumn{5}{c}{Bias} & \multicolumn{5}{|c}{MSE} \\ \hline 
 & & $\hat{\theta}_n^{\mathrm{MO2}}$ & $\hat{\theta}_n^{\mathrm{MO3}}$ & $\hat{\theta}_n^{\mathrm{ML}}$ & $\hat{\theta}_n^{\mathrm{ST}}$ & $\hat{\theta}_n^{\mathrm{ST2}}$ & $\hat{\theta}_n^{\mathrm{MO2}}$ & $\hat{\theta}_n^{\mathrm{MO3}}$ & $\hat{\theta}_n^{\mathrm{ML}}$ & $\hat{\theta}_n^{\mathrm{ST}}$ & $\hat{\theta}_n^{\mathrm{ST2}}$ \\ \hline
\multirow{1}{*}{$(1,1)$} & $m$  &  $0.269$ & $0.146$ & \hl $ \hl 0.138$ & $0.19$ & $0.183$ & $0.31$ & $0.156$ & \hl $ \hl 0.149$ & $0.211$ & $0.175$  \\ \hline 
\multirow{1}{*}{$(0.8,1)$} & $m$  & $0.242$ & $0.119$ &  \hl $ \hl 0.11$ & $0.164$ & $0.153$ & $0.221$ & $0.099$ &  \hl $ \hl 0.094$ & $0.141$ & $0.153$  \\ \hline 
\multirow{1}{*}{$(1.4,0.8)$} & $m$  & $0.347$ & $0.22$ &  \hl $ \hl 0.214$ & $0.264$ & $0.259$ & $0.573$ & $0.349$ &  \hl $ \hl 0.337$ & $0.428$ & $0.383$  \\ \hline 
\multirow{1}{*}{$(3,5)$} & $m$  & $0.623$ & $0.487$ &  \hl $ \hl 0.48$ & $0.533$ & $0.612$ & $2.26$ & $1.75$ &  \hl $ \hl 1.71$ & $1.93$ & $11.4$\\ \hline 
\multirow{1}{*}{$(3,1)$} & $m$  & $0.632$ & $0.502$ &  \hl $ \hl 0.496$ & $0.545$ & $0.543$ & $2.27$ & $1.79$ &  \hl $ \hl 1.77$ & $1.96$ & $1.91$ \\ \hline 
\multirow{1}{*}{$(2,5)$} & $m$  & $0.467$ & $0.332$ &  \hl $ \hl 0.324$ & $0.378$ & $0.766$ & $1.12$ & $0.793$ &  \hl $ \hl 0.774$ & $0.906$ & $144$ \\ \hline 
\multirow{1}{*}{$(4,0.5)$} & $m$  & $0.825$ & $0.685$ &  \hl $ \hl 0.675$ & $0.732$ & $0.729$ & $4.02$ & $3.38$ &  \hl $ \hl 3.35$ & $3.6$ & $3.55$ \\ \hline 
\multirow{1}{*}{$(8,4)$} & $m$  & $1.49$ & $1.36$ &  \hl $ \hl 1.35$ & $1.4$ & $1.4$ & $14.3$ &  \hl $ \hl 13.1$ &  \hl $ \hl 13.1$ & $13.5$ & $13.5$ \\ \hline 
\multirow{1}{*}{$(3,3)$} & $m$  & $0.649$ & $0.506$ &  \hl $ \hl 0.496$ & $0.555$ & $0.552$ & $2.27$ & $1.76$ &  \hl $ \hl 1.73$ & $1.94$ & $1.88$ \\ \hline 
\multirow{1}{*}{$(0.8,0.2)$} & $m$  & $0.241$ & $0.121$ &  \hl $ \hl 0.115$ & $0.164$ & $0.148$ & $0.223$ & $0.101$ &  \hl $ \hl 0.097$ & $0.144$ & $0.114$ \\ \hline 
\end{tabular} 
\caption{\protect\label{nakagami_sim_n20} Simulation results for the $N(m,O)$ distribution for $n=20$ and $10,000$ repetitions.}
\end{table}

\begin{table} \small
\centering
\begin{tabular}{cc|ccccc|ccccc}
 $\theta_0$ & & \multicolumn{5}{c}{Bias} & \multicolumn{5}{|c}{MSE} \\  \hline 
 & & $\hat{\theta}_n^{\mathrm{MO2}}$ & $\hat{\theta}_n^{\mathrm{MO3}}$ & $\hat{\theta}_n^{\mathrm{ML}}$ & $\hat{\theta}_n^{\mathrm{ST}}$ & $\hat{\theta}_n^{\mathrm{ST2}}$ & $\hat{\theta}_n^{\mathrm{MO2}}$ & $\hat{\theta}_n^{\mathrm{MO3}}$ & $\hat{\theta}_n^{\mathrm{ML}}$ & $\hat{\theta}_n^{\mathrm{ST}}$ & $\hat{\theta}_n^{\mathrm{ST2}}$ \\  \hline
\multirow{1}{*}{$(1,1)$} & $m$  & $0.109$ & $0.053$ &  \hl $ \hl 0.05$ & $0.072$ & $0.07$ & $0.091$ & $0.042$ &  \hl $ \hl 0.04$ & $0.06$ & $0.044$  \\ \hline 
\multirow{1}{*}{$(0.8,1)$} & $m$  & $0.103$ & $0.047$ &  \hl $ \hl 0.043$ & $0.067$ & $0.063$ & $0.069$ & $0.027$ &  \hl $ \hl 0.026$ & $0.043$ & $0.032$ \\ \hline 
\multirow{1}{*}{$(1.4,0.8)$} & $m$  & $0.14$ & $0.084$ &  \hl $ \hl 0.081$ & $0.103$ & $0.102$ & $0.165$ & $0.092$ &  \hl $ \hl 0.088$ & $0.118$ & $0.098$ \\ \hline 
\multirow{1}{*}{$(3,5)$} & $m$  & $0.25$ & $0.187$ &  \hl $ \hl 0.183$ & $0.208$ & $0.213$ & $0.625$ & $0.451$ &  \hl $ \hl 0.439$ & $0.514$ & $0.474$ \\ \hline 
\multirow{1}{*}{$(3,1)$} & $m$  & $0.243$ & $0.18$ &  \hl $ \hl 0.176$ & $0.202$ & $0.2$ & $0.62$ & $0.458$ &  \hl $ \hl 0.447$ & $0.516$ & $0.493$ \\ \hline 
\multirow{1}{*}{$(2,5)$} & $m$  & $0.175$ & $0.115$ &  \hl $ \hl 0.112$ & $0.135$ & $0.234$ & $0.286$ & $0.187$ &  \hl $ \hl 0.183$ & $0.222$ & $13.3$ \\ \hline 
\multirow{1}{*}{$(4,0.5)$} & $m$  & $0.312$ & $0.249$ &  \hl $ \hl 0.245$ & $0.27$ & $0.268$ & $1.05$ & $0.826$ &  \hl $ \hl 0.811$ & $0.906$ & $0.885$ \\ \hline 
\multirow{1}{*}{$(8,4)$} & $m$  & $0.551$ & $0.484$ &  \hl $ \hl 0.48$ & $0.507$ & $0.506$ & $3.74$ & $3.31$ &  \hl $ \hl 3.28$ & $3.46$ & $3.43$ \\ \hline 
\multirow{1}{*}{$(3,3)$} & $m$  & $0.241$ & $0.18$ &  \hl $ \hl 0.177$ & $0.2$ & $0.201$ & $0.623$ & $0.459$ &  \hl $ \hl 0.449$ & $0.517$ & $0.487$ \\ \hline 
\multirow{1}{*}{$(0.8,0.2)$} & $m$  & $0.098$ & $0.043$ &  \hl $ \hl 0.04$ & $0.063$ & $0.056$ & $0.066$ & $0.027$ &  \hl $ \hl 0.026$ & $0.041$ & $0.028$ \\ \hline 
\end{tabular} 
\caption{\protect\label{nakagami_sim_n50} Simulation results for the $N(m,O)$ distribution for $n=50$ and $10,000$ repetitions.}
\end{table}

\begin{figure}[!]
\centering
\captionsetup{width=.95\textwidth}
\includegraphics[width=7.4cm]{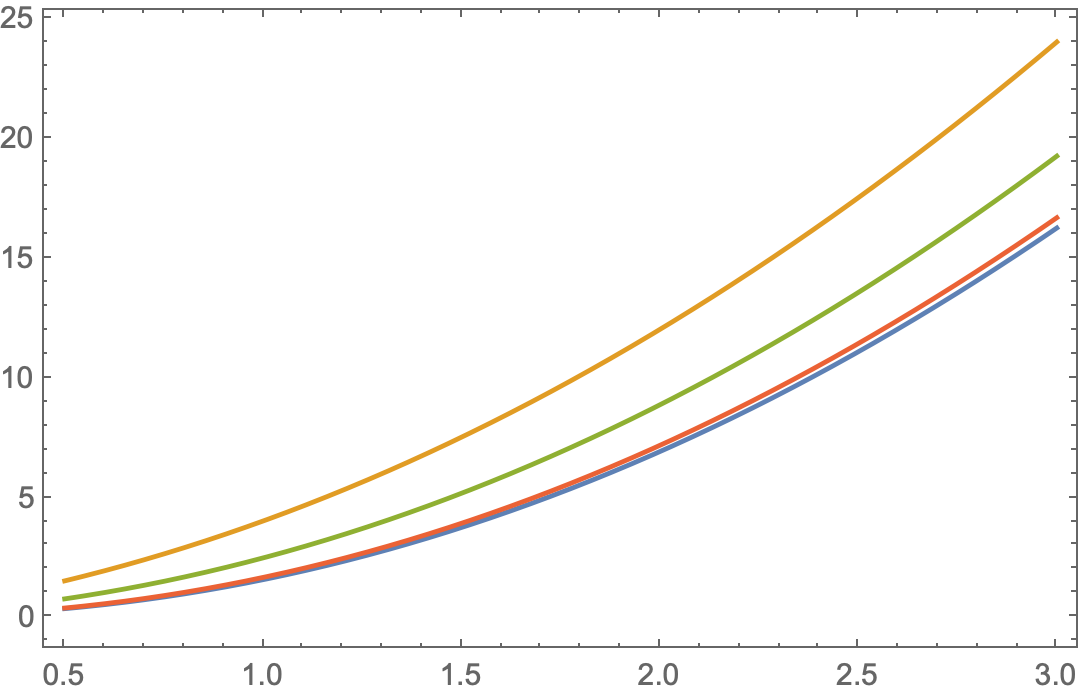}
\caption{\protect\label{fig:nakagami_variance} \it Asymptotic variances of the estimators of the parameter $m$ of the Nakagami distribution. Plotted are the MLE $\hat{m}_n^{\mathrm{ML}}$ \includegraphics[scale=1]{blue}, the moment estimator $\hat{m}_n^{\mathrm{MO2}}$ \includegraphics[scale=1]{orange} \eqref{nakagami_modified_moment_estimators}, moment estimator $\hat{m}_n^{\mathrm{MO3}}$ \includegraphics[scale=1]{red}  \eqref{nakagami_modified_moment_estimators2} and the Stein estimator $\hat{m}_n^{\mathrm{ST}}$ \includegraphics[scale=1]{green} \eqref{nakagami_stein_estimators}. Note the value is independent of $O$.}
\end{figure}

\subsection{One-sided truncated inverse-gamma distribution}
We give another example of how Stein estimators can be used to fit a truncated univariate distibution. The density of the one-sided truncated inverse-gamma distribution on $(a,\infty)$ with $a >0$, denoted by $TIG(\alpha,\beta)$, for $\theta=(\alpha,\beta)$, $\alpha,\beta>0$, is given by
\begin{align*}
    p_{\theta}(x)= \frac{\beta^{\alpha}}{\gamma(\alpha,\beta/a)}  x^{-\alpha-1} \exp\bigg(-\frac{\beta}{x} \bigg), \quad x>a.
\end{align*}
With $\tau_{\theta}(x)=x^2$ we obtain the Stein operator 
\begin{align*}
\mathcal{A}_{\theta}f(x)=(x-\alpha x+\beta)f(x)+x^2f'(x)
\end{align*}
(see also \cite{kl2014}). Considering two test functions $f_1$, $f_2$ yields the estimators
\begin{align*}
\hat{\alpha}_n&=1+\frac{\overline{f_2(X)} \ \overline{X^2f_1'(X)}-\overline{f_1(X)} \ \overline{X^2f_2'(X)}}{\overline{f_2(X)} \ \overline{xf_1(X)}- \overline{f_1(X)}) \ \overline{Xf_2(X)}}, \\
\hat{\beta}_n&=\frac{\overline{X^2f_2'(X)} \ \overline{Xf_1(X)} -\overline{X^2f_1'(X)} \ \overline{Xf_2(X)}}{\overline{f_1(X)}) \ \overline{Xf_2(X)} - \overline{f_2(X)} \ \overline{Xf_1(X)}}.
\end{align*}
Note that coincident with Section \ref{sec_truncnormal}, the normalising constant vanishes. We propose the test functions
\begin{align} \label{truncated_inverse_gamma_stein_estiators}
    f_1(x)= \frac{x-a}{1+(x-a)^2}  \quad \text{and} \quad 
    f_2(x)= \frac{\arctan(x-a)}{(1+x)^2}
\end{align}
and denote the Stein estimator based on the latter two test functions by $\hat{\theta}_n^{\mathrm{ST}}=(\hat{\alpha}_n^{\mathrm{ST}},\hat{\beta}_n^{\mathrm{ST}})$. Regarding other possibilities to estimate $\alpha$ and $\beta$, we contemplate moment estimation, since the first two moments of the $TIG(\alpha,\beta)$-distribution can be calculated and are given by
\begin{align} \label{truncated_inverse_gamma_moment_estimators}
\begin{split}
    \mathbb{E}[X]&=\frac{\beta \gamma(\alpha-1,\beta/a)}{\gamma(\alpha,\beta/a)}, \\
    \mathbb{E}[X^2]&=\frac{\beta^2 \gamma(\alpha-2,\beta/a)}{\gamma(\alpha,\beta/a)}.
\end{split}
\end{align}
Note that the latter expectations only exist for $\alpha>2$. The moment estimator $\hat{\theta}_n^{\mathrm{MO}}=(\hat{\alpha}_n^{\mathrm{MO}},\hat{\beta}_n^{\mathrm{MO}})$ is then defined as the solution for $(\alpha,\beta)$ to \eqref{truncated_inverse_gamma_moment_estimators} with expectations replaced by sample means under the assumption that such a solution exists for a given sample. Moreover, we consider the MLE $\hat{\theta}_n^{\mathrm{ML}}=(\hat{\alpha}_n^{\mathrm{ML}},\hat{\beta}_n^{\mathrm{ML}})$, which is defined as the value for $(\alpha,\beta)$ that maximises the log-likelihood function (again, if such a maximum exists) and is asymptotically efficient in this setting. Due to the high variance of the estimators, we leave out a finite sample simulation study and consider, as in Section \ref{sec_truncnormal}, $q$-confidence regions for two parameter constellations, which are given in Figure \ref{fig:truncinvgamma_variance}. The variance of the moment estimator obtained through \eqref{truncated_inverse_gamma_moment_estimators} is very large and is therefore excluded from the plot. Further, we observe that the Stein estimator \eqref{truncated_inverse_gamma_stein_estiators} is outperformed by the MLE, although variances are still in a reasonable range, not to mention the fact that the Stein estimator is explicit.

\begin{figure}[!]
    \centering
\vspace{.2cm}
\begin{subfigure}{\textwidth}
 \begin{subfigure}{.49\textwidth}
\captionsetup{width=.95\textwidth}
  \centering
  \includegraphics[width=7.4cm]{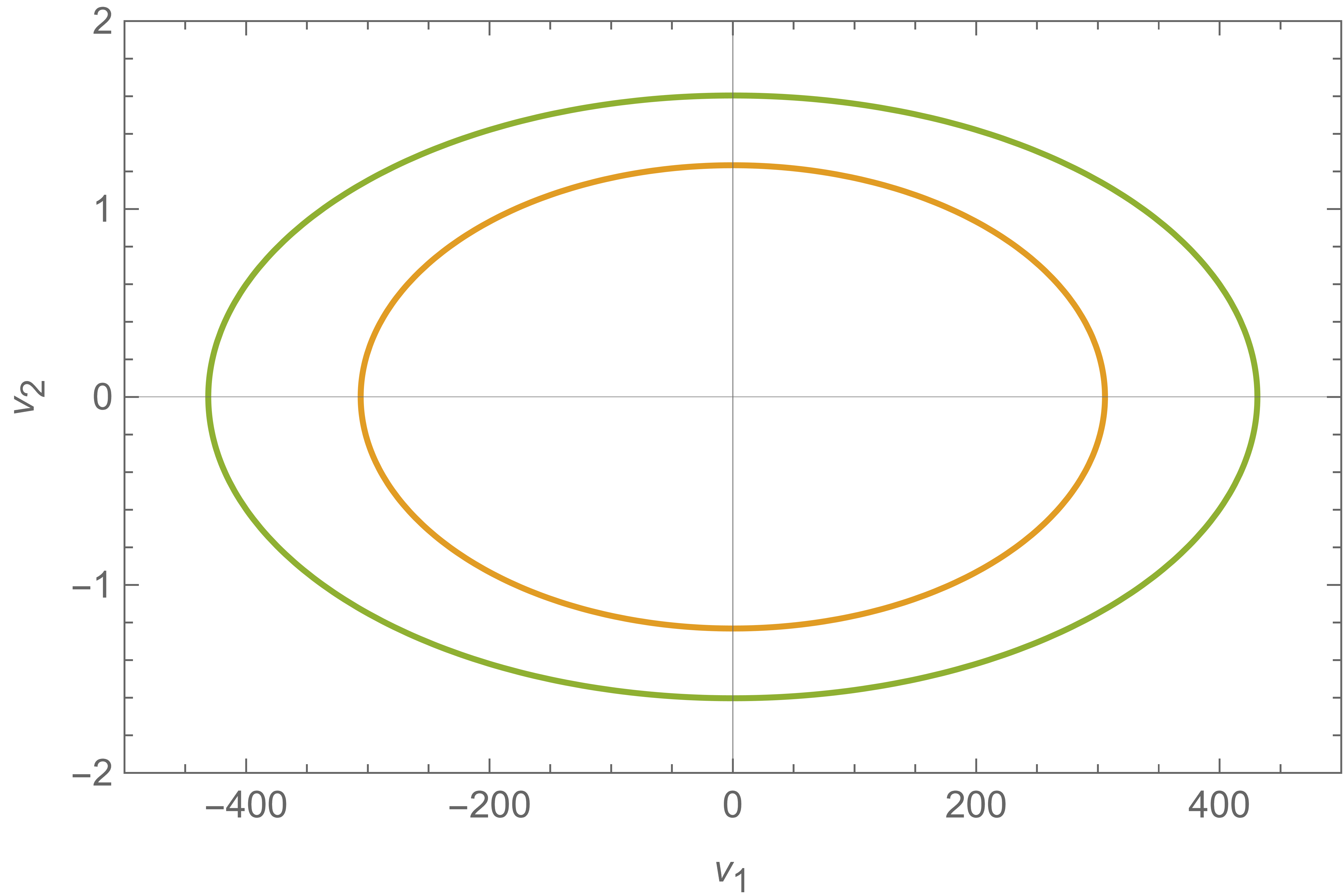}
  \caption{$\alpha=0.5$, $\beta=3$}
\end{subfigure}%
  \begin{subfigure}{.49\textwidth}
\captionsetup{width=.95\textwidth}
  \centering
  \includegraphics[width=7.4cm]{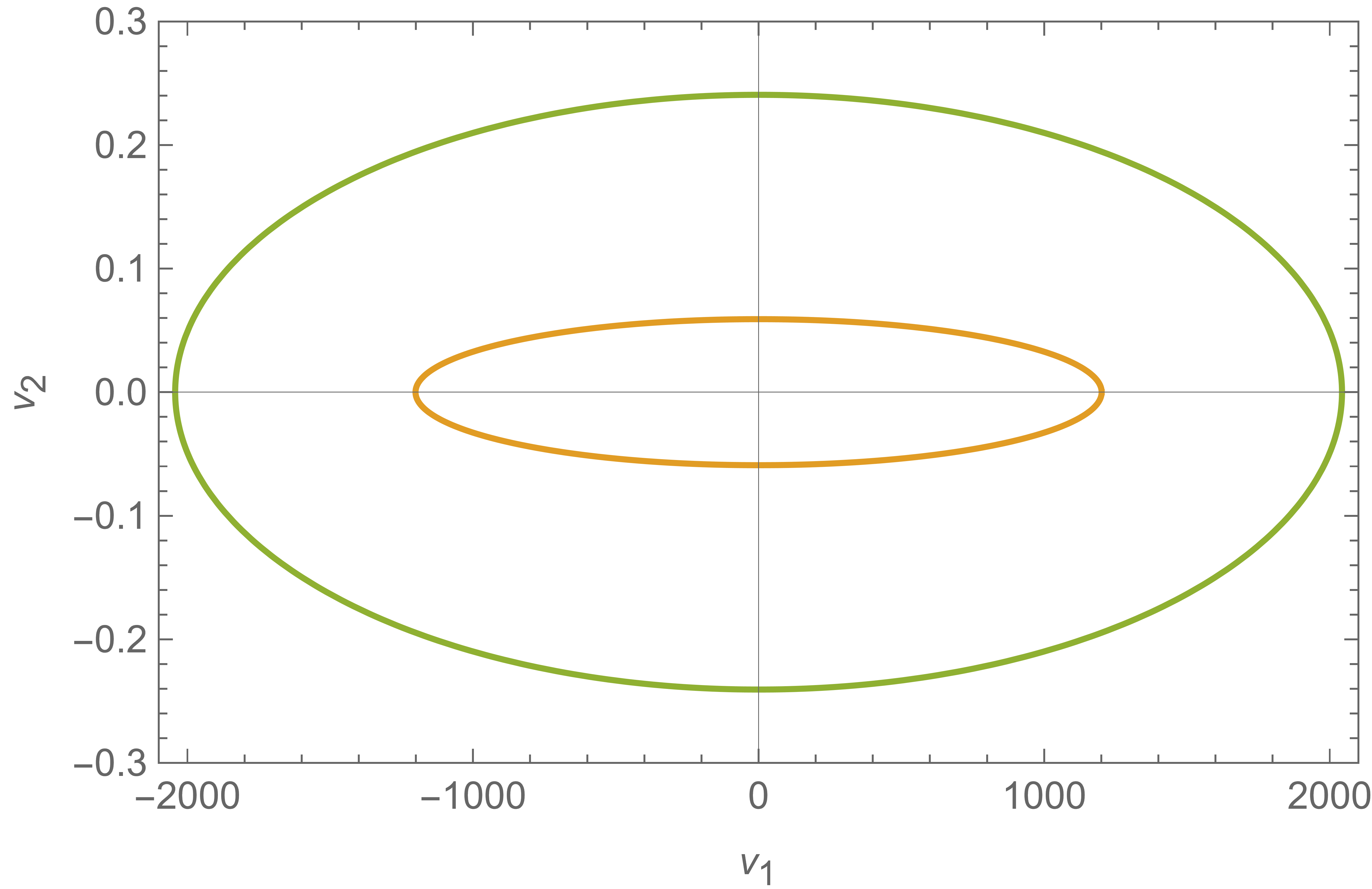}
  \caption{$\alpha=0.1$, $\beta=4$}
\end{subfigure}%
\end{subfigure}%
\caption{\protect\label{fig:truncinvgamma_variance} \it Asymptotic confidence regions for the estimators of the truncated inverse-gamma distribution $TIG(\alpha,\beta)$ for $q=0.95$ and $a=1$ in the directions of the eigenvectors $v_1$ and $v_2$. Plotted are the MLE $\hat{\theta}_n^{\mathrm{ML}}$ \includegraphics[scale=1]{orange} and the Stein estimator $\hat{\theta}_n^{\mathrm{ST}}$ \includegraphics[scale=1]{green}.}
\end{figure}

\subsection{Generalised logistic distribution}
    The density of the generalised logistic distribution $GL(\alpha,\beta)$ with $\theta=(\alpha,\beta)$, $\alpha,\beta>0$, is given by
    \begin{align*}
        p_{\theta}(x)=\frac{1}{B(\alpha,\beta)} \frac{e^{-\beta x}}{(1+e^{-x})^{\alpha+\beta}}, \quad x \in \mathbb{R}.
    \end{align*}
    We choose $\tau_{\theta}(x)=1+e^{x}$ and get
    \begin{align*}
    \mathcal{A}_{\theta}f(x)=\big( \alpha -e^x (\beta-1)\big)f(x) + (1+e^x) f'(x).
\end{align*}
 With two test functions $f_1$, $f_2$ we have the estimators
\begin{align*}
    \hat{\alpha}_n&= \frac{\overline{(1+e^X)f_2'(X)} \ \overline{e^Xf_1(X)} - \overline{(1+e^X)f_1'(X)} \ \overline{e^Xf_2(X)}}{\overline{e^Xf_2(X)} \ \overline{f_1(X)} - \overline{e^Xf_1(X)} \ \overline{f_2(X)}}, \\
    \hat{\beta}_n&= \frac{\Big( \overline{(1+e^X)f_2'(X)} + \overline{e^Xf_2(X)} \Big) \ \overline{f_1(X)} - \Big( \overline{(1+e^X)f_1'(X)} + \overline{e^Xf_1(X)} \Big) \ \overline{f_2(X)} }{\overline{e^Xf_2(X)} \ \overline{f_1(X)} - \overline{e^Xf_1(X)} \ \overline{f_2(X)}}.
\end{align*}
The first two moments are given by
\begin{align*}
    \mathbb{E}[X]=\psi(\alpha)-\psi(\beta) \quad \text{and} \quad \mathbb{E}[X^2]=\psi'(\alpha)+\psi'(\beta) + (\psi(\alpha)-\psi(\beta))^2.
\end{align*}
If existent, the moment estimator $\hat{\theta}_n^{\mathrm{MO}}=(\hat{\alpha}_n^{\mathrm{MO}},\hat{\beta}_n^{\mathrm{MO}})$ is defined as the solution to the empirical versions of the equations above. It is clear that the latter is not explicit. The MLE $\hat{\theta}_n^{\mathrm{ML}}=(\hat{\alpha}_n^{\mathrm{ML}},\hat{\beta}_n^{\mathrm{ML}})$ is defined as the values in $\Theta$, at which the log-likelihood function attains its maximum. Supposing the latter exists and is uniquely determined by the critical point of the derivative, the MLE is the solution to the equations
\begin{align*}
    \overline{\log(1+e^{-X})}=\psi(\alpha)-\psi(\alpha+\beta) \quad \text{and} \quad \overline{X} + \overline{\log(1+e^{-X})}=\psi(\beta)-\psi(\alpha+\beta).
\end{align*}
The Fisher information matrix is given by
\begin{align*}
    I_{\mathrm{ML}}(\alpha, \beta)=
    \begin{pmatrix}
        \psi'(\alpha) - \psi'(\alpha+\beta) & -\psi'(\alpha+\beta) \\
        -\psi'(\alpha+\beta) &  \psi'(\beta)- \psi'(\alpha+\beta)
    \end{pmatrix}.
\end{align*}
Considering the Stein estimator $\hat{\theta}_n^{\mathrm{ST}}=(\hat{\alpha}_n^{\mathrm{ST}}, \hat{\beta}_n^{\mathrm{ST}})$, we choose the test functions $f_1(x)=1/(1+e^x)$ and $f_2(x)=1/(1+e^x)^2$ such that Assumptions \ref{general_assumptions}(a)--(b) are easily verified. The results of the simulation study can be found in Tables \ref{genlog_sim_n20} and \ref{genlog_sim_n50}. In order to calculate the moment estimators $\hat{\alpha}_n^{\mathrm{MO}}$ and $\hat{\beta}_n^{\mathrm{MO}}$, we employed the R function \texttt{nleqslv} from the R package \textit{nleqslv} \cite{hasselman2023nleqslv} with initial values $(1,1)$. The MLE are obtained through numerical optimisation of the log-likelihood function (as well with initial guess $(1,1)$). Interestingly, the MLE -- although asymptotically efficient -- is outperformed by the moment estimators for most parameter constellations for sample sizes $n=20$ and $n=50$. The Stein estimator is competitive, in terms of MSE, mostly residing between the MLE and moment estimator. However, the Stein estimator is the only one among the considered estimation techniques that is completely explicit, and does not seem to generate estimates that lie outside of the parameter space, as can be seen in column \textit{NE}.

\begin{table}[h] \small
\centering
\begin{tabular}{cc|ccc|ccc|ccc}
 $\theta_0$ & & \multicolumn{3}{|c}{Bias} & \multicolumn{3}{|c}{MSE} & \multicolumn{3}{|c}{NE} \\ \hline
 & & $\hat{\theta}_n^{\mathrm{ML}}$ & $\hat{\theta}_n^{\mathrm{MO}}$  & $\hat{\theta}_n^{\mathrm{ST}}$ & $\hat{\theta}_n^{\mathrm{ML}}$ & $\hat{\theta}_n^{\mathrm{MO}}$ & $\hat{\theta}_n^{\mathrm{ST}}$ & $\hat{\theta}_n^{\mathrm{ML}}$  & $\hat{\theta}_n^{\mathrm{MO}}$  & $\hat{\theta}_n^{\mathrm{ST}}$ \\ \hline
\multirow{2}{*}{$(1,1)$} & $\alpha$  & 0.182 & 0.157 & \hl 0.132 & 0.195 & \hl 0.18 & 0.189 & \multirow{2}{*}{0} & \multirow{2}{*}{0} & \multirow{2}{*}{0} \\ & $\beta$  & 0.178 & 0.153 & \hl 0.129 & 0.196 & \hl 0.181 & 0.191\\ \hline 
\multirow{2}{*}{$(0.4,1)$} & $\alpha$  & 0.083 & \hl 0.05 & 0.055 & 0.027 & \hl 0.021 & 0.034 & \multirow{2}{*}{9} & \multirow{2}{*}{0} & \multirow{2}{*}{0} \\ & $\beta$  & 0.311 & 0.19 & \hl 0.18 & 0.442 & \hl 0.268 & 0.35\\ \hline 
\multirow{2}{*}{$(0.8,1.5)$} & $\alpha$  & 0.136 & 0.115 & \hl 0.111 & 0.114 & \hl 0.105 & 0.126 & \multirow{2}{*}{0} & \multirow{2}{*}{0} & \multirow{2}{*}{0} \\ & $\beta$  & 0.314 & 0.259 & \hl 0.239 & 0.572 & \hl 0.49 & 0.551\\ \hline 
\multirow{2}{*}{$(3,1)$} & $\alpha$  & 0.662 & \hl 0.551 & 0.592 & 2.51 & \hl 2.15 & 2.73 & \multirow{2}{*}{0} & \multirow{2}{*}{0} & \multirow{2}{*}{0} \\ & $\beta$  & 0.171 & \hl 0.147 & 0.165 & 0.183 & \hl 0.166 & 0.22\\ \hline 
\multirow{2}{*}{$(0.2,3)$} & $\alpha$  & 0.094 & 0.021 & \hl 0.08 & 0.016 & $ \hl 4\text{\hl e-3}$ & 0.024 & \multirow{2}{*}{81} & \multirow{2}{*}{0} & \multirow{2}{*}{0} \\ & $\beta$  & 3.46 & \hl 1.23 & 2.34 & 44.3 & \hl 10.3 & 28.1\\ \hline 
\multirow{2}{*}{$(4,6)$} & $\alpha$  & 0.71 & 0.686 & \hl 0.651 & 3.39 & 3.34 & \hl 3.33 & \multirow{2}{*}{0} & \multirow{2}{*}{0} & \multirow{2}{*}{0} \\ & $\beta$  & 1.1 & 1.06 & \hl 1.01 & 7.92 & 7.79 & \hl 7.76\\ \hline 
\multirow{2}{*}{$(7,0.5)$} & $\alpha$  & 2.98 & \hl 1.7 & 2.88 & 35.4 & \hl 18.6 & 38.3 & \multirow{2}{*}{21} & \multirow{2}{*}{0} & \multirow{2}{*}{0} \\ & $\beta$  & 0.102 & \hl 0.064 & 0.142 & 0.047 & \hl 0.032 & 0.09\\ \hline 
\multirow{2}{*}{$(0.2,8)$} & $\alpha$  & 0.108 & \hl 0.021 & 0.104 & 0.02 & $ \hl 4\text{\hl e-3}$ & 0.031 & \multirow{2}{*}{90} & \multirow{2}{*}{0} & \multirow{2}{*}{0} \\ & $\beta$  & 10.1 & \hl 3.28 & 7.95 & 333 & \hl 61.5 & 297\\ \hline 
\multirow{2}{*}{$(3,10)$} & $\alpha$  & 0.511 & \hl 0.48 & 0.501 & 1.77 & \hl 1.71 & 1.87 & \multirow{2}{*}{0} & \multirow{2}{*}{0} & \multirow{2}{*}{0} \\ & $\beta$  & 1.85 & \hl 1.73 & 1.79 & 22.6 & \hl 21.7 & 23.6\\ \hline 
\multirow{2}{*}{$(0.5,0.4)$} & $\alpha$  & 0.118 & 0.075 & \hl 0.062 & 0.065 & \hl 0.046 & 0.056 & \multirow{2}{*}{10} & \multirow{2}{*}{0} & \multirow{2}{*}{0} \\ & $\beta$  & 0.083 & 0.051 & \hl 0.044 & 0.029 & \hl 0.024 & 0.032\\ \hline 
\end{tabular} 
\caption{\protect\label{genlog_sim_n20} Simulation results for the $GL(\alpha,\beta)$ distribution for $n=20$ and $10,000$ repetitions.}
\end{table}

\begin{table}[h]  \small
\centering
\begin{tabular}{cc|ccc|ccc|ccc}
 $\theta_0$ & & \multicolumn{3}{|c}{Bias} & \multicolumn{3}{|c}{MSE} & \multicolumn{3}{|c}{NE} \\ \hline
 & & $\hat{\theta}_n^{\mathrm{ML}}$ & $\hat{\theta}_n^{\mathrm{MO}}$  & $\hat{\theta}_n^{\mathrm{ST}}$ & $\hat{\theta}_n^{\mathrm{ML}}$ & $\hat{\theta}_n^{\mathrm{MO}}$ & $\hat{\theta}_n^{\mathrm{ST}}$ & $\hat{\theta}_n^{\mathrm{ML}}$  & $\hat{\theta}_n^{\mathrm{MO}}$  & $\hat{\theta}_n^{\mathrm{ST}}$ \\ \hline
\multirow{2}{*}{$(1,1)$} & $\alpha$  & 0.065 & 0.055 & \hl 0.046 & 0.048 & \hl 0.045 & 0.053 & \multirow{2}{*}{0} & \multirow{2}{*}{0} & \multirow{2}{*}{0} \\ & $\beta$  & 0.064 & 0.053 & 0.044 & 0.048 & 0.046 & 0.054\\ \hline 
\multirow{2}{*}{$(0.4,1)$} & $\alpha$  & 0.03 & \hl 0.018 & 0.02 & $7.16\text{e-3}$ & $ \hl 5.75\text{\hl e-3}$ & 0.01 & \multirow{2}{*}{6} & \multirow{2}{*}{0} & \multirow{2}{*}{0} \\ & $\beta$  & 0.119 & 0.065 & \hl 0.06 & 0.095 & \hl 0.062 & 0.08\\ \hline 
\multirow{2}{*}{$(0.8,1.5)$} & $\alpha$  & 0.051 & 0.042 & \hl 0.04 & 0.029 & \hl 0.027 & 0.035 & \multirow{2}{*}{0} & \multirow{2}{*}{0} & \multirow{2}{*}{0} \\ & $\beta$  & 0.12 & 0.096 & \hl 0.088 & 0.14 & \hl 0.121 & 0.14\\ \hline 
\multirow{2}{*}{$(3,1)$} & $\alpha$  & 0.253 & \hl 0.202 & 0.214 & 0.627 & \hl 0.524 & 0.65 & \multirow{2}{*}{0} & \multirow{2}{*}{0} & \multirow{2}{*}{0} \\ & $\beta$  & 0.064 & \hl 0.052 & 0.059 & 0.047 & \hl 0.042 & 0.058\\ \hline 
\multirow{2}{*}{$(0.2,3)$} & $\alpha$  & 0.038 & $ \hl 6.98\text{\hl e-3}$ & 0.029 & $2.98\text{e-3}$ & $ \hl 1.14\text{\hl e-3}$ & $6.51\text{e-3}$ & \multirow{2}{*}{86} & \multirow{2}{*}{0} & \multirow{2}{*}{0} \\ & $\beta$  & 1.48 & \hl 0.396 & 0.715 & 7.92 & \hl 1.6 & 3.52\\ \hline 
\multirow{2}{*}{$(4,6)$} & $\alpha$  & 0.237 & 0.229 & \hl 0.218 & 0.792 & \hl 0.783 & 0.804 & \multirow{2}{*}{0} & \multirow{2}{*}{0} & \multirow{2}{*}{0} \\ & $\beta$  & 0.369 & 0.356 & \hl 0.337 & 1.85 & \hl 1.82 & 1.85\\ \hline 
\multirow{2}{*}{$(7,0.5)$} & $\alpha$  & 1.09 & \hl 0.581 & 1.03 & 8.3 & \hl 4.25 & 8.44 & \multirow{2}{*}{18} & \multirow{2}{*}{0} & \multirow{2}{*}{0} \\ & $\beta$  & 0.03 & \hl 0.023 & 0.056 & 0.011 & $ \hl 8.57\text{\hl e-3}$ & 0.025\\ \hline 
\multirow{2}{*}{$(0.2,8)$} & $\alpha$  & 0.048 & $ \hl 8.29\text{\hl e-3}$ & 0.041 & $4.97\text{e-3}$ & $ \hl 1.17\text{\hl e-3}$ & $8.38\text{e-3}$ & \multirow{2}{*}{94} & \multirow{2}{*}{0} & \multirow{2}{*}{0} \\ & $\beta$  & 3.99 & \hl 1.21 & 2.6 & 70.2 & \hl 13 & 33.6\\ \hline 
\multirow{2}{*}{$(3,10)$} & $\alpha$  & 0.186 & \hl 0.173 & 0.185 & 0.463 & \hl 0.443 & 0.505 & \multirow{2}{*}{0} & \multirow{2}{*}{0} & \multirow{2}{*}{0} \\ & $\beta$  & 0.666 & \hl 0.618 & 0.652 & 5.8 & \hl 5.52 & 6.16\\ \hline 
\multirow{2}{*}{$(0.5,0.4)$} & $\alpha$  & 0.041 & 0.025 & \hl 0.021 & 0.013 & \hl 0.011 & 0.016 & \multirow{2}{*}{6} & \multirow{2}{*}{0} & \multirow{2}{*}{0} \\ & $\beta$  & 0.031 & 0.018 & \hl 0.015 & $7.09\text{e-3}$ & $ \hl 6.22\text{\hl e-3}$ & $9.7\text{e-3}$\\ \hline 
\end{tabular} 
\caption{\protect\label{genlog_sim_n50} Simulation results for the $GL(\alpha,\beta)$ distribution for $n=50$ and $10,000$ repetitions.}
\end{table}

\subsection{A non-i.i.d.\ example} \label{cauchy_noniid}
We revisit Section \ref{cauchysec} and construct a sequence of strictly stationary, ergodic but non-i.i.d.\  random variables whose marginal probability distribution is $C(\mu,\gamma)$. Let $\epsilon_t$, $t \in \mathbb{Z}$, be a sequence of i.i.d.\ random variables with $\epsilon_t \sim C(\mu,\gamma)$ for all $t \in \mathbb{Z}$. We then define 
\begin{align*}
    X_n= \frac{1}{q+1} \bigg( \sum_{i=1}^q \epsilon_{n-i} + \epsilon_n \bigg), \quad n \in \mathbb{Z},
\end{align*}
and conclude that the marginal distribution of each $X_n$ is $C(\mu,\gamma)$ since the mean of i.i.d.\ Cauchy random variables is again Cauchy. Observe that $\{X_n,n\in \mathbb{Z}\}$ can be seen as a $\mathrm{MA}(q)$-process with i.i.d.\ Cauchy white noise. It is obvious that the sequence $X_n$, $n \in \mathbb{Z}$, is strictly stationary. We show that it is also ergodic. It is clear that $X_n$, $n \in \mathbb{Z}$, is $m$-dependent, i.e.\ $X_i$ and $X_j$ are independent if $\vert i-j \vert \geq m$ with $m=q+1$. Therefore, $X_{n+mj}$, $j \in \mathbb{Z}$ is i.i.d. for any $n \in \mathbb{Z}$. In the sequel, we shall write $\sigma(\cdot)$ for the generated $\sigma$-algebra. Let $A \in \mathcal{A}$ be an invariant event belonging to $\sigma(X_1,X_2,\ldots)$, i.e. $\zeta(A)=A$. Then we have that 
\begin{align*}
A \in \bigcap_{n \geq 1} \sigma(X_n,X_{n+1},\ldots) \subset \sigma\bigg( \bigcup_{j=1}^{m-1} \bigcap_{n \geq 0} \sigma(X_{(n+k)m+j},k \geq 0) \bigg).
\end{align*}
Now for each $j=0,\ldots,m-1$ we can interpret $\mathcal{A}_j=\bigcap_{n \geq 0} \sigma(X_{(n+k)m+j},k \geq 0)$ as the terminal $\sigma$-algebra of an i.i.d.\ process and with Kolmogorov's $0$-$1$ law we conclude $\mathbb{P}(A') \in\{0,1\} $ for each element $A' \in \mathcal{A}_j$ for any $j=1,\ldots,m-1$. Since $A \in \sigma\big( \bigcup_{j=1}^{m-1} \mathcal{A}_j \big)$ we can write $A$ as a finite intersection/union of elements in $\bigcup_{j=1}^{m-1} \mathcal{A}_j$. We conclude $\mathbb{P}(A) \in \{0,1\}$. We can apply the same argument to an invariant set $A \in \sigma(X_0,X_{-1},\ldots)$ and hence obtain ergodicity. \par
We include pseudo-MLE estimation in our simulation study, that is the maximum of the log-likelihood function of the data $X_1,\ldots,X_n$ as if it was i.i.d. Hence the estimator is defined through
\begin{align*}
    \hat{\theta}_n^{\mathrm{ML}} = \argmax_{\theta \in \Theta} \sum_{i=1}^n \log p_{\theta}(X_i). 
\end{align*}
Note that this procedure still yields a consistent estimator since the ergodic theorem also applies to the score function but the estimator is in general not asymptotically efficient. \par
Next, we consider the median estimator $\hat{\theta}_n^{\mathrm{MED}}$, which is defined as
\begin{align*}
    \hat{\mu}_n^{\mathrm{MED}}=\mathrm{Med}(X_1,\dots,X_n), \quad \hat{\gamma}_n^{\mathrm{MED}}=\mathrm{Med}\big(\vert X_1-  \hat{\mu}_n^{\mathrm{MED}} \vert,\ldots,\vert X_n-  \hat{\mu}_n^{\mathrm{MED}} \vert\big),
\end{align*}
where $\mathrm{Med}(\cdot)$ denotes the empirical median of a sample $X_1,\ldots, X_n$. The estimator for $\gamma$ is motivated by the fact that the mode of $\vert X \vert$ is equal to $\gamma$, where $X \sim C(0,\gamma)$. \par
We use the Stein estimator $\hat{\theta}_n^{\mathrm{ST}}$, as worked out in Section \ref{cauchysec}, with the optimal functions from \eqref{optimal_func_cauchy}. As a first step estimate we employ $\hat{\theta}_n^{\mathrm{MED}}$ (note that the latter is consistent and asymptotically normal with rate $\sqrt{n}$). With Remark \ref{remark_extension_ergodic_optimal functions} the Theorems \ref{theorem_consistency_two_step} and \ref{theorem_efficiency_optimal_functions} are still applicable and we deduce consistency as well as asymptotic normality. \par
Simulation results are reported in Table \ref{cauchynoniid_sim}, whereby we chose the same parameter constellations as in Section \ref{cauchysec}. We observe that $\hat{\theta}_n^{\mathrm{ST}}$ significantly outperforms $\hat{\theta}_n^{\mathrm{MED}}$ and shows a very similar behaviour to $\hat{\theta}_n^{\mathrm{ML}}$.

\begin{table}[h] \small
\centering
\begin{tabular}{cc|ccc|ccc}
 $\theta_0$ & & \multicolumn{3}{c|}{Bias} & \multicolumn{3}{c}{MSE} \\  \hline
 & & $\hat{\theta}_n^{\mathrm{ML}}$ & $\hat{\theta}_n^{\mathrm{MED}}$  & $\hat{\theta}_n^{\mathrm{ST}}$ & $\hat{\theta}_n^{\mathrm{ML}}$ & $\hat{\theta}_n^{\mathrm{MED}}$ & $\hat{\theta}_n^{\mathrm{ST}}$ \\ \hline
\multirow{2}{*}{$(-5,1)$} & $\mu$  & $2.55\text{e-3}$ & $ \hl 2.42\text{\hl e-3}$ & $2.79\text{e-3}$ & $ \hl 0.07$ & $0.085$ & $0.071$ \\ & $\gamma$ & $ \hl 7.87\text{\hl e-3}$ & $0.026$ & $7.97\text{e-3}$ & $ \hl 0.07$ & $0.078$ & $0.072$ \\ \hline 
\multirow{2}{*}{$(-4,1.5)$} & $\mu$  & $ \hl 1.15\text{\hl e-3}$ & $-1.27\text{e-3}$ & $1.25\text{e-3}$ & $ \hl 0.155$ & $0.192$ & $0.159$ \\ & $\gamma$ & $ \hl 0.023$ & $0.046$ & $0.024$ & $ \hl 0.156$ & $0.173$ & $0.16$ \\ \hline 
\multirow{2}{*}{$(-2,2)$} & $\mu$  & $8.54\text{e-3}$ & $ \hl 7.03\text{\hl e-3}$ & $9.96\text{e-3}$ & $ \hl 0.273$ & $0.34$ & $0.279$ \\ & $\gamma$ & $0.021$ & $0.052$ & $0.022$ & $0.28$ & $0.31$ & $0.285$ \\ \hline 
\multirow{2}{*}{$(0,1)$} & $\mu$  & $ \hl -6.51\text{\hl e-3}$ & $-7.74\text{e-3}$ & $-6.82\text{e-3}$ & $ \hl 0.069$ & $0.084$ & $0.07$ \\ & $\gamma$ & $ \hl 0.015$ & $0.033$ & $0.016$ & $ \hl 0.071$ & $0.079$ & $0.073$ \\ \hline 
\multirow{2}{*}{$(0,3)$} & $\mu$  & $-2.51\text{e-3}$ & $-3.51\text{e-3}$ & $ \hl -1.26\text{\hl e-3}$ & $ \hl 0.628$ & $0.762$ & $0.641$ \\ & $\gamma$ & $ \hl 0.023$ & $0.071$ & $0.024$ & $ \hl 0.613$ & $0.681$ & $0.626$ \\ \hline 
\multirow{2}{*}{$(2,0.1)$} & $\mu$  & $3.33\text{e-4}$ & $ \hl 2.81\text{\hl e-4}$ & $3.68\text{e-4}$ & $ \hl 6.85\text{\hl e-4}$ & $8.35\text{e-4}$ & $7.01\text{e-4}$ \\ & $\gamma$ & $ \hl 1.57\text{\hl e-3}$ & $3.28\text{e-3}$ & $1.61\text{e-3}$ & $ \hl 7.28\text{\hl e-4}$ & $8.15\text{e-4}$ & $7.47\text{e-4}$ \\ \hline 
\multirow{2}{*}{$(2,0.5)$} & $\mu$  & $-1.35\text{e-3}$ & $ \hl -7.11\text{\hl e-4}$ & $-1.22\text{e-3}$ & $ \hl 0.018$ & $0.022$ & $ \hl 0.018$ \\ & $\gamma$ & $ \hl 7.1\text{\hl e-3}$ & $0.016$ & $7.17\text{e-3}$ & $ \hl 0.018$ & $0.02$ & $ \hl 0.018$ \\ \hline 
\multirow{2}{*}{$(4,0.8)$} & $\mu$  & $-7.25\text{e-4}$ & $-1.01\text{e-3}$ & $ \hl -6.78\text{\hl e-4}$ & $ \hl 0.045$ & $0.055$ & $0.046$ \\ & $\gamma$ & $ \hl 0.012$ & $0.026$ & $0.013$ & $ \hl 0.045$ & $0.05$ & $0.047$ \\ \hline 
\multirow{2}{*}{$(6,2.3)$} & $\mu$  & $ \hl 5.6\text{\hl e-3}$ & $9.03\text{e-3}$ & $6.03\text{e-3}$ & $ \hl 0.373$ & $0.452$ & $0.379$ \\ & $\gamma$ & $ \hl 0.023$ & $0.06$ & $ \hl 0.023$ & $ \hl 0.363$ & $0.403$ & $0.372$ \\ \hline 
\multirow{2}{*}{$(10,0.2)$} & $\mu$  & $ \hl 3\text{\hl e-4}$ & $4.3\text{e-5}$ & $3.09\text{e-4}$ & $ \hl 2.74\text{\hl e-3}$ & $3.37\text{e-3}$ & $2.81\text{e-3}$ \\ & $\gamma$ & $ \hl 1.42\text{\hl e-3}$ & $4.92\text{e-3}$ & $1.43\text{e-3}$ & $ \hl 2.7\text{\hl e-3}$ & $3.05\text{e-3}$ & $2.74\text{e-3}$ \\ \hline 
\end{tabular} 
\caption{\protect\label{cauchynoniid_sim} Simulation results for the $\mathrm{MA}(q)$-process with marginal $C(\mu,\gamma)$ distribution for $n=150$, $q=5$ and $10,000$ repetitions.}
\end{table}

\section{Proofs}

\subsection{Proof of Theorem \ref{theorem_stein_operator}}

\begin{proof}
The proof follows along the lines of \cite[Theorem 2.2]{ley2013stein}. To see sufficiency, note that
\begin{align*}
    \mathbb{E}[\mathcal{A}_{\theta}f(X)]= \int_a^b \mathcal{A}_{\theta}f(x) p_{\theta}(x)\,dx = \int_a^b \big( f(x) \tau_{\theta}(x) p_{\theta}(x)\big)'\,dx =0,
\end{align*}
where we used that $f \in \mathscr{F}_{\theta}$ in the last step. For the necessity, we define $l_z(x)=1_{(a,z)}(x) - P_{\theta}(z)$ for $z, x \in (a,b)$,
where $P_{\theta}$ is the CDF with respect to $\mathbb{P}_{\theta}$. We obviously have $\mathbb{E}[l_z(X)]=0$ for all $z \in (a,b)$ if $X \sim \mathbb{P}_{\theta}$. Then the function
\begin{align*}
    f_z(x)=\frac{1}{p_{\theta}(x) \tau_{\theta}(x)} \int_a^x l_z(u)p_{\theta}(u)\,du
\end{align*}
belongs to $\mathscr{F}_{\theta}$ for all $z \in (a,b)$ and satisfies the equation $\mathcal{A}_{\theta} f_z(x)=l_z(x)$ for all $x,z \in (a,b)$. Hence, for each $z \in (a,b)$ we have
\begin{align*}
    0= \mathbb{E}[l_z(X)] =\mathbb{E}[\mathcal{A}_{\theta} f_z(X)] = \mathbb{P}(X \leq z)-P_{\theta}(z),
\end{align*}
and it follows that $X \sim \mathbb{P}_{\theta}.$
\end{proof}

\subsection{Proof of Theorem \ref{theorem_consistency}}

\begin{proof}
Let $X \sim \mathbb{P}_{\theta_0}$. Define the function $F(M,\theta)=Mg(\theta)$, where $M \in \mathbb{R}^{p \times q}$ and $g$ is as in Assumption \ref{general_assumptions}(b). Then $F$ is continuously differentiable on $\mathbb{R}^{p \times q} \times \Theta$. By Assumption \ref{general_assumptions}(a), we have $F(\mathbb{E}[M(X)],\theta_0)=0$ and by Assumption \ref{general_assumptions}(b) the Jacobian $\mathbb{E}[M(X)]\frac{\partial}{\partial \theta} g(\theta) \vert_{\theta=\theta_0}$ is invertible. Now the implicit function theorem implies that there are neighbourhoods $U \subset \mathbb{R}^{p \times q}$, $V \subset \mathbb{R}^p$ of $\mathbb{E}[M(X)]$ and $\theta_0$ such that there exists a continuously differentiable function $h:U \rightarrow V$ with $F(M,h(M))=0$ for all $M \in U$. By an application of the ergodic theorem (see for example \cite[Theorem 20.14]{klenke2013probability}),
\begin{align*}
    \frac{1}{n} \sum_{i=1}^n M(X_i) \overset{\mathrm{a.s.}}{\longrightarrow} \mathbb{E}[M(X)].
\end{align*}
Thus, by defining $A_n=\big\{n^{-1} \sum_{i=1}^n M(X_i) \in U \big\}$ we have $\mathbb{P}(A_n) \rightarrow 1$ as $n \rightarrow \infty$ and a solution $\hat{\theta}_n$ to \eqref{stein_equation_system} (in $V$) exists for each $\omega \in A_n$ since $\hat{\theta}_n=h\big(n^{-1} \sum_{i=1}^n M(X_i)\big)$. Measurability is clear and the consistency part now follows from the continuous mapping theorem.
\end{proof}

\begin{Remark} \label{remark_hansen_cons}
Although our estimators can be seen as generalised method of moments estimators, in our setting it is not convenient to apply the consistency proof of \cite{hansen1982large} (see \cite{hansen2012proofs} for a more concrete derivation), which is based on a compactification of the parameter space. To see this, we revisit Example \ref{example_gaussian} and choose $g(\theta)=(\mu,\sigma^2,1)^\top$. Let $\Theta^\star$ be the closure of the bounded set
 \begin{align*}
     \Theta'=\bigg\{ \frac{g(\theta)}{\sqrt{1+\Vert g(\theta) \Vert^2}} \ \vert \ \theta \in \Theta \bigg\}.
 \end{align*}
Then $\Theta^\star =[-1,1] \times [0,1] \times [0,1/\sqrt{2}]$. We have
\begin{align*}
    M(x)= 
\begin{pmatrix}
    f_1(x) & f_1'(x) & -xf_1(x) \\
    f_2(x) & f_2'(x) & -xf_2(x)
\end{pmatrix},
\end{align*}
and now $\mathbb{E}[M(X)\theta]=0$ for $X \sim \mathbb{P}_{\theta_0}$ and some $\theta \in \Theta^\star$ does not imply $\theta=g(\theta_0)/\sqrt{1+\Vert g(\theta_0) \Vert^2}$, since $0 \in \Theta^\star$.
\end{Remark}

\subsection{Proof of Theorem \ref{theorem_asymptotic_normality}}

We remind the reader of the Kronecker product for two matrices $M \in \mathbb{R}^{p_1 \times q_2}$ and $N \in \mathbb{R}^{p_2 \times q_2} $ defined by
\begin{align*}
    M \otimes N =
    \begin{pmatrix}
        m_{11} N & \ldots & m_{1q_1} N \\
        \vdots & \ddots & \vdots \\
        m_{p_11} N & \ldots & m_{p_1q_1} N
    \end{pmatrix} \in \mathbb{R}^{p_1p_2 \times q_1q_2}.
\end{align*}

\noindent \emph{Proof of Theorem \ref{theorem_asymptotic_normality}.} We remark that it is enough to show weak convergence of the probability measures that are obtained when conditioning on $A_n$, since for any measurable set $B \subset \mathbb{R}^p$ we have
\begin{align*}
    \mathbb{P}(\hat{\theta}_n \in B) =  \mathbb{P}(\hat{\theta}_n \in B \, \vert \, A_n)  \mathbb{P}(A_n) + \mathbb{P}(\hat{\theta}_n \in B \, \vert \, \bar{A}_n)  \mathbb{P}(\bar{A}_n),
\end{align*}
where $\bar{A}$ denotes the complement of the set $A$ in $\mathbb{R}^p$. The second part of the last expression converges to $0$ and $ \mathbb{P}(A_n)$ converges to $1$, therefore it is enough to show that the random vector $\hat{\theta}_n \, \vert \, A_n$ ($\hat{\theta}_n$ conditioned on the events $A_n$) is asymptotically normal. On the sets $A_n$, we have $\hat{\theta}_n=h\big(n^{-1} \sum_{i=1}^n M(X_i)\big)$, where the function $h$ is defined in the proof of Theorem \ref{theorem_consistency}. By stacking up the columns, we consider the function $h$ to be defined on $\mathbb{R}^{pq}$. Theorem \ref{ergodic_central_limit_theorem} yields that
\begin{align*}
    \sqrt{n}\bigg(\frac{1}{n} \sum_{i=1}^n \mathrm{vec}(M(X_i)) - \mathbb{E}[\mathrm{vec}(M(X))]\bigg) \stackrel{D}{\longrightarrow} N(0,\Lambda),
\end{align*}
and 
\begin{align*}
\Lambda=\sum_{j \in \mathbb{Z}}\mathbb{E}\Big[\big(\mathrm{vec}(M(X_0))-\mathbb{E}[\mathrm{vec}(M(X))]\big)\big(\mathrm{vec}(M(X_{j}))-\mathbb{E}[\mathrm{vec}(M(X))]\big)^{\top}\Big]. 
\end{align*}
Next we calculate the derivative of $h$, which, by the implicit function theorem, is given by
\begin{align*}
    h'(\mathrm{vec}(M))= -\bigg( M\frac{\partial}{\partial \theta} g(\theta) \bigg)^{-1} M g(\theta) = -\bigg( M\frac{\partial}{\partial \theta} g(\theta) \bigg)^{-1} ((g_1(\theta), \ldots, g_q(\theta)) \otimes \mathrm{I}_{p})\mathrm{vec}(M).
\end{align*}
Since $n^{-1} \sum_{i=1}^n \mathrm{vec}(M(X_i))$ converges to $\mathbb{E}[\mathrm{vec}(M(X))]$ in probability and $h(\mathbb{E}[\mathrm{vec}(M(X))])=\theta_0$, an application of the delta method yields 
\begin{align*}
    \sqrt{n}(\hat{\theta}_n-\theta_0) \stackrel{D}{\longrightarrow} N\big(0,G^{-1} \big( (g_1(\theta), \ldots, g_q(\theta)) \otimes I_d \big) \, \Lambda  \,\big((g_1(\theta), \ldots, g_q(\theta)) \otimes I_d \big)^\top  G^{-\top}\big),
\end{align*}
where $G$ is defined in the statement of the theorem. Noting that 
\begin{align*}
    \big( (g_1(\theta), \ldots, g_q(\theta)) \otimes I_d \big) \mathrm{vec}(M(x)) = M(x)g(\theta)=\mathcal{A}_{\theta}f(x)
\end{align*}
gives the claim. \hfill $\Box$

\begin{Remark}
    Recall that $\hat{\theta}_n$ can be seen as a generalised method of moments estimator. The asymptotic normality result above could also be obtained through an application of \cite[Theorem 3.1]{hansen1982large}. For that, one has to show that the function $x \mapsto \frac{\partial}{\partial \theta }\mathcal{A}_{\theta}f(x)$ is first moment continuous, that is
    \begin{align*}
        \lim_{\delta \rightarrow 0} \mathbb{E}\bigg[ \sup_{\theta' \in \Theta} \bigg\{\bigg\Vert \frac{\partial}{\partial \theta }\mathcal{A}_{\theta}f(X) \Big\vert_{\theta=\theta'} -\frac{\partial}{\partial \theta}\mathcal{A}_{\theta}f(X)\Big\vert_{\theta=\theta_0} \bigg\Vert \, \bigg\vert \, \Vert \theta'-\theta_0 \Vert < \delta \bigg\} \bigg]=0
    \end{align*}
    for $X \sim \mathbb{P}_{\theta_0}$, which is clear with Assumption \ref{general_assumptions}(b) by choosing a submultiplicative matrix norm $\Vert \cdot \Vert$ in the formula above. However, the setting in \cite{hansen1982large} is more general and therefore requires more effort than needed in our proof above in order to establish the result.
\end{Remark}

\subsection{Proof of Theorem \ref{theorem_consistency_two_step}}

\begin{proof}
Let $X \sim \mathbb{P}_{\theta_0}$. Define the function $F(M,\theta)=Mg(\theta)$, where $M \in \mathbb{R}^{p \times q}$ and $g$ is defined as in Assumption \ref{assumptions_optimal_func}(c). Then $F$ is continuously differentiable on $\mathbb{R}^{p \times q} \times \Theta$ and we have $F(\mathbb{E}[M_{\theta_0}(X)],\theta_0)=0$ by Assumption \ref{assumptions_optimal_func}(b), where $M_{\theta}(x)$ is defined as in Assumption \ref{assumptions_optimal_func}(c). The implicit function theorem implies that there are neighbourhoods $U \subset \mathbb{R}^{p \times q},$ $V \subset \mathbb{R}^p$ of $\mathbb{E}[M_{\theta_0}(X)]$ and $\theta_0$ such that there is a continuously differentiable function $h:U \rightarrow V$ with $F(M,h(M))=0$ for all $M \in U$. Assumption \ref{assumptions_optimal_func}(a) ensures that, for a compact set $K \subset \Theta$ with $\theta_0 \in K$ and a set $B_n=\{\omega \in \Omega \, \vert \, \tilde{\theta}_n \in K\}$, we have $\mathbb{P}(B_n) \rightarrow 1$. By a uniform strong law of large numbers (see for example \cite[Theorem 16(a)]{ferguson2017course}) together with Assumption \ref{assumptions_optimal_func}(d) we know that 
\begin{align*}
    \sup_{\theta \in K} \bigg\Vert \frac{1}{n} \sum_{i=1}^n M_{\theta}(X_i) - \mathbb{E}[M_{\theta}(X)]  \bigg\Vert \overset{\mathrm{a.s.}}{\longrightarrow} 0
\end{align*}
(take $K=\Theta'$). Moreover, we know that $\mathbb{E}[M_{\tilde{\theta}_n}(X)] \overset{\mathbb{P}}{\longrightarrow} \mathbb{E}[M_{\theta_0}(X)]$, where the first expectation is taken with respect to $X$. Hence, we can conclude that $n^{-1}\sum_{i=1}^n M_{\tilde{\theta}_n}(X_i)$ is converging in probability to $\mathbb{E}[M_{\theta_0}(X)]$ with respect to the probability measure conditioned on the sets $B_n$. Thus, by defining $A_n=\big\{ n^{-1} \sum_{i=1}^n M_{\tilde{\theta}_n}(X_i) \in U \big\}$, we have $\mathbb{P}(A_n \, \vert \, B_n) \rightarrow 1$ as $n \rightarrow \infty$. Therefore, we have shown existence and  measurability for each $\omega \in A_n \cap B_n$ with $\mathbb{P}(A_n \cap B_n) \rightarrow 1$, since we have  $\hat{\theta}_n^\star=h(n^{-1} \sum_{i=1}^n M_{\tilde{\theta}_n}(X_i))$. The consistency part follows with $C_n=A_n \cap B_n$.
\end{proof}

\subsection{Proof of Theorem \ref{theorem_efficiency_optimal_functions}}

\begin{proof}
Throughout the proof, let $X \sim \mathbb{P}_{\theta_0}$ be a random variable independent of all randomness involved. We observe that it suffices to consider the random variable $\hat{\theta}_n^\star \, \vert \, C_n$ ($\hat{\theta}_n^\star$ conditioned on the events $C_n$). By a multivariate Taylor expansion with explicit remainder term we have
\begin{align*}
    \sqrt{n}(\hat{\theta}_n^\star - \theta_0) &= -\bigg( \int_0^1 \frac{1}{n} \sum_{i=1}^n \frac{\partial}{\partial \theta} \mathcal{A}_{\theta} f_{a_t(\tilde{\theta}_n)}(X_i) \Big\vert_{\theta=a_t(\hat{\theta}_n^\star)}\, dt \bigg)^{-1} \\
    & \times \bigg( \frac{1}{\sqrt{n}} \sum_{i=1}^n \mathcal{A}_{\theta_0} f_{\theta_0}(X_i)  + \int_0^1  \frac{1}{n} \sum_{i=1}^n \frac{\partial}{\partial \theta} \mathcal{A}_{a_t(\hat{\theta}_n^\star)} f_{\theta}(X_i) \Big\vert_{\theta=a_t(\tilde{\theta}_n)}\, dt \sqrt{n}(\tilde{\theta}_n - \theta_0) \bigg),
\end{align*}
where $a_t(\theta)=\theta_0+ t(\theta-\theta_0)$. Both integrals above have to be understood component-wise. We show
\begin{align} \label{proof_two_step_stoch_conv}
    \int_0^1 \frac{1}{n} \sum_{i=1}^n \frac{\partial}{\partial \theta} \mathcal{A}_{\theta} f_{a_t(\tilde{\theta}_n)}(X_i) \Big\vert_{\theta=a_t(\hat{\theta}_n^\star)} \,dt \overset{\mathbb{P}}{\longrightarrow} \mathbb{E} \bigg[ \frac{\partial}{\partial \theta} \mathcal{A}_{\theta} f_{\theta_0}(X) \Big\vert_{\theta=\theta_0} \bigg].
\end{align}
With Theorem \ref{theorem_consistency_two_step} we know that $\hat{\theta}_n^\star$ is consistent. We assume without loss of generality that $\hat{\theta}_n^\star$ and $\tilde{\theta}_n$ fall in some compact set whose interior contains $\theta_0$, since for any compact neighbourhood $K \subset \mathbb{R}^p$ of $\theta_0$ we have $\mathbb{P}(\{\hat{\theta}_n^\star \notin K \} \cup \{\tilde{\theta}_n \notin K\}) \rightarrow 0$. Take $K=\Theta' \cap \Theta''$. We write
\begin{align} \label{proof_two_step_inequality}
    \bigg\Vert \int_0^1 \frac{1}{n} \sum_{i=1}^n & \frac{\partial}{\partial \theta} \mathcal{A}_{\theta} f_{a_t(\tilde{\theta}_n)}(X_i) \Big\vert_{\theta=a_t(\hat{\theta}_n^\star)}\, dt - \mathbb{E} \bigg[ \frac{\partial}{\partial \theta} \mathcal{A}_{\theta} f_{\theta_0}(X) \Big\vert_{\theta=\theta_0} \bigg] \bigg\Vert \nonumber \\
    \begin{split}  
    \leq & \sup_{\theta_1,\theta_2 \in K} \bigg\Vert \frac{1}{n} \sum_{i=1}^n \frac{\partial}{\partial \theta} \mathcal{A}_{\theta} f_{\theta_2}(X_i) \Big\vert_{\theta=\theta_1} - \mathbb{E} \bigg[ \frac{\partial}{\partial \theta} \mathcal{A}_{\theta} f_{\theta_2}(X) \Big\vert_{\theta=\theta_1} \bigg\Vert \\
    & + \bigg\Vert \int_0^1 \mathbb{E} \bigg[ \frac{\partial}{\partial \theta} \mathcal{A}_{\theta} f_{a_t(\tilde{\theta}_n)}(X) \Big\vert_{\theta=a_t(\hat{\theta}_n^\star)} \bigg]\, dt - \mathbb{E} \bigg[ \frac{\partial}{\partial \theta} \mathcal{A}_{\theta} f_{\theta_0}(X) \Big\vert_{\theta=\theta_0} \bigg] \bigg\Vert.
    \end{split}
\end{align}
The first expression on the right-hand side converges to $0$ in probability by a uniform strong law of large numbers. To see this, note that we have 
\begin{align*}
   \bigg\Vert \frac{\partial}{\partial \theta} \mathcal{A}_{\theta} f_{\theta_2}(x)\Big\vert_{\theta=\theta_1} \bigg\Vert =\bigg\Vert M_{\theta_2}(x)\frac{\partial}{\partial \theta}g(\theta)\Big\vert_{\theta=\theta_1} \bigg\Vert \leq F_1(x) \sup_{\theta \in K} \bigg\Vert \frac{\partial}{\partial \theta}g(\theta) \bigg\Vert
\end{align*}
for all $\theta_1,\theta_2 \in K$. The supremum is finite by Assumption \ref{assumptions_optimal_func}(c) and the expectation exists by Assumption \ref{assumptions_optimal_func}(d). For the second expression in \eqref{proof_two_step_inequality}, we note that the map
\begin{align*}
    (\theta_1,\theta_2) \mapsto \int_0^1 \mathbb{E} \bigg[ \frac{\partial}{\partial \theta} \mathcal{A}_{\theta} f_{a_t(\theta_2)}(X) \Big\vert_{\theta=a_t(\theta_1)} \bigg]\, dt
\end{align*}
is continuous on the set $K \times K$ by Assumptions \ref{assumptions_optimal_func}(c),(d) and dominated convergence. We therefore have, due to the consistency of $\hat{\theta}_n^\star$ and $\tilde{\theta}_n$, that
\begin{align*}
     \int_0^1 \mathbb{E} \bigg[ \frac{\partial}{\partial \theta} \mathcal{A}_{\theta} f_{a_t(\tilde{\theta}_n)}(X) \Big\vert_{\theta=a_t(\hat{\theta}_n^\star)} \bigg]\, dt \overset{\mathbb{P}}{\longrightarrow} \mathbb{E} \bigg[ \frac{\partial}{\partial \theta} \mathcal{A}_{\theta} f_{\theta_0}(X) \Big\vert_{\theta=\theta_0} \bigg],
\end{align*}
and \eqref{proof_two_step_stoch_conv} is established. We calculate the latter expectation and and obtain with Assumptions (iii), (iv) and with the definition of $f_{\theta_0}$
\begin{align*}
   &  \mathbb{E} \bigg[ \frac{\partial}{\partial \theta} \mathcal{A}_{\theta} f_{\theta_0}(X) \Big\vert_{\theta=\theta_0} \bigg]\\
   &\quad= \mathbb{E} \bigg[ \frac{\frac{\partial}{\partial \theta} \big(p_{\theta}(X) \tau_{\theta}(X) f_{\theta_0}(X) \big)' \big\vert_{\theta=\theta_0}}{p_{\theta_0}(X)} - \frac{\frac{\partial}{\partial \theta} p_{\theta}(X) \big\vert_{\theta=\theta_0} \big(p_{\theta_0}(X) \tau_{\theta_0}(X) f_{\theta_0}(X) \big)'}{\big(p_{\theta_0}(X)\big)^2}  \bigg] \\
    &\quad=  \int_a^b \frac{\partial}{\partial \theta}  \big(p_{\theta}(x) \tau_{\theta}(x) f_{\theta_0}(x) \big)' \Big\vert_{\theta=\theta_0}\, dx - I(\theta_0)\\
    &\quad= - I(\theta_0).
\end{align*}
One can show in a similar manner that we have
\begin{align*}
      \int_0^1 \frac{1}{n} \sum_{i=1}^n \frac{\partial}{\partial \theta} \mathcal{A}_{a_t(\hat{\theta}_n^\star)} f_{\theta}(X_i) \Big\vert_{\theta=a_t(\tilde{\theta}_n)}\, dt \overset{\mathbb{P}}{\longrightarrow} \mathbb{E} \bigg[ \frac{\partial}{\partial \theta} \mathcal{A}_{\theta_0} f_{\theta}(X) \Big\vert_{\theta=\theta_0} \bigg].
\end{align*}
With Assumption \ref{assumptions_optimal_func}(d) we are allowed to switch expectation and differentiation to obtain
\begin{align*}
     \mathbb{E} \bigg[ \frac{\partial}{\partial \theta} \mathcal{A}_{\theta_0} f_{\theta}(X) \Big\vert_{\theta=\theta_0} \bigg]=0,
\end{align*}
since $f_{\theta} \in \mathscr{F}_{\theta}$ for all $\theta \in \Theta$. The claim follows with Slutsky's theorem, the central limit theorem, Assumptions (i) and (iv) and the continuous mapping theorem.
\end{proof}

\subsection{Proof of Theorem \ref{theorem_sequence_stein_estimator}}

\begin{proof}
Let $X \sim \mathbb{P}_{\theta_0}$ and let $L_n:\Theta \rightarrow \Theta$ be the map which maps $\hat{\theta}_n^{(m)}$ to  $\hat{\theta}_n^{(m+1)}$. We show that $L_n$ is a contraction on $\Theta_0$ with probability converging to $1$. With Assumption \ref{assumptions_iteratice_proc}(d) we know that
\begin{align} \label{proof_mle_uniform_conv}
    \sup_{\theta \in \Theta_0} \bigg\Vert \frac{1}{n} \sum_{i=1}^n M_{\theta}(X_i) - \mathbb{E}[M_{\theta}(X)] \bigg\Vert \overset{\mathrm{a.s.}}{\longrightarrow} 0.
\end{align}
Suppose without loss of generality that $\Theta_0 \neq \{\theta_0\}$ and let $V \subset \Theta_0$ be an open neighbourhood of $\theta_0$. Assumption \ref{assumptions_iteratice_proc}(c) implies in a similar way as in the proof of Theorem \ref{theorem_consistency} that for each $\theta \in \Theta_0$ there exists a unique continuously differentiable function $h_{\theta}:U_{\theta}\rightarrow V$ where $U_{\theta} \subset \mathbb{R}^{p \times q}$ is an open neighbourhood of $\mathbb{E}[M_{\theta}(X)]$ such that $F(M, h_{\theta}(M))=0$ for all $M \in U_{\theta}$ where $ F(M,\theta)= Mg(\theta), M \in \mathbb{R}^{p \times q}, \theta \in \mathbb{R}^p$. Since $\Theta_0$ is compact, we find a finite cover $U_{\theta_1},\ldots,U_{\theta_b}$ from $\cup_{\theta \in \Theta_0} U_{\theta}$ of the set $\{ \mathbb{E}[M_{\theta}(X)], \, \theta \in \Theta_0 \}$. Let $\tilde{U}=\cup_{i=1}^b U_{\theta_i} $ and we conclude the existence of a continuously differentiable function $h:\tilde{U} \rightarrow V$ with $F(M,h(M))=0$ for all $M \in \tilde{U}$. We define $B_n \subset \Omega$ by $ B_n= \big\{\frac{1}{n} \sum_{i=1}^n M_{\theta}(X_i)   \in \tilde{U} \text{ for all } \theta \in \Theta_0 \big\}$ and have $\mathbb{P}(B_n) \rightarrow 1$ by \eqref{proof_mle_uniform_conv}. Then, on $ B_n$ we know that $L_n(\theta)=h\big(\frac{1}{n} \sum_{i=1}^n M_{\theta}(X_i)\big)$, $\theta \in \Theta_0$ and therefore $L_n(\Theta_0) \subset \Theta_0$. \par
We use an argument similar to \cite[Lemma 3]{dominitz2005some}. With a Taylor expansion we have, for $\theta_1,\theta_2 \in \Theta_0$ and $\omega \in B_n$, that
\begin{align*}
 &\Vert  L_n(\theta_1) - L_n(\theta_2) \Vert  \\
 &\leq \bigg\Vert \int_0^1 \bigg(\frac{1}{n} \sum_{i=1}^n M_{a_t}(X_i) \frac{\partial}{\partial \theta} g(\theta) \Big\vert_{\theta=L_n(a_t)}\bigg)^{-1} \bigg(\frac{1}{n} \sum_{i=1}^n \frac{\partial}{\partial \theta} M_{\theta}(X_i)  g(L_n(a_t)) \Big\vert_{\theta=a_t}\bigg) \,dt \bigg\Vert \Vert \theta_1 - \theta_2 \Vert,
\end{align*}
where $a_t=\theta_1+t(\theta_2-\theta_1)$. The integral has to be understood component-wise. We denote the matrix-valued integral above by $K_n^{(1)}$ and define
\begin{align*}
    K_n^{(2)}=\int_0^1 \bigg(\mathbb{E} \bigg[ M_{a_t}(X) \frac{\partial}{\partial \theta} g(\theta) \Big\vert_{\theta=L_n(a_t)} \bigg]\bigg)^{-1} \bigg(\frac{1}{n} \sum_{i=1}^n \frac{\partial}{\partial \theta} M_{\theta}(X_i)  g(L_n(a_t)) \Big\vert_{\theta=a_t}\bigg)\, dt,
\end{align*}
where the expectation is taken with respect to the random variable $X$. Then we have
\begin{align*}
    \Vert K_n^{(1)} - K_n^{(2)} \Vert \leq & \int_0^1 \bigg\Vert \bigg( \mathbb{E} \bigg[ M_{a_t}(X) \frac{\partial}{\partial \theta} g(\theta) \Big\vert_{\theta=L_n(a_t)} \bigg] \bigg)^{-1} \bigg\Vert \bigg\Vert \frac{1}{n} \sum_{i=1}^n M_{a_t}(X_i) \frac{\partial}{\partial \theta} g(\theta) \Big\vert_{\theta=L_n(a_t)} \\
    &- \mathbb{E} \bigg[ M_{a_t}(X) \frac{\partial}{\partial \theta} g(\theta) \Big\vert_{\theta=L_n(a_t)} \bigg] \bigg\Vert \bigg\Vert \bigg( \frac{1}{n} \sum_{i=1}^n M_{a_t}(X_i) \frac{\partial}{\partial \theta} g(\theta) \Big\vert_{\theta=L_n(a_t)} \bigg)^{-1} \bigg\Vert \\
    & \times \bigg\Vert  \frac{1}{n} \sum_{i=1}^n \frac{\partial}{\partial \theta} M_{\theta}(X_i)  g(L_n(a_t)) \Big\vert_{\theta=a_t} \bigg\Vert \,dt.
\end{align*}
This expression can be further bounded by
\begin{align*}
     & \sup_{\gamma_1, \gamma_2 \in \Theta_0} \bigg\Vert \bigg( \mathbb{E} \bigg[ M_{\gamma_1}(X) \frac{\partial}{\partial \theta} g(\theta) \Big\vert_{\theta=\gamma_2} \bigg] \bigg)^{-1} \bigg\Vert \\
     & \times \sup_{\gamma_1,\gamma_2 \in \Theta_0} \bigg\Vert \frac{1}{n} \sum_{i=1}^n M_{\gamma_1}(X_i) \frac{\partial}{\partial \theta} g(\theta) \Big\vert_{\theta=\gamma_2} -   \mathbb{E} \bigg[ M_{\gamma_1}(X) \frac{\partial}{\partial \theta} g(\theta) \Big\vert_{\theta=\gamma_2} \bigg] \bigg\Vert \\
     & \times \sup_{\gamma_1,\gamma_2 \in \Theta_0} \bigg\Vert \bigg( \frac{1}{n} \sum_{i=1}^n M_{\gamma_1}(X_i) \frac{\partial}{\partial \theta} g(\theta) \Big\vert_{\theta=\gamma_2} \bigg)^{-1} \bigg\Vert \\
    & \times \sup_{\gamma_1,\gamma_2 \in \Theta_0} \bigg\Vert  \frac{1}{n} \sum_{i=1}^n \frac{\partial}{\partial \theta} M_{\theta}(X_i)  g(\gamma_2) \Big\vert_{\theta=\gamma_1} \bigg\Vert \\
    &=V^{(1)} \times V_n^{(2)} \times V_n^{(3)} \times V_n^{(4)}.
\end{align*}
We have $V^{(1)}<\infty$ by Assumptions \ref{assumptions_iteratice_proc}(c),(d). Moreover, $V_n^{(2)}$ converges almost surely to $0$ due to Assumption \ref{assumptions_iteratice_proc}(d). This implies as well that there exists a sequence of sets $C_n$ with $\mathbb{P}_{\theta_0}(C_n) \rightarrow 1$ such that $n^{-1} \sum_{i=1}^n M_{\gamma_1}(X_i) \frac{\partial}{\partial \theta} g(\theta) \vert_{\theta=\gamma_2}$ is invertible for all $\omega \in C_n$ and $\gamma_1,\gamma_2 \in \Theta_0$. We now tackle $V_n^{(3)}$. An application of a standard inequality for the spectral norm yields
\begin{align*}
    V_n^{(3)} & \leq  \sup_{\gamma_1,\gamma_2 \in \Theta_0} \bigg\vert \det\bigg(  \frac{1}{n} \sum_{i=1}^n M_{\gamma_1}(X_i) \frac{\partial}{\partial \theta} g(\theta) \Big\vert_{\theta=\gamma_2} \bigg)\bigg\vert^{-1} \bigg\Vert \frac{1}{n} \sum_{i=1}^n M_{\gamma_1}(X_i) \frac{\partial}{\partial \theta} g(\theta) \Big\vert_{\theta=\gamma_2} \bigg\Vert^{p-1} \\
   & \leq  \bigg( \inf_{\gamma_1,\gamma_2 \in \Theta_0}  \bigg\vert \det\bigg(  \frac{1}{n} \sum_{i=1}^n M_{\gamma_1}(X_i) \frac{\partial}{\partial \theta} g(\theta) \Big\vert_{\theta=\gamma_2} \bigg)\bigg\vert \bigg)^{-1}\\
   &\quad\times\bigg( \sup_{\gamma_3,\gamma_4 \in \Theta_0}  \bigg\Vert \frac{1}{n} \sum_{i=1}^n M_{\gamma_3}(X_i) \frac{\partial}{\partial \theta} g(\theta) \Big\vert_{\theta=\gamma_4} \bigg\Vert \bigg)^{p-1}.
\end{align*}
For $\omega \in C_n$ we have that the infimum above is always strictly positive. The supremum can be bounded as follows:
\begin{align*}
   & \sup_{\gamma_3,\gamma_4 \in \Theta_0}  \bigg\Vert \frac{1}{n} \sum_{i=1}^n M_{\gamma_3}(X_i) \frac{\partial}{\partial \theta} g(\theta) \Big\vert_{\theta=\gamma_4} \bigg\Vert \\
   &\quad\leq  \sup_{\gamma_1,\gamma_2 \in \Theta_0} \bigg\Vert \frac{1}{n} \sum_{i=1}^n M_{\gamma_1}(X_i) \frac{\partial}{\partial \theta} g(\theta) \Big\vert_{\theta=\gamma_2} -   \mathbb{E} \bigg[ M_{\gamma_1}(X) \frac{\partial}{\partial \theta} g(\theta) \Big\vert_{\theta=\gamma_2} \bigg] \bigg\Vert\\
  &\quad\quad + \sup_{\gamma_1,\gamma_2 \in \Theta_0} \bigg\Vert \mathbb{E} \bigg[ M_{\gamma_1}(X) \frac{\partial}{\partial \theta} g(\theta) \Big\vert_{\theta=\gamma_2} \bigg] \bigg\Vert,
\end{align*}
where the first term converges to $0$ almost surely and the second term is bounded by Assumptions \ref{assumptions_iteratice_proc}(c),(d). We conclude that $V_n^{(3)}$ is bounded with probability converging to $1$ (take $\omega \in B_n \cap C_n$). We can proceed with $V_n^{(4)}$ in a similar way. This entails that for all $\omega \in B_n \cap C_n$ we have $ \Vert K_n^{(1)} - K_n^{(2)} \Vert \rightarrow 0$. We continue and define
\begin{align*}
        K_n^{(3)}=\int_0^1 \bigg(\mathbb{E} \bigg[ M_{a_t}(X) \frac{\partial}{\partial \theta} g(\theta) \Big\vert_{\theta=\theta_0} \bigg]\bigg)^{-1} \bigg(\frac{1}{n} \sum_{i=1}^n \frac{\partial}{\partial \theta} M_{\theta}(X_i)  g(L_n(a_t)) \Big\vert_{\theta=a_t}\bigg)\, dt.
\end{align*}
Direct calculations give
\begin{align*}
    \Vert K_n^{(2)} - K_n^{(3)} \Vert \leq & \int_0^1 \bigg\Vert \bigg(\mathbb{E} \bigg[ M_{a_t}(X) \frac{\partial}{\partial \theta} g(\theta) \Big\vert_{\theta=L_n(a_t)} \bigg]\bigg)^{-1}  -\bigg(\mathbb{E} \bigg[ M_{a_t}(X) \frac{\partial}{\partial \theta} g(\theta) \Big\vert_{\theta=\theta_0} \bigg]\bigg)^{-1} \bigg\Vert \\
    & \times \bigg\Vert\frac{1}{n} \sum_{i=1}^n \frac{\partial}{\partial \theta} M_{\theta}(X_i)  g(L_n(a_t)) \Big\vert_{\theta=a_t}\bigg\Vert \,dt.
\end{align*}
Then, with dominated convergence, the continuity of $\frac{\partial}{\partial \theta} g(\theta)$ and the matrix inversion, and Theorem \ref{theorem_consistency}, we conclude that $ \Vert K_n^{(2)} - K_n^{(3)} \Vert \rightarrow 0$ on $B_n$. Next, we take
\begin{align*}
    K_n^{(4)}=\int_0^1 \bigg(\mathbb{E} \bigg[ M_{a_t}(X) \frac{\partial}{\partial \theta} g(\theta) \Big\vert_{\theta=\theta_0} \bigg]\bigg)^{-1} \mathbb{E} \bigg[ \frac{\partial}{\partial \theta} M_{\theta}(X)  g(\theta_0) \Big\vert_{\theta=a_t}\bigg] \,dt.
\end{align*}
With similar techniques as before, it is not difficult to see that we have  $ \Vert K_n^{(3)} - K_n^{(4)} \Vert \rightarrow 0$ on $B_n$. Moreover, since we have $\mathcal{A}_{\theta_1}f_{\theta_2}(x)=M_{\theta_2}(x)g(\theta_1)$, $\theta_1,\theta_2 \in \Theta$, we know the integrand in $ K_n^{(4)}$ is just the derivative of the constant function which maps any $\theta \in \Theta_0$ to $\theta_0$ and is therefore equal to $0$ for any $t \in [0,1]$. Therefore, $ K_n^{(4)}=0$ and together with all preceding calculations we have shown $\Vert K_n^{(1)} \Vert \rightarrow 0$ on $B_n \cap C_n$. Hence we know that there exists a sequence of sets $A_n$ with $\mathbb{P}_{\theta_0}(A_n) \rightarrow 1$ such that $ K_n^{(1)} \leq 1$ on $A_n$, and the function $L_n$ is a contraction on $\Theta_0$ for each $\omega \in A_n$. Since the MLE is the unique fixed point of $L_n$ and will fall into $\Theta_0$ with probability converging to $1$ by Assumption \ref{assumptions_iteratice_proc}(a), the Banach fixed point theorem yields the claim.
\end{proof}

\bibliographystyle{abbrv}

\begin{thebibliography}{10}

\bibitem{adler1998practical}
R.~Adler, R.~Feldman, and M.~Taqqu.
\newblock {\em A Practical Guide to {H}eavy Tails: Statistical Techniques and
  Applications}.
\newblock Springer Science \& Business Media, 1998.

\bibitem{anastasiou2023stein}
A.~Anastasiou et~al.
\newblock Stein’s method meets computational statistics: A review of some
  recent developments.
\newblock {\em Statistical Science}, 38(1):120--139, 2023.

\bibitem{arnold2001multivariate}
B.~C. Arnold, E.~Castillo, and J.~M. Sarabia.
\newblock A multivariate version of {S}tein's identity with applications to
  moment calculations and estimation of conditionally specified distributions.
\newblock {\em Communications in Statistics - Theory and Methods},
  30(12):2517--2542, 2001.

\bibitem{artyushenko2019nakagami}
V.~Artyushenko and V.~Volovach.
\newblock {N}akagami distribution parameters comparatively estimated by the
  moment and maximum likelihood methods.
\newblock {\em Optoelectronics, Instrumentation and Data Processing},
  55:237--242, 2019.

\bibitem{baugci11nakagami}
K.~Ba{\u{g}}c{\i}.
\newblock {N}akagami distribution for modeling monthly precipitations in {V}an,
  {T}{\"u}rkiye.
\newblock {\em International Journal of Environment and Geoinformatics},
  11(3):19--23.

\bibitem{bai1987maximum}
Z.~Bai and J.~Fu.
\newblock On the maximum-likelihood estimator for the location parameter of a
  {C}auchy distribution.
\newblock {\em Canadian Journal of Statistics}, 15(2):137--146, 1987.

\bibitem{barp2019minimum}
A.~Barp, F.-X. Briol, A.~Duncan, M.~Girolami, and L.~Mackey.
\newblock Minimum {S}tein discrepancy estimators.
\newblock {\em Advances in Neural Information Processing Systems}, 32, 2019.

\bibitem{betsch2018characterizations}
S.~Betsch and B.~Ebner.
\newblock Fixed point characterizations of continuous univariate probability
  distributions and their applications.
\newblock {\em Annals of the Institute of Statistical Mathematics}, 73:31--59,
  2021.

\bibitem{betsch2021minimum}
S.~Betsch, B.~Ebner, and B.~Klar.
\newblock Minimum ${L}^q$-distance estimators for non-normalized parametric
  models.
\newblock {\em Canadian Journal of Statistics}, 49(2):514--548, 2021.

\bibitem{bloch1966note}
D.~Bloch.
\newblock A note on the estimation of the location parameter of the {C}auchy
  distribution.
\newblock {\em Journal of the American Statistical Association},
  61(315):852--855, 1966.

\bibitem{carrasco2007efficient}
M.~Carrasco, M.~Chernov, J.-P. Florens, and E.~Ghysels.
\newblock Efficient estimation of general dynamic models with a continuum of
  moment conditions.
\newblock {\em Journal of Econometrics}, 140(2):529--573, 2007.

\bibitem{carrasco2000generalization}
M.~Carrasco and J.-P. Florens.
\newblock Generalization of {GMM} to a continuum of moment conditions.
\newblock {\em Econometric Theory}, 16(6):797--834, 2000.

\bibitem{carrasco2002efficient}
M.~Carrasco and J.-P. Florens.
\newblock Efficient {GMM} estimation {U}sing the empirical characteristic
  {F}unction.
\newblock {\em IDEI working paper}, 2002.

\bibitem{carrasco2002simulation}
M.~Carrasco and J.-P. Florens.
\newblock Simulation-based method of moments and efficiency.
\newblock {\em Journal of Business \& Economic Statistics}, 20(4):482--492,
  2002.

\bibitem{carrasco2014asymptotic}
M.~Carrasco and J.-P. Florens.
\newblock On the asymptotic efficiency of {GMM}.
\newblock {\em Econometric Theory}, 30(2):372--406, 2014.

\bibitem{chernoff1967asymptotic}
H.~Chernoff, J.~L. Gastwirth, and M.~V. Johns.
\newblock Asymptotic distribution of linear combinations of {F}unctions of
  {O}rder statistics with applications to estimation.
\newblock {\em Annals of Mathematical Statistics}, 38(1):52--72, 1967.

\bibitem{copas1975unimodality}
J.~Copas.
\newblock On the unimodality of the likelihood for the {C}auchy distribution.
\newblock {\em Biometrika}, 62(3):701--704, 1975.

\bibitem{degroot2005optimal}
M.~H. DeGroot.
\newblock {\em Optimal Statistical Decisions}.
\newblock John Wiley \& Sons, 2005.

\bibitem{diaconis1991closed}
P.~Diaconis and S.~Zabell.
\newblock Closed {F}orm summation for classical distributions: {V}ariations on
  a theme of de {M}oivre.
\newblock {\em Statistical Science}, 6(3):284--302, 1991.

\bibitem{dobler2015stein}
C.~D{\"o}bler.
\newblock Stein's method of exchangeable pairs for the beta distribution and
  generalizations.
\newblock {\em Electronic Journal of Probability}, 20(109):1--34, 2015.

\bibitem{dominitz2005some}
J.~Dominitz and R.~P. Sherman.
\newblock Some convergence theory for iterative estimation procedures with an
  application to semiparametric estimation.
\newblock {\em Econometric Theory}, 21(4):838--863, 2005.

\bibitem{ebner2024independent}
B.~Ebner and Y.~Swan.
\newblock Independent additive weighted bias distributions and associated
  goodness-of-fit tests.
\newblock In {\em Recent Advances in Econometrics and Statistics: Festschrift
  in Honour of Marc Hallin}, pages 511--532. Springer, 2024.

\bibitem{ernst2020first}
M.~Ernst, G.~Reinert, and Y.~Swan.
\newblock First-order covariance inequalities via {S}tein’s method.
\newblock {\em Bernoulli}, 26(3):2051--2081, 2020.

\bibitem{ferguson1978maximum}
T.~S. Ferguson.
\newblock Maximum likelihood estimates of the parameters of the {C}auchy
  distribution for samples of size 3 and 4.
\newblock {\em Journal of the American Statistical Association},
  73(361):211--213, 1978.

\bibitem{ferguson2017course}
T.~S. Ferguson.
\newblock {\em A Course in Large Sample Theory}.
\newblock Routledge, 2017.

\bibitem{fgs23truncated}
A.~Fischer, R.~E. Gaunt, and Y.~Swan.
\newblock Stein's method of moments for truncated multivariate distributions.
\newblock {\em arXiv:2312.09344}, 2023.

\bibitem{fgs23}
A.~Fischer, R.~E. Gaunt, and Y.~Swan.
\newblock Stein's method of moments on the sphere.
\newblock {\em arXiv:2407.02299}, 2024.

\bibitem{freue2007pitman}
G.~V.~C. Freue.
\newblock The {P}itman estimator of the {C}auchy location parameter.
\newblock {\em Journal of Statistical Planning and Inference},
  137(6):1900--1913, 2007.

\bibitem{gabrielsen1982unimodality}
G.~Gabrielsen.
\newblock On the unimodality of the likelihood for the {C}auchy distribution:
  Some comments.
\newblock {\em Biometrika}, 69(3):677--678, 1982.

\bibitem{gaunt14}
R.~E. Gaunt.
\newblock Variance-gamma approximation via {S}tein's method.
\newblock {\em Electronic Journal of Probability}, 19:1--33, 2014.

\bibitem{gauntlaplace2021}
R.~E. Gaunt.
\newblock New error bounds for laplace approximation via {S}tein's method.
\newblock {\em ESAIM: Probability and Statistics}, 25:325--345, 2021.

\bibitem{gms19}
R.~E. Gaunt, G.~Mijoule, and Y.~Swan.
\newblock An algebra of {S}tein operators.
\newblock {\em Journal of Mathematical Analysis and Applications},
  496:260--279, 2019.

\bibitem{gilbert2019numDeriv}
P.~Gilbert and R.~Varadhan.
\newblock {\em Numderiv: {a}ccurate {n}umerical {d}erivatives}, 2019.
\newblock R package version 2016.8-1.1.

\bibitem{giles2013bias}
D.~E. Giles, H.~Feng, and R.~T. Godwin.
\newblock On the bias of the maximum likelihood estimator for the two-parameter
  {L}omax distribution.
\newblock {\em Communications in Statistics-Theory and Methods},
  42(11):1934--1950, 2013.

\bibitem{goldstein2013stein}
L.~Goldstein and G.~Reinert.
\newblock Stein's method for the beta distribution and the
  {P}{\'o}lya-{E}ggenberger urn.
\newblock {\em Journal of Applied Probability}, 50(4):1187--1205, 2013.

\bibitem{gordin1969central}
M.~I. Gordin.
\newblock The central limit theorem for stationary processes.
\newblock In {\em Dokl. Akad. Nauk SSSR}, volume 188, page~6, 1969.

\bibitem{gutmann2010noise}
M.~Gutmann and A.~Hyv{\"a}rinen.
\newblock Noise-contrastive estimation: A new estimation principle for
  {U}nnormalized statistical models.
\newblock In {\em Proceedings of the Thirteenth International Conference on
  Artificial Intelligence and Statistics}, pages 297--304. JMLR Workshop and
  Conference Proceedings, 2010.

\bibitem{gutmann2012noise}
M.~U. Gutmann and A.~Hyv{\"a}rinen.
\newblock Noise-contrastive estimation of {U}nnormalized statistical models,
  with a pplications to natural image statistics.
\newblock {\em Journal of Machine Learning Research}, 13(2):307--361, 2012.

\bibitem{hannan1973central}
E.~J. Hannan.
\newblock Central limit theorems for time series regression.
\newblock {\em Zeitschrift f{\"u}r Wahrscheinlichkeitstheorie und verwandte
  Gebiete}, 26(2):157--170, 1973.

\bibitem{hansen1982large}
L.~P. Hansen.
\newblock Large sample properties of generalized method of moments estimators.
\newblock {\em Econometrica}, 50(4):1029--1054, 1982.

\bibitem{hansen2012proofs}
L.~P. Hansen.
\newblock Proofs for large sample properties of generalized method of moments
  estimators.
\newblock {\em Journal of Econometrics}, 170(2):325--330, 2012.

\bibitem{hasselman2023nleqslv}
B.~Hasselman.
\newblock {\em Nleqslv: {s}olve {s}ystems Of {n}onlinear {e}quations}, 2023.
\newblock R package version 3.3.4.

\bibitem{hayakawa2016estimation}
J.~Hayakawa and A.~Takemura.
\newblock Estimation of exponential-polynomial distribution by holonomic
  gradient descent.
\newblock {\em Communications in Statistics-Theory and Methods},
  45(23):6860--6882, 2016.

\bibitem{hegde1989estimation}
L.~M. Hegde and R.~C. Dahiya.
\newblock Estimation of the parameters in a truncated normal distribution.
\newblock {\em Communications in Statistics-Theory and Methods},
  18(11):4177--4195, 1989.

\bibitem{hyvarinen2007some}
A.~Hyv{\"a}rinen.
\newblock Some extensions of score matching.
\newblock {\em Computational Statistics \& Data Analysis}, 51(5):2499--2512,
  2007.

\bibitem{hyvarinen2005estimation}
A.~Hyv{\"a}rinen and P.~Dayan.
\newblock Estimation of non-normalized statistical models by score matching.
\newblock {\em Journal of Machine Learning Research}, 6(4):695--709, 2005.

\bibitem{klenke2013probability}
A.~Klenke.
\newblock {\em Probability Theory: A Comprehensive Course}.
\newblock Springer Science \& Business Media, 2013.

\bibitem{kolar2004estimator}
R.~Kolar, R.~Jirik, and J.~Jan.
\newblock Estimator comparison of the {N}akagami-$m$ parameter and its
  application in echocardiography.
\newblock {\em Radioengineering}, 13(1):8--12, 2004.

\bibitem{kl2014}
A.~E. Koudou and C.~Ley.
\newblock Characterizations of {GIG} laws: A survey.
\newblock {\em Probability Surveys}, 11:161--176, 2014.

\bibitem{koutrouvelis1982estimation}
I.~A. Koutrouvelis.
\newblock Estimation of location and scale in {C}auchy distributions using the
  empirical characteristic function.
\newblock {\em Biometrika}, 69(1):205--213, 1982.

\bibitem{kumar2024q}
N.~Kumar, A.~Dixit, and V.~Vijay.
\newblock q-generalization of {N}akagami distribution with applications.
\newblock {\em Japanese Journal of Statistics and Data Science}, 2024.

\bibitem{labban20192}
J.~A. Labban.
\newblock On 2-parameter estimation of {L}omax distribution.
\newblock {\em Journal of Physics: Conference Series}, 1294(032018), 2019.

\bibitem{ley2017distances}
C.~Ley, G.~Reinert, and Y.~Swan.
\newblock Distances between nested densities and a measure of the impact of the
  prior in {B}ayesian statistics.
\newblock {\em Annals of Applied Probability}, 27:216--241, 2017.

\bibitem{ley2013stein}
C.~Ley and Y.~Swan.
\newblock Stein's density approach and information inequalities.
\newblock {\em Electronic Communications in Probability}, 18(7):1--14, 2013.

\bibitem{meintanis2016review}
S.~G. Meintanis.
\newblock A review of testing procedures based on the empirical characteristic
  function.
\newblock {\em South African Statistical Journal}, 50(1):1--14, 2016.

\bibitem{mijoule2023stein}
G.~Mijoule, M.~Raic, G.~Reinert, and Y.~Swan.
\newblock Stein’s density method for multivariate continuous distributions.
\newblock {\em Electronic Journal of Probability}, 28:1--40, 2023.

\bibitem{miyoshi2015downlink}
N.~Miyoshi and T.~Shirai.
\newblock Downlink coverage probability in a cellular network with ginibre
  deployed base stations and {N}akagami-m fading channels.
\newblock {\em arXiv preprint arXiv:1503.05377}, 2015.

\bibitem{nakagami1960m}
M.~Nakagami.
\newblock The m-distribution—a general formula of intensity distribution of
  rapid fading.
\newblock In {\em Statistical methods in radio wave propagation}, pages 3--36.
  Elsevier, 1960.

\bibitem{nakayama2011holonomic}
H.~Nakayama, K.~Nishiyama, M.~Noro, K.~Ohara, T.~Sei, N.~Takayama, and
  A.~Takemura.
\newblock Holonomic gradient descent and its application to the
  {F}isher--{B}ingham integral.
\newblock {\em Advances in Applied Mathematics}, 47(3):639--658, 2011.

\bibitem{neweylarge}
W.~Newey and D.~McFadden.
\newblock Large sample estimation and hypothesis testing.
\newblock {\em Handbook of Econometrics}, 4:2113--2245, 1994.

\bibitem{nw23}
S.~Nik and C.~H. Wei{\ss}.
\newblock Generalized moment estimators based on {S}tein identities.
\newblock {\em arXiv:2312.14601}, 2023.

\bibitem{oates2022minimum}
C.~Oates.
\newblock Minimum {K}ernel discrepancy estimators.
\newblock In {\em A. Hinrichs, P. Kritzer, F. Pillichshammer (eds.). Monte
  Carlo and Quasi-Monte Carlo Methods 2022}. Springer Verlag, 2024.

\bibitem{olver2010nist}
F.~W. Olver, D.~W. Lozier, R.~F. Boisvert, and C.~W. Clark.
\newblock {\em NIST {H}andbook of Mathematical {F}unctions}.
\newblock Cambridge university press, 2010.

\bibitem{papadatos2022point}
N.~Papadatos.
\newblock On point estimators for gamma and beta distributions.
\newblock {\em In press, The American Statistician}, 2024.

\bibitem{reeds1985asymptotic}
J.~A. Reeds.
\newblock Asymptotic number of roots of {C}auchy location likelihood equations.
\newblock {\em Annals of Statistics}, 13:775--784, 1985.

\bibitem{reyes2020nakagami}
J.~Reyes, M.~A. Rojas, O.~Venegas, and H.~W. G{\'o}mez.
\newblock {N}akagami distribution with heavy tails and applications to mining
  engineering data.
\newblock {\em Journal of Statistical Theory and Practice}, 14:1--20, 2020.

\bibitem{rothenberg1964note}
T.~J. Rothenberg, F.~M. Fisher, and C.~B. Tilanus.
\newblock A note on estimation from a {C}auchy sample.
\newblock {\em Journal of the American Statistical Association},
  59(306):460--463, 1964.

\bibitem{saumard2019weighted}
A.~Saumard.
\newblock Weighted {P}oincar{\'e} inequalities, concentration inequalities and
  tail bounds related to {S}tein kernels in dimension one.
\newblock {\em Bernoulli}, 25(4b):3978--4006, 2019.

\bibitem{schoutens2001orthogonal}
W.~Schoutens.
\newblock Orthogonal polynomials in {S}tein's method.
\newblock {\em Journal of Mathematical Analysis and Applications},
  253(2):515--531, 2001.

\bibitem{schwartz2013improved}
J.~Schwartz, R.~T. Godwin, and D.~E. Giles.
\newblock Improved maximum-likelihood estimation of the shape parameter in the
  {N}akagami distribution.
\newblock {\em Journal of Statistical Computation and Simulation},
  83(3):434--445, 2013.

\bibitem{stein1972bound}
C.~Stein.
\newblock A bound for the error in the normal approximation to the distribution
  of a sum of dependent random variables.
\newblock In {\em Proceedings of the sixth Berkeley symposium on mathematical
  statistics and probability, volume 2: Probability theory}, volume~6, pages
  583--603. University of California Press, 1972.

\bibitem{stein86}
C.~Stein.
\newblock {\em Approximate Computation of Expectations}.
\newblock IMS, Hayward, California, 1986.

\bibitem{tegos2021distribution}
S.~A. Tegos, D.~Tyrovolas, P.~D. Diamantoulakis, C.~K. Liaskos, and G.~K.
  Karagiannidis.
\newblock On the distribution of the sum of double-{N}akagami-m random vectors
  and application in randomly reconfigurable surfaces.
\newblock {\em arXiv preprint arXiv:2102.05591}, 2021.

\bibitem{wang2023new}
S.~Wang and C.~H. Wei{\ss}.
\newblock New characterizations of the (discrete) {L}indley distribution and
  their applications.
\newblock {\em Mathematics and Computers in Simulation}, 212:310--322, 2023.

\bibitem{weber2006minimum}
M.~D. Weber, L.~M. Leemis, and R.~K. Kincaid.
\newblock Minimum {K}olmogorov--{S}mirnov test statistic parameter estimates.
\newblock {\em Journal of Statistical Computation and Simulation},
  76(3):195--206, 2006.

\bibitem{wiens2003class}
D.~P. Wiens, J.~Cheng, and N.~C. Beaulieu.
\newblock A class of method of moments estimators for the two-parameter gamma
  family.
\newblock {\em Pakistan Journal of Statistics}, 19(1):129--141, 2003.

\bibitem{wilks1990maximum}
D.~S. Wilks.
\newblock Maximum likelihood estimation for the gamma distribution using data
  containing zeros.
\newblock {\em Journal of Climate}, pages 1495--1501, 1990.

\bibitem{xu19}
L.~Xu.
\newblock Approximation of stable law in {W}asserstein-1 distance by {S}tein's
  method.
\newblock {\em Annals of Applied Probability}, 29:458--504, 2019.

\bibitem{ye2017closed}
Z.-S. Ye and N.~Chen.
\newblock Closed-form estimators for the gamma distribution derived from
  likelihood equations.
\newblock {\em The American Statistician}, 71(2):177--181, 2017.

\bibitem{zhang2010highly}
J.~Zhang.
\newblock A highly efficient {L}-estimator for the location parameter of the
  {C}auchy distribution.
\newblock {\em Computational Statistics}, 25(1):97--105, 2010.

\bibitem{zhao2021closed}
J.~Zhao, S.~Kim, and H.-M. Kim.
\newblock Closed-form estimators and bias-corrected estimators for the
  {N}akagami distribution.
\newblock {\em Mathematics and Computers in Simulation}, 185:308--324, 2021.

\end{thebibliography}

\end{document}